\documentclass[11pt]{article}

\usepackage{rheelab}

\newif\ifexample

\begin{document}
\ifshownavigationpage
\section*{Navigation Links}
\noindent
\hyperlink{location of theorem tree}{Theorem Tree}
\bigskip

\noindent
\hyperlink{location of notation index}{Notation Index}
\begin{itemize}
\item[] 
    \hyperlink{location, notation index A}{A},
    \hyperlink{location, notation index B}{B},
    \hyperlink{location, notation index C}{C},
    \hyperlink{location, notation index D}{D},
    \hyperlink{location, notation index E}{E},
    \hyperlink{location, notation index F}{F},
    \hyperlink{location, notation index G}{G},
    \hyperlink{location, notation index H}{H},
    \hyperlink{location, notation index I}{I},
    \hyperlink{location, notation index J}{J},
    \hyperlink{location, notation index K}{K},
    \hyperlink{location, notation index L}{L},
    \hyperlink{location, notation index M}{M},
    \hyperlink{location, notation index N}{N},
    \hyperlink{location, notation index O}{O},
    \hyperlink{location, notation index P}{P},
    \hyperlink{location, notation index Q}{Q},
    \hyperlink{location, notation index R}{R},
    \hyperlink{location, notation index S}{S},
    \hyperlink{location, notation index T}{T},
    \hyperlink{location, notation index U}{U},
    \hyperlink{location, notation index V}{V},
    \hyperlink{location, notation index W}{W},
    \hyperlink{location, notation index X}{X},
    \hyperlink{location, notation index Y}{Y},
    \hyperlink{location, notation index Z}{Z},
    \hyperlink{location, notation index others}{Others}
\end{itemize}

\bigskip

\noindent
\hyperlink{location of equation number list}{Numbered Equations}
\bigskip


\tableofcontents
\fi


\title{Sample-Path Large Deviations for L\'evy Processes and Random Walks with Lognormal Increments}
\author{Zhe Su} 
\author{Chang-Han Rhee}
\affil{Industrial Engineering and Management Sciences, Northwestern University\\
    Evanston, IL, 60613, USA}

\maketitle

\begin{abstract}
\noindent 
The large deviations theory for heavy-tailed processes has seen significant advances in the recent past. 
In particular, \cite{MR4038038} and \cite{MR4187125} established large deviation asymptotics at the sample-path level for L\'evy processes and random walks with regularly varying and (heavy-tailed) Weibull-type increments. 
This leaves the lognormal case---one of the three most prominent classes of heavy-tailed distributions, alongside regular variation and Weibull---open. 
This article establishes the \emph{extended large deviation principle} (extended LDP) at the sample-path level for one-dimensional L\'evy processes and random walks with lognormal-type increments. 
Building on these results, we also establish the extended LDPs for multi-dimensional processes with independent coordinates. 
We demonstrate the sharpness of these results by constructing counterexamples, thereby proving that our results cannot be strengthened to a standard LDP under $J_1$ topology and $M_1'$ topology.

\end{abstract}

\tableofcontents

\section{Introduction}

\theme{What problem has been solved in this paper?}%
\topic{The description of the problem.}%
This paper develops the sample-path large deviations for L\'evy processes and random walks whose increment distributions have lognormal-type tails.
\noindent
Let $\{X(t), t\geq 0\}$ be a centered L\'evy process.
We assume its L\'evy measure $\nu$ is light-tailed on the negative half-line, and heavy-tailed on the positive half-line. 
In particular, we consider the lognormal case, i.e., $\nu[x,\infty) = \exp(-r(\log x))$ where $r(x)$ is a regularly varying function with index $\gamma > 1$ as $x\to \infty$.
Consider a scaled process $\bar{X}_n(t) = X(nt)/ n$ for $t\in [0,1]$. 
Similarly, consider a scaled and centered random walk $\bar W_n (t)  = \frac1n \sum_{i=1}^{\lfloor nt \rfloor} (Z_i - \E {Z_1})$ where $\P(Z_1 > x) = \exp(-r(\log x))$.
This article investigates the sample-path large deviations of $\bar{X}_n$ and $\bar W_n$, their multidimensional extensions, and the limitations of the classical large deviation principle framework.

\theme{Why we want to solve this problem?}%
When the increment distributions are light-tailed, the large deviations of $\bar X_n$ and $\bar W_n$ have been thoroughly studied in probability theory. 
The classical theory of large deviations provide powerful tools for analyzing a wide variety of rare events. 
In particular, the sample-path level LDPs allow one to systematically characterize how a stochastic system deviates from their nominal behaviors for a wide variety of rare events. 
In the heavy-tailed context, the seminal papers \cite{MR0282396,MR0438438} initiated the analysis of the tail aysmptotics of $\bar W_n(1)$, followed by vigorous research activities in the extreme value theory literature; see, for example, \cite{MR2424161,MR2440928,MR1458613}.
In particular, \cite{MR2440928} assumes a very general class of heavy-tailed distributions in $X$ and describes in detail how fast $x$ needs to grow with $n$ for the asymptotic relation
\begin{equation}\label{equation:one-big-jump}
    \pr{X(n)> x } = n\pr{X(1) > x}(1+o(1))
\end{equation}
to hold, as $n\rightarrow\infty$. 
If (\ref{equation:one-big-jump}) is valid, it is said that the \emph{principle of one big jump} holds.
\cite{MR2187307} generalized this insight to a sample-path level of $X$. 
On the other hand, other related works, including \cite{blanchet2012measuring,MR2931277, MR2052908},  investigated the asymptotics of $\pr{f(X)\in A}$ for many functionals $f$ and many sets $A$ using ad-hoc approaches, revealing that rare events can also be governed by multiple jumps, rather than just one big jump.

\topic{Describe the work with SPLD that happen more recently}%
More recently, \cite{MR4038038} and \cite{MR4187125} took a systematic approach inspired by \cite{MR2187307}, but for any number of big jumps, rather than a single big jump:
they established asymptotic estimates of $\pr{\bar{X}_n\in A}$, with $A$ being sufficiently general set of c\'adl\'ag functions, thus facilitating a systematic way of studying rare events defined in terms of continuous functions of $\bar{X}_n$. 
They also clarified how the large deviations in the heavy tailed settings are connected to the standard large-deviations approach. 
More specifically, \cite{MR4187125} examined large deviations for L\'evy measures with Weibull tails, offering the logarithmic asymptotics 
    \begin{equation}\label{inequalities:SPLD_Weibull}
        -\inf_{\xi\in A^\circ} I(x) \leq \liminf_{n\rightarrow\infty} \frac{\log\pr{\bar{X}_n\in A}}{\log n}\leq\limsup_{n\rightarrow\infty} \frac{\log\pr{\bar{X}_n\in A}}{\log n} \leq -\lim_{\epsilon\downarrow 0}\inf_{\xi\in A^\epsilon} I(x)
    \end{equation}
    with the rate function
    \begin{equation}
        I(\xi) = 
        \begin{cases}
            \sum_{t:\xi(t)\neq \xi(t-)} (\xi(t) - \xi(t-))^\alpha& \text{if }\xi\text{ is a non-decreasing pure jump path}\\
            \infty & \text{otherwise}
        \end{cases}
    \end{equation}
For regularly varying L\'evy measures, \cite{MR4038038} established an exact asymptotics
\begin{equation}\label{inequalities:SPLD_regularly_varying}
        C_{\ms{J}(A)}(A^\circ)\leq \liminf_{n\rightarrow\infty}\frac{\pr{\bar{X}_n\in A}}{n^{-\ms{J}(A)(\alpha-1)}}\leq \limsup_{n\rightarrow\infty}\frac{\pr{\bar{X}_n\in A}}{n^{-\ms{J}(A)(\alpha-1)}}\leq C_{\ms{J}(A)}(\bar{A}),
\end{equation}
where $\alpha$ is the index of L\'evy measure's regularly varying tail, $\ms{J}(A)$ is the smallest number of jumps for a step function to be contained in $A$, and $C_{\ms{J}(A)}(\cdot)$ is a measure on the space of c\'adl\'ag functions with $\ms{J}(A)$ or less jumps.
Both asymptotic bounds imply that the rare events are driven by big jumps, hence characterizes the \emph{catastrophe principle}.

\topic{The results in this work fill the gap of lognormal case}%
Among the arguably most important classes of tail distributions for modeling heavy-tailed phenomena---regularly-varying, heavy-tailed Weibull, and log-normal distributions---the characterization of the catastrophe principle in the log-normal case remains open. 
This paper addresses this gap and establishes the sample-path large deviations for L\'evy processes and random walks with lognormal increments.

\theme{Describe the structure and detailed results by sections}%
\topic{Describe the extended LDP of L\'evy process under the $J_1$ topology}%
Specifically, in Section \ref{section:extended-LDP-levy-J_1}, we establish 
the extended large deviation principledev (extended LDP) for $\bar{X}_n$ under the Skorokhod $J_1$ topology:
\begin{equation}\label{inequalities:SPLD_lognomal}
        -\inf_{\xi\in A^\circ} I(x) \leq \liminf_{n\rightarrow\infty} \frac{\log\pr{\bar{X}_n\in A}}{r(\log n)}\leq\limsup_{n\rightarrow\infty} \frac{\log\pr{\bar{X}_n\in A}}{r(\log n)} \leq -\lim_{\epsilon\downarrow 0}\inf_{\xi\in A^\epsilon} I(x)
\end{equation}
where
\begin{equation}\label{rate-function-lornormal-introduction}
    I(\xi) = 
        \begin{cases}
            \xi \text{'s number of jumps } & \text{if }\xi\in\mb{D}_{<\infty}\\
            \infty & \text{otherwise}
            .
        \end{cases}
\end{equation}
Here, $\gamma>1$ is a modeling parameter for the lognormal-type tail, where the standard lognormal distribution corresponds to the case $\gamma = 2$. 
$\mb{D}_{<\infty}$ is the space of non-decreasing step functions with finite number of jumps vanishing at the origin and continuous at 1.
We accomplish this by first establishing the extended LDP for $\hat{J}_n^{\leqslant k}$(with respect to $n$), where $\hat{J}^{\leqslant k}$ represents, roughly speaking, the process constructed by taking the $k$ largest jumps of the L\'evy process. 
Then, we argue that the asymptotic behavior of  $\hat{J}_n^{\leqslant k}$ governs that of $\bar{X}_n$ for sufficiently large $k$'s.
This allows us to obtain the extended LDP for $\bar{X}_n$ from that of $\hat{J}_n^{\leqslant k}$'s by leveraging the approximation lemma in \cite{MR4187125}.

\topic{Briefly mention the standard LDP}%
Extended LDP was first introduced in \cite{MR2797595}. 
Despite its strong-sounding name, extended LDP is actually a weaker statement than the standard LDP, as the upper bound in (\ref{inequalities:SPLD_lognomal}) involves the $\epsilon$-fattening $A^\epsilon$ of the set $A$, hence, increasing the value of the upper bound.
\topic{Describe the the standard LDP is not satisfied even under $M'_1$ topology}%
We show that the standard LDP cannot be satisfied for the lognormal tails even under the coarser Skorokhod $M'_1$ topology. 
Section \ref{section:nonexist-standard-ldp} constructs a closed set in the $M'_1$ topology for which the standard LDP upper bound is violated.
\topic{Compare to the Weibull case}%
This contrasts to the conclusion in \cite{MR4187125}, where the extended LDP of $\bar{X}_n$ under the $J_1$ topology is strenghened to the standard LDP under the $M'_1$ topology.

\topic{Describe the extended LDP for the random walks}%
We also derive an extended LDP for random walks in section \ref{section:randomwalk-ldp-J1}, assuming a lognormal-type tail in its increment distribuion. 
In contrast to the L\'evy process setting, the rate function here takes finite values for step functions that are discontinuous at time 1. 
This difference is due to the fact that rescaled random walk has a jump at the right boundary with a probability bounded from below. 
See Theorem \ref{theorem: extended-ldp-randomwalk} for the precise statement.

\topic{Describe the extended LDP for multi-dimensional processes}%
Many applications require modeling multiple sources of uncertainties. 
In such cases, large deviations for multidimensional processes provide the means to model such systems. 
For example, \cite{MR4032925, MR4493394} analyze the queue length asymptotics for the multiple server queues and stochastic fluid networks with heavy-tailed Weibull service times based on the large deviations results in \cite{MR4187125} for multi-dimensional L\'evy processes and random walks.  
Section \ref{section:extended-ldp-j1-multidimension} obtains extended LDP for multi-dimensional processes with independent components $\bm{\bar{X}}_n = (\bar{X}^{(1)}_n, \cdots, \bar{X}^{(d)}_n)$.
Here $\bar{X}^{(i)}_n$'s are centered and scaled 1-dimensional L\'evy processes or random walks independent of each other.
As the rate function in (\ref{rate-function-lornormal-introduction}) is not good---i.e., does not have compact level set---the standard results such as Theorem 4.14 of \cite{MR2045489} do not apply directly in our context, and the derivation of the extended LDP for $\bm{\bar{X}}_n$ from those of $\bar{X}^{(1)}_n$ requires careful justification.  
We take advantage of the discrete nature of $\bar X_n$ and $\bar W_n$'s rate functions to establish the extended LDP for a $d$-fold product of probability measures, where each coordinate satisfies an extended LDP with a rate function that only takes at most countable values. 

Section~\ref{section:notation-definitions} provides preliminaries and Section~\ref{section:proofs} presents most of the proofs for the results in Sections~\ref{section:notation-definitions} and \ref{section:extended-ldp-levy}.

\section{Preliminaries}\label{section:notation-definitions}
This section provides preliminary results useful for later sections. 
All proofs in this subsection are deferred to Section~\ref{section:preliminary-proof}.
We begin by introducing recurring notations. 
Let $(\ms{X}, d)$ denote a \defnota{metric-space}{metric space $\ms{X}$ equipped with a metric $d$.}
For a set $A$ let $A^c$ denote the complement of $A$. 
Let 
    $
    \newnota
        {d-metric}
        {d(x, A)} 
    \triangleq
    \defnota
        {d-metric}
        {\inf_{y\in A}d(x, y)}
    $,
    \newnota
        {b-open-ball-radius-r}
        {$B_r(x)$}
    \defnota
        {b-open-ball-radius-r}
        {$\delequal \{y\in\ms{X}: d(x,y)<r\}$},
    \newnota
        {a-epsilon-blowup set}
        {$A^\epsilon$}
    \defnota
        {a-epsilon-blowup set}
        {$\delequal\{x\in\ms{X}: d(x,A)\leq \epsilon\}$},
and
    \newnota
        {a-epsilon-shrink-set}
        {$A^{-\epsilon}$}
    \defnota
        {a-epsilon-shrink-set}
        {$\delequal \Big(\big(A^c\big)^\epsilon\Big)^c$}
denote 
    the distance between $x$ and $A$,
    the open ball with radius $r$ centered at $x$,  
    the closed $\epsilon$-fattening of $A$, 
    and the open $\epsilon$-shrinking of $A$, 
    respectively.

We will denote the \defnota{skorokhod-space-domain-E}{the space of real-valued c\`adl\`ag functions from $[0,1]$ to $\R$%
}, i.e., Skorokhod space, 
with \newnota{d-skorokhod-space-domain-E}{$\mb{D}[0,1]$} or simply \newnota{d-skorokhod-space-domain-01}{$\mb{D}$} when the domain $[0,1]$ is clear. 
\mdefnota{d-skorokhod-space-domain-01}{$\D([0,1])$}%
Let 
$$
\newnota{d-metric-uniform}{d_{\|\cdot\|}(\xi, \zeta)}\defnota{d-metric-uniform}{ = \|\xi - \zeta\|} = \sup_{t\in[0,1]} |\xi(t) - \zeta(t)|
$$ 
denote the uniform metric
and recall the $J_1$ metric
\linkdest{notation:j1metric}
    \begin{equation*}\label{definition:j1metric-d1}
        \newnota{d-metric-j1}{d_{J_1}(\xi,\zeta)}\defnota{d-metric-j1}{= \inf_{\lambda\in\Lambda} \|\lambda - e\|\vee\|\xi\circ \lambda - \zeta\|},
    \end{equation*}
where \newnota{lambda}{$\Lambda$} is \defnota{lambda}{the collection of all non-decreasing homeomorphisms on $[0,1]$}, and
\newnota{e-identity_map}{$e$} is \defnota{e-identity_map}{the identity map} on $[0,1]$. 
The $M'_1$ metric on $\mb{D}[0,1]$ is defined as follows.
Let $\Gamma(\xi)$ denote the extended completed graph of $\xi\in\mb{D}$. 
That is,
\begin{equation}\label{defnition:extended-complete-graph}
    \newnota{gamma-extended-complete-graph}{\Gamma(\xi)}\defnota{gamma-extended-complete-graph}{ = \{(y,t)\in \mb{R}\times [0,1]: y\in[\xi(t-)\wedge\xi(t), \xi(t-)\vee \xi(t)]\}}
\end{equation} 
with the convention that $\xi(0-)\triangleq 0$.
Note that if $\xi(0)\neq 0$, then one can consider $\xi$ as a path with a jump at time 0. 
Roughly speaking, $\Gamma(\xi)$ is the union of the graph of $\xi$ and the vertical lines that concatenate the connected pieces. 
Define an order `\newnota{order-on-graph}{$\prec$}'\mdefnota{order-on-graph}{an order on the extended complete graph $\Gamma(\cdot)$} on $\Gamma(\xi)$ as follows: for two points on $\Gamma(\xi)$, we say that $(y_1, t_1) \prec (y_2, t_2)$ if either $t_1 < t_2$, or $t_1 = t_2$ and $|\xi(t_1-) - y_1| < |\xi(t_2-)-y_2|$.
A continuous, nondecreasing (w.r.t.\ $\prec$), and surjective function $(u,r):[0,1]\rightarrow \Gamma(\xi)$ is called a parametrization of $\Gamma(\xi)$.
Let \newnota{pi-parametrization-extended-complete-graph}{$\Pi(\xi)$} denote \defnota{pi-parametrization-extended-complete-graph}{the collection of all parametrizations of $\Gamma(\xi)$}. 
The $M'_1$ metric is 
\begin{equation}
    \newnota{d-metric-m1prime}{d_{M'_1}}(\xi, \zeta)\defnota{d-metric-m1prime}{= \inf_{\substack{(u_1,r_1)\in \Pi(\xi)\\
    (u_2,r_2)\in \Pi(\zeta)}} \|u_1 - u_2\|\vee \|r_1- r_2\|}.
\end{equation}
We will often consider multiple paths in $\D$ and work with their extended completed graphs as subsets of $\R \times [0,1]$. 
We use the $\ell_\infty$ distance in this space, i.e., 
$d\big((x,t), (y,s)\big) = |x-y| \vee |t-s|$ for $(x,t), (y,s) \in \R \times [0,1]$.
See \cite{MR4187125} for the full details of $M_1'$ metric on $\mb{D}[0,1]$ and \cite{MR1442318} or \cite{MR1876437} for $M_1'$ metric on $\mb{D}[0,\infty)$.
It is obvious from their definitions that $d_{\|\cdot\|}\geq d_{J_1}$. 
Note also that the $M_1'$ distance is upper bounded 
by the $M_1$ distance, and hence, Theorem 12.3.2 in \cite{MR1876437} implies that $d_{J_1}\geq d_{M'_1}$.
The next two simple observations are useful for bounding the $M_1'$ distances from below.  
\begin{lemma}
\label{lemma:m1prime-gap-criteria}
\linksinthm{lemma:m1prime-gap-criteria}
Suppose that $\xi,\zeta \in \mb{D}$. 
Then for any $(u,r) \in \Gamma(\xi)$,
$$
d_{M_1'}(\xi, \zeta) \geq d\big((u,r), \Gamma(\zeta)\big)
.
$$
\end{lemma}

\begin{lemma}\label{lemma:constant-jump-m1prime-distance}
\linksinthm{lemma:constant-jump-m1prime-distance}
    Suppose that $\xi,\zeta\in\mb{D}$. If there exists $s,t \in [0,1]$ and $\delta>0$ such that $|\zeta(t) - \zeta(s-)| \geq 2\delta$, and $\xi$ is constant on $[s-\delta, t+\delta]\cap [0,1]$, then
    $$
    d_{M'_1}(\xi,\zeta)\geq \delta.
    $$
\end{lemma}
Throughout this paper, we consider the following subsets of $\mb{D}$:
\begin{equation*}
\begin{aligned}
&
\newnota{d-uparrow}{\mb{D}^\uparrow}\defnota{d-uparrow}{\delequal \{\xi\in\mb{D}: \xi\text{ is a non-decreasing, pure jump function} \}};
\\
&
\newnota{d-j-jumps-origin-zero}{\mb{D}_{= j}}\defnota{d-j-jumps-origin-zero}{\delequal\{\xi\in\mb{D}^\uparrow: \xi(0) = 0,\ \xi(1) = \xi(1-),\ \xi\text{ has } j \text{ jumps}\}};
\\
&
\newnota{d-lessequalto-j-jumps-origin-zero}{\mb{D}_{\leqslant j}}\defnota{d-lessequalto-j-jumps-origin-zero}{\delequal \cup_{i=0}^j \mb{D}_{=i}},\quad \newnota{d-finite-jumps-origin-zero}{\mb{D}_{< \infty}}\defnota{d-finite-jumps-origin-zero}{\delequal \cup_{j=0}^\infty \mb{D}_{\leqslant j}};
\\
&
\newnota{d-tilde_equal-j}{\tilde{\mb{D}}_{= j}}\defnota{d-tilde_equal-j}{\delequal\{\xi\in\mb{D}^\uparrow: \xi(0) = 0,\ \xi\text{ has } j \text{ jumps}\}};
\\
&
\newnota{d-tilde_leqslant-j}{\tilde{\mb{D}}_{\leqslant j}}\defnota{d-tilde_leqslant-j}{\delequal \cup_{i=0}^j \tilde{\mb{D}}_{=i}} ,\quad \newnota{d-tilde_leqslant-infty}{\tilde{\mb{D}}_{< \infty}}\defnota{d-tilde_leqslant-infty}{\delequal \cup_{j=0}^\infty \tilde{\mb{D}}_{\leqslant j}};
\\
&
\newnota
    {d_hat-j-jumps-origin-positive-endjump}
    {\hat{\mb{D}}_{= j}}
\defnota
    {d_hat-j-jumps-origin-positive-endjump}
    {\delequal\{\xi\in\mb{D}^\uparrow: \xi(0)\geq 0,\ \xi\text{ has } j \text{ jumps}\}};
\\
&
\newnota
    {d-hat_lessequalto-j-jumps-origin-positive-endjump}
    {\hat{\mb{D}}_{\leqslant j}}
\defnota
    {d-hat_lessequalto-j-jumps-origin-positive-endjump}
    {\delequal \cup_{i=0}^j \hat{\mb{D}}_{=i}},
    \quad 
\newnota
    {d-finite-jumps-origin-positive-endjump}
    {\hat{\mb{D}}_{< \infty}}
\defnota
    {d-finite-jumps-origin-positive-endjump}
    {\delequal \cup_{j=0}^\infty \hat{\mb{D}}_{\leqslant j}}.
\end{aligned}
\end{equation*}
It's clear that $\mb{D}_{=j}\subset\tilde{\mb{D}}_{=j}\subset\hat{\mb{D}}_{=j}$ for any $j\in\mb{N}$, and hence, $\mb{D}_{<\infty}\subset\tilde{\mb{D}}_{<\infty}\subset\hat{\mb{D}}_{<\infty}$.
It can be verified that $\mb{D}_{\leqslant j}$, $\tilde{\mb{D}}_{\leqslant j}$ and $\hat{\mb{D}}_{\leqslant j}$ are closed subsets of $(\mb{D}, d_{J_1})$. 
In Lemma \ref{lemma: rate-func-m1prime-levelset-closed}, we show that $\hat{\mb{D}}_{\leqslant j}$ is closed in $(\mb{D}, d_{M'_1})$ as well.

We conclude this section with a brief review of the formal definition of \emph{extended large deviation principle} (extended LDP) introduced in \cite{MR2797595}. 
Extended LDP is often useful in heavy-tailed contexts where the standard LDP fails to hold.
\begin{definition}\label{def:extend_LDP}
    A sequence of measures
    $\{\mu_n\}_{n\geq 1}$ satisfies the extened LDP with speed $a_n$ on $(\ms{X}, d)$ if there is a 
    rate function $I:\ms{X}\rightarrow\mb{R}_+$ and a sequence $\{a_n\}_{n\geq 1}$ with $a_n\rightarrow \infty$ such that, for any open set $G$ and any closed set $F$, the following inequalities hold:
    \begin{equation}\label{inequality:extended_LDP}
        -\inf_{x\in G} I(x) \leq \liminf_{n\rightarrow\infty} \frac{\log\mu_n(G)}{a_n}\text{ and } \limsup_{n\rightarrow\infty} \frac{\log\mu_n(F)}{a_n}\leq -\lim_{\epsilon\downarrow 0}\inf_{x\in F^\epsilon} I(x).
    \end{equation}
\end{definition}
For a sequence of $\mathcal X$-valued random variables $\{X_n\}_{n\geq 1}$, we say that $X_n$ satisfies the extended LDP on $(\mathcal X, d)$ if $\{\pr{X_n\in \cdot}\}_{n\geq 1}$ satisfies the extended LDP on $(\mathcal X, d)$.
The rate function $I$ in Definition \ref{def:extend_LDP} is uniquely determined; see \cite{MR2797595}.

We present a couple of tools that are useful for establishing extended LDPs:

\begin{lemma}
\label
    {E-LDP-on-subspaces-full-measure}
\linksinthm
    {E-LDP-on-subspaces-full-measure}
Let $\{X_n\}_{n\geq 1}$ be a sequence of c\`adl\`ag stochastic processes. 
Let $\mb{E}\subset\mb{D}$ be a closed set such that $\pr{X_n\in\mb{E}} = 1$ for all $n \in\mb{N}$. If $\{X_n\}_{n\geq 1}$ satisfies the extended LDP on $(\mb{E}, d)$ with the rate function $I$ and speed $a_n$. Then $\{X_n\}_{n\geq 1}$ satisfies the extended LDP on $(\mb{D}, d)$ with the same speed sequence and the rate function given by
    \begin{equation*}
        I'(x) = 
        \begin{cases}
            I(x) & x\in \ms{E}\\
            \infty & x\notin \ms{E}
        \end{cases}
    \end{equation*}
\end{lemma}

\begin{proposition}[Proposition 2.1 of \cite{MR4493394}]
\label{lemma:approximate}
\linksinthm{lemma:approximate}
Let $\{X_n\}_{n\geq 1}$ and $\{Y^k_n\}_{n\geq 1}$ for each $k\in\mb{N}$ be sequence of random objects in a metric space $(\ms{X},d)$. Let $I$ and $I_k$ for $k\in\mb{N}$ be non-negative lower-semicontinuous functions.  Suppose that the following conditions hold:
\begin{enumerate}[(1)]
    \item 
    For each $k\in\mb{N}$, the sequence $\{Y_n^k\}_{n\geq 1}$ satisfy the extended LDP with the rate function $I_k$ and the speed $\{a_n\}_{n\geq 1}$. 
    \item 
    For each closed set $F$, 
    \begin{equation}\label{require1}
        \lim_{k\rightarrow\infty}\inf_{x\in F} I_k(x)\geq \inf_{x\in F} I(x).
    \end{equation}
    \item
    For each $\delta > 0$ and each open set $G$, there exists $\epsilon > 0$ and $K\geq 0$ such that $k\geq K$ implies
    \begin{equation}\label{require2}
        \inf_{x\in G^{-\epsilon}} I_k(x)\leq\inf_{x\in G} I(x)+\delta.
    \end{equation}

    \item
    For every $\epsilon > 0$, it holds that 
    \begin{equation}\label{require3}
        \lim_{k\rightarrow\infty}\limsup_{n\rightarrow\infty} \frac{1}{a_n}\log\pr{d(X_n, Y_n^k) > \epsilon} = -\infty.
    \end{equation}
\end{enumerate}
Then $\{X_n\}_{n\geq 1}$ satisfy the extended LDP with the rate function $I$ and speed $a_n$.
\end{proposition}

When the sequences $\{Y_n^k\}_{n\geq 1}$ and the rate functions $I_k$'s do not change with $k$, the above proposition directly implies the following corollary.

\begin{corollary}[Corollary 2.1 of \cite{MR4493394}]
\label{lemma:approximate2}
\linksinthm{lemma:approximate2}
Assume both the sequence $\{X_n\}_{n\geq 1}$ and $\{Y_n\}_{n\geq 1}$ take values in a  metric space $(\ms{X},d)$ and the following relation holds:
\begin{equation}\label{require4}
    \limsup_{n\rightarrow\infty} \frac{\log\pr{d(X_n, Y_n) > \epsilon}}{a_n} = -\infty.
\end{equation}
If $\{X_n\}_{n\geq 1}$ satisfies the extended LDP with the rate function $I$ and speed $\{a_n\}_{n\geq 1}$, then $\{Y_n\}_{n\geq 1}$ satisfies the extended LDP with the same rate function and speed.
\end{corollary}

\section{Main Results}\label{section:extended-ldp-levy}

This section presents the main results of the paper. 
Section~\ref{section:extended-LDP-levy-J_1} establishes the extended LDP at the sample-path level for one-dimensional L\'evy processes, while Section~\ref{section:randomwalk-ldp-J1} establishes it for one-dimensional random walks.
Section~\ref{section:extended-ldp-j1-multidimension} establishes the extended LDP for multi-dimensional processes with independent coordinates. 
Section~\ref{section:nonexist-standard-ldp} constructs a counterexample that shows our extended LDPs
cannot be strengthened to a standard LDP.
Section~\ref{subsec:boundary-crossing-example} illustrates the applicability of our results with a boundary crossing problem.

\subsection{Extended LDP for L\'evy processes}\label{section:extended-LDP-levy-J_1}

According to the L\'evy-Ito decomposition (see, for example, Chapter 2 of \cite{MR3185174}), the L\'evy process with triplets $(a,b,\nu)$ has the following distributional representation: for $t\in[0,\infty)$
\begin{equation}\label{equ:levy-ito-decomposition}
    \dnewnota{x-levy-process}{X(\cdot)} X(t) = bt + aB(t) + \int_{|x|<1}x\big(\hat{N}([0,t]\times dx) - t\nu(dx)\big)+\int_{|x|\geq 1}x\hat{N}([0,t]\times dx)
    ,
\end{equation}
\mdefnota{x-levy-process}{general notation used for L\'evy process}%
where $a$ and $b$ are scalars, $\newnota{b-standard-brownian-motion}{B}$ is \defnota{b-standard-brownian-motion}{the standard Brownian motion}, 
and
\newnota{n-levy-measure}{$\nu$} is a \defnota{n-levy-measure}{L\'evy measure, i.e., a $\sigma$-finite measure supported on $\mb{R}\backslash\{0\}$ that satisfies $\int \min(1,|x|^2)\nu(dx) < \infty$}. \newnota{n-poisson-random-measure}{$\hat{N}$} is a \defnota{n-poisson-random-measure}{Poisson random measure on $[0,\infty)\times (0,\infty)$ with mean measure $\newnota{leb}{\textbf{LEB}}\times\nu$}, where LEB is \defnota{leb}{the Lebesque measure on $\mb{R}$}. 
Throughout the rest of this paper, we make the following lognormal-type tail assumption on the L\'evy measure $\nu$:

\begin{assumption}\label{assumption-lognormaltail}
    $\nu$ and $\P(Z_1 \in \cdot)$ are supported on \newnota{r-non-negative-real-number}{$\mb{R}_+$}\defnota{r-non-negative-real-number}{$ = \{x\in\mb{R}: x\geq 0\}$}, and 
    \begin{equation}\label{lognormal-lambda-gamma}
        \nu[x,\infty) = \exp(-r(\log x)) 
    \end{equation}
$\text{for a regularly varying function } r(\cdot) \text{ with index $\gamma > 1$.}$
\end{assumption}
\begin{remark}
Note that such $\nu$ is always heavier than Weibull tails, and the case $\gamma=2$ corresponds with the standard lognormal tail, whereas $\gamma \leq 1$ corresponds with the power law tails or even heavier tails. 
\end{remark}
Consider a scaled process
\begin{equation}\label{def:xnbar}
    \newnota{xnbar-scaled-centered-1d-levy-process}{\bar{X}_n(t)} \defnota{xnbar-scaled-centered-1d-levy-process}{=\frac{1}{n}
    \bigg(X(nt) - bnt - nt\int_{|x|\geq 1} x\nu(dx)\bigg)}.
\end{equation}
The main result of this section is the following extended LDP for $\bar X_n$.

\begin{theorem}
\label{theorem:ex-ldp-xnbar}
\linksinthm{theorem:ex-ldp-xnbar}
The sequence $\{\bar X_n\}_{n\geq 1}$ of scaled processes  defined in (\ref{def:xnbar}) satisfies the extended LDP in $(\mb{D}, d_{J_1})$ with the rate function 
\begin{equation}\label{xnbarrate2}
\newnota{i-xnbar-rate-function}{I^{J_1}(\xi)}\defnota{i-xnbar-rate-function}{ = 
\begin{cases}
\sum_{t\in[0,1]}\I\{\xi(t)\neq \xi(t-)\} & \text{ if }\xi\in\mb{D}_{<\infty}\\
\infty & \text{ otherwise}
\end{cases}}
\end{equation}
and speed $r(\log n)$. 
\end{theorem}

As we will see in Section~\ref{section:nonexist-standard-ldp}, this result cannot be strengthened to the standard LDP w.r.t.\ the $J_1$ topology. 
However, one may naturally wonder if the standard LDP holds w.r.t.\ weaker topologies. 
For example, \cite{MR4187125} establishes the standard LDP w.r.t.\ the $M_1'$ topology, while showing that the standard LDP w.r.t.\ the $J_1$ topology is impossible in the heavy-tailed Weibull case.
It turns out that the counterexample in Section~\ref{section:nonexist-standard-ldp} proves that even w.r.t.\ the $M_1'$ topology, the standard LDP cannot be satisfied. 
Here we also mention that the rate function $I^{J_1}$ fails to be lower-semicontinuous under the $M'_1$ topology. To see this, consider $\xi = \I([0,1])$ and $\xi_n = \I([1/n,1])$ for $n\in\mb{N}$. 
It's straightforward to verify that $\xi_n\rightarrow\xi$ under the $M'_1$ topology, but $\lim_{n\rightarrow\infty} I^{J_1}(\xi_n) < I^{J_1}(\xi)$. 

Although the standard LDP cannot be satisfied, the extended LDP can be established w.r.t.\ the $M_1'$ topology with a slightly different rate function, however. 
The next theorem establishes the extended LDP for $\bar X_n$ w.r.t.\ the $M_1'$ topology.
It should be clear from the proof of Theorem~\ref{theorem:ex-ldp-xnbar-m1prime} that the $M_1'$ extended LDP in the theorem doesn't provide any new useful bounds other than those that are already implied by the $J_1$ counterpart in Theorem~\ref{theorem:ex-ldp-xnbar}. 
However, we state the next theorem here and prove it in Section~\ref{section: extended-ldp-J1-proof} for the purpose of completeness.


\begin{theorem}\label{theorem:ex-ldp-xnbar-m1prime}
\linksinthm{theorem:ex-ldp-xnbar-m1prime}
The sequence $\{\bar X_n\}_{n\geq 1}$ defined in (\ref{def:xnbar}) satisfy the extended LDP in $(\mb{D}, d_{M'_1})$ with the rate function given by
\begin{equation}\label{notation:m1prime-rate-func}
    \newnota{i-xnbar-rate-function-m1prime}{I^{M'_1}(\xi)}\defnota{i-xnbar-rate-function-m1prime}{ = 
    \begin{cases}
        \sum_{t\in[0,1]} \I\{\xi(t) \neq \xi(t-)\} & \xi\in\hat{\mb{D}}_{< \infty} \\
        \infty & \text{otherwise}
    \end{cases}}
\end{equation}
and speed $r(\log n)$.
\end{theorem}

\subsection{Extended LDP for random walks }\label{section:randomwalk-ldp-J1}

Let $Z_i$'s be i.i.d.\ random variables, and consider the random walk \newnota{w-random_walk}{$W(t)$}\defnota{w-random_walk}{ $= \sum_{i=1}^{\floor{t}} Z_i$} embedded in $\D$.
Throughout the rest of the paper, we make the following assumption on the tail distributions of $Z_i$'s:
\begin{assumption}\label{assumption-lognormaltail2}
    $Z_i$'s takes non-negative values with probability 1 and satisfies  
    \begin{equation}
        \pr{Z_i>x} = \exp(-r(\log x)) 
\end{equation} 
for a regularly varying function $r(\cdot)$ with index $\gamma > 1$.
\end{assumption}
\noindent
Here we emphasize again that $Z_i$'s with standard lognormal distribution satisfies the above assumption with $\gamma = 2$.
This section establishes the extended LDP for the scaled processes:
\begin{equation}\label{definition:scaled-centered-randomwalk}
    \newnota{w-scaled_centered_RW}{\bar{W}_n(t)}\defnota{w-scaled_centered_RW}{ = \frac{1}{n} \sum_{i=1}^{\floor{nt}} (Z_i - \E{Z_1})}.
\end{equation}

Unlike the L\'evy process $\bar{X}_n$ in section \ref{section:extended-LDP-levy-J_1}, the random walks $\bar{W}_n$'s have jump at $t=1$. 
This makes the rate function finite on $\tilde{\mb{D}}_{<\infty}$, not just $\mb{D}_{<\infty}$. 
To recall the definitions of 
$\mb{D}_{<\infty}$ 
and 
$\tilde{\mb{D}}_{<\infty}$, 
see section \ref{section:notation-definitions}. 
The following is the main result of this subsection.

\begin{theorem}\label{theorem: extended-ldp-randomwalk}
\linksinthm{theorem: extended-ldp-randomwalk}
The sequence of random walk $\{\bar{W}_n\}_{n\geq 1}$ satisfies the extended large deviation principle on $(\mb{D}, J_1)$ with the rate function
\begin{equation}\label{def:rate-function-randomwalk-J1}
    \newnota{i-rate-random-walk}{\tilde{I}(\xi)}=
    \defnota{i-rate-random-walk}{\begin{cases}
    \sum_{t\in[0,1]} \I\{\xi(t)\neq\xi(t-)\}&\text{ if } \xi\in\tilde{\mb{D}}_{<\infty}\\
    \infty & \text{otherwise}
    \end{cases}}
\end{equation}
and speed $r(\log n)$.
\end{theorem}

\subsection{Extended LDP for multi-dimensional processes}
\label{section:extended-ldp-j1-multidimension}
This section establishes the extended LDP for multi-dimensional L\'evy processes and random walks, each of whose coordinates are independent, and their increment distributions are lognormal-type.
We accomplish this by building on the one-dimensional results in Section~\ref{section:extended-LDP-levy-J_1} and \ref{section:randomwalk-ldp-J1}. 
Note that since we are working with extended LDPs---as opposed to standard LDPs---and the rate functions that controls our extended LDPs are not good, standard results such as Theorem 4.14 of \cite{MR2045489} cannot be directly applied in our context.
However, the discrete nature of our rate function enables us to adapt some key proof ideas in Theorem 4.14 of \cite{MR2045489} to establish the same result in our context in Proposition \ref{theorem:product-extended-ldp} without resorting to the goodness of the rate function.
Theorem~\ref{theorem:product-extended-ldp} is formulated in general metric spaces. 
Specifically, we consider the following sequence of random variables: 
$$
\bm X_n = (X^{(1)}_n,\cdots, X^{(k)}_n)
$$
where for each $i=1,\ldots, k$,

\begin{enumerate}[(i)]
\item 
$X^{(i)}_n$ is a random object taking value in the metric space $(\ms{X}^{(i)}, d^{(i)})$.
\item 
$\{X^{(i)}_n\}_{n\geq 1}$ satisfies the extended LDP with the rate function $I^{(i)}$ and speed $a_n$
\item 
$I^{(i)}$ takes at most countable distinct values in $\mb{R}_+$ with no limit point.

\end{enumerate}


\begin{proposition}\label{theorem:product-extended-ldp}
\linksinthm{theorem:product-extended-ldp}
Suppose that $\{\bm{X}_n\}_{n\geq 1}$ satisfies conditions (i), (ii), and (iii) above. 
Then $\{\bm{X}_n\}_{n\geq 1}$ satisfies the extended LDP on the product space $\big(\prod_{i=1}^k \ms{X}^{(i)}, \max_{i=1}^k d^{(i)}\big)$ with the rate function $\bar I = \sum_{i=1}^k I^{(i)}$  and speed $a_n$.
\end{proposition}

Moving back to L\'evy processes, for $i= 1,\cdots, d$, let $X^{(i)}$ be an independent one-dimensional L\'evy process with generating triplet $(a_i, b_i, \nu_i)$. 
Each $\nu_i$ satisfies Assumption~\ref{assumption-lognormaltail} with regularly varying $r_i(x) = \log cx^\beta -\lambda_i \log^\gamma x$ with a common $\gamma > 1$.
\here%
Consider the scaled and centered processes 
\begin{equation}
    \bar{X}^{(i)}_n(t) =\frac{1}{n}\big[X^{(i)}(nt) - b_int - nt\int_{[1,\infty)}x\nu_i(dx)\big]\text{ for } t\in [0,1]
    .
\end{equation}
Theorem \ref{theorem:ex-ldp-xnbar} ensures that each $\{\bar{X}^{(i)}_n\}_{n\geq 1}$ on $(\mb{D}, d_{J_1})$ satisfies the extended LDP with the rate function $I_{\lambda_i}^{J_1}$, which only takes values in $\{k\lambda_i: k\in\mb{N}\}$. 
Hence those $d$ L\'evy processes meet requirements (i), (ii) and (iii), and this directly leads to the following implication of Proposition~\ref{theorem:product-extended-ldp}:

\begin{theorem}\label{corollary: extended-LDP-multidimensional}
\linksinthm{corollary: extended-LDP-multidimensional}
On $\big(\prod_{i=1}^d \mb{D},\sum_{i=1}^d d_{J_1}\big)$, the multidimensional L\'evy process with independent coordinates $\bar{\bm X}_n= (\bar{X}^{(1)}_n, \cdots \bar{X}^{(d)}_n)$  satisfies the extended LDP with the rate function $I^d:\prod_{i=1}^d \mb{D}\rightarrow \mb{R}$ and speed $r(\log n)$.
\begin{equation}\label{defnition:rate-function-multidimensional}
    \newnota{i-rate-function-multidimensional}{I^d(\xi_1,\cdots, \xi_d)} \defnota{i-rate-function-multidimensional}{\delequal \sum_{i=1}^d I_{\lambda_i}^{J_1}(\xi_i)}
\end{equation}
Here $I_{\lambda_i}^{J_1}(\cdot)$ is defined in (\ref{xnbarrate2}).
\end{theorem}

Likewise, the extended LDP for the multidimensional random walks follows directly from Proposition~\ref{theorem:product-extended-ldp}.
\begin{theorem}\label{corollary: extended-LDP-multidimensional-rw}
\linksinthm{corollary: extended-LDP-multidimensional-rw}
    On $\big(\prod_{i=1}^d \mb{D},\sum_{i=1}^d d_{J_1})$, the vector-valued sequence $\{\notationdef{scaled-and-centered-rw-multi}{\bar{\bm S}_n}= (\bar{S}^{(1)}_n, \cdots \bar{S}^{(d)}_n)\}_{n\geq 1}$ with independent coordinates satisfies the extended LDP with speed $\log^\gamma n$ and the rate function $\tilde{I}^d:\prod_{i=1}^d \mb{D}\rightarrow \mb{R}_+$:
\begin{equation}\label{defnition:rate-function-multidimensional-rw}
    \newnota{i-rate-function-multidimensional-rw}{\tilde{I}^d(\bm\xi)}\defnota{i-rate-function-multidimensional-rw}{\delequal  
    \begin{cases}
    \sum_{i=1}^d\lambda_i\cdot\sum_{t\in[0,1]} \I\{\xi_i(t)\neq\xi_i(t-)\}&\text{ if } \bm\xi\in\prod_{i=1}^d\tilde{\mb{D}}_{<\infty}\\
    \infty & \text{otherwise}
    \end{cases}}
\end{equation}
\end{theorem}

\subsection{Nonexistence of standard LDPs}\label{section:nonexist-standard-ldp}
In subsection \ref{section:extended-LDP-levy-J_1}, we established the extend LDP for L\'evy processes under the $J_1$ topology and $M'_1$ topology.
A natural question is whether it is possible to strengthen those results and establish the standard LDPs. 
Theorem~\ref{proposition-nonexistence-of-standard-LDP} below constructs a counterexample to confirm that the extended LDPs can not be strengthened to the standard LDPs.

\begin{theorem}
\label{proposition-nonexistence-of-standard-LDP}
\linksinthm{proposition-nonexistence-of-standard-LDP}
Recall $\bar X_n$ given in \eqref{def:xnbar} and 
assume that its L\'evy measure $\nu$ is supported on $\mb{R}_+$ and satisfies $\nu[x,\infty) = cx^\beta e^{-\lambda\log^\gamma x}$. 
That is, $\nu[x,\infty) = \exp(-r(\log n))$ with $r(x) = \log cx^\beta -\lambda \log^\gamma x$ and $\gamma > 1$.  
Then $\bar X_n$ cannot satisfy an LDP in the $M_1'$ topology. 
In particular, there exists a set $F \subseteq \D$ such that 
\begin{equation}\label{eq:theorem-nonexistence}
        \limsup_{n\rightarrow\infty}\frac{\log \pr{\bar{X}_n\in F}}{r(\log n)}
        > 
        -\inf_{\xi\in \bar{F}} I^{M'_1}(\xi) 
        ,
\end{equation}
where the closure $\bar F$ of $F$ is w.r.t.\ the $M_1'$ metric.
Since the $M_1'$ metric is  bounded by the $J_1$ metric, this also implies that the $\bar X_n$ cannot satisfy an LDP in the $J_1$ topology. 
\end{theorem}

Before proving Theorem~\ref{proposition-nonexistence-of-standard-LDP}, we set some notations and state two lemmata, Lemma~\ref{lemma:F-disjoint-property} and \ref{lemma:F-disjoint-lowerbound}. 
The proofs these lemmata are deferred to Section~\ref{section: nonexist-standard-ldp-proof}.
Consider the following mapping:
\linkdest{notation:map-pi}
\begin{align}\label{pimap}
    \notationdef{path-to-two-largest-jump-process}{\pi}: \hat{\mb{D}}_{<\infty}
    &
    \longrightarrow
    \hat{\mb{D}}_{\leqslant 2}
    \nonumber
    \\
    \xi
    &
    \longmapsto\pi(\xi) = 
    \begin{cases}
        \xi 
        &
        \xi\in\hat{\mb{D}}_{\leqslant 2}
        \\
        a_1\I_{[t_1,1]} + a_2\I_{[t_2,1]} 
        &
        \text{otherwise}
    \end{cases},
\end{align}
where $a_1$ and $a_2$ are the first and second largest jump sizes of $\xi$, and $t_1$ and $t_2$ are the earliest times for those jumps. 
Consider the following sets:
\linkdest{notation:set-An-Bn-Cn-Fn-F}
\begin{align}
    &
    A_n
    \delequal
    \bigg\{
        \xi 
        = 
        \sum_{i=1}^2 
            z_i\I_{[v_i,1]}: 
                z_1\in[\log n,\infty), 
                z_2\in\Big[\frac{1}{n^{1/3}},\infty\Big), 
                z_1 \geq z_2, 
                v_1\in\Big(\frac{1}{4}, \frac{1}{2}\Big], 
                v_2\in\Big(\frac{3}{4}, 1\Big]
    \bigg\}
    \label{setan}
    \\
    &
    B_n 
    \delequal
    \pi^{-1}(A_n)
    \label{setbn}
    \\
    &
    C_n 
    \delequal
    \bigg\{
        \xi\in \mb{D}: 
            d_{\|\cdot\|}(\xi ,\, -\mu_1\nu_1 e)
            \leq 
            \frac{1}{3} \frac{1}{n^{1/3}}
    \bigg\}
    \label{setcn}
    \\
    &
    F_n
    \delequal
    \big\{
        \eta\in \mb{D}: \eta = \xi_1 + \xi_2, \xi_1\in B_n, \xi_2\in C_n
    \big\}
    \label{setfn}
\end{align}
For some $N$ such that $\log N - \frac{1}{2}\nu_1\mu_1 - \frac{1}{3}N^{-\frac{1}{3}} > 1$, we further define
\begin{equation}\label{definition:counterexample-F}
    F \delequal \cup_{n=N}^\infty F_n
\end{equation}
The following two lemmata are key to Theorem~\ref{proposition-nonexistence-of-standard-LDP}. 
\begin{lemma}\label{lemma:F-disjoint-property}
\linksinthm{lemma:F-disjoint-property}
On $(\mb{D}, d_{M'_1})$, the set $F$ defined in (\ref{definition:counterexample-F}) satisfies $\bar{F}\cap\hat{\mb{D}}_{\leqslant 1} = \emptyset$.
\end{lemma}

\begin{lemma}\label{lemma:F-disjoint-lowerbound}
\linksinthm{lemma:F-disjoint-lowerbound}
    The set $F$ defined in (\ref{definition:counterexample-F}) satisfies
    \begin{equation}\label{counterexample-lowerbound}
        \limsup_{n\rightarrow\infty}\frac{\log \pr{\bar{X}_n\in \bar{F}}}{r(\log n)} > -2, 
    \end{equation}
\end{lemma}

\begin{proof}[Proof of Theorem~\ref{proposition-nonexistence-of-standard-LDP}]
Lemma \ref{lemma:F-disjoint-property} implies that 
\begin{equation}\label{eq:interpretation-of-lemma-3.1}
    - 2 \geq -\inf_{\xi\in \bar{F}} I^{M'_1}(\xi).
\end{equation}
The conclusion follows from \eqref{eq:interpretation-of-lemma-3.1} and Lemma~\ref{lemma:F-disjoint-lowerbound}.

\end{proof}


\subsection{Example: boundary crossing with regulated jumps}
\label{subsec:boundary-crossing-example}
This section illustrates an application of Theorem \ref{theorem:ex-ldp-xnbar} to analyze the large deviation probabilities associated with level crossing events. 
Specifically, we consider the probability
\begin{equation}\label{level-crossing-event}
    \P
    \bigg(
        \sup_{t\in[0,1]}\bar{X}_n\geq b,\ 
        \sup_{t\in[0,1]}|\bar{X}_n(t) - \bar{X}_n(t-)|\leq c
    \bigg)
    .
\end{equation}
This probability is closely related to the insolvency risk of reinsured insurance lines in actuarial science. 
To facilitate the use of extended LDP, we introduce the mapping
\begin{align}
    \phi: \mb{D}\longrightarrow
    &
    \mb{R}^2\nonumber
    \\
    \xi\longmapsto
    &
    \bigg(
        \sup_{t\in[0,1]} \xi(t), \ 
        \sup_{t\in[0,1]} \big|\xi(t)- \xi(t-)\big|
    \bigg),
\end{align}
which captures the maximum value of a function $\xi$ and its largest jump.
The probability in (\ref{level-crossing-event}) can be rephrased as follows:
\begin{equation}\label{level-crossing-event2}
    \pr{\phi(\bar{X}_n)\in [b,\infty)\times [0,c]}.
\end{equation}
The Lipschitz continuity of $\phi$ (as discussed in Section 4 of \cite{MR4187125}) enables us to apply the contraction principle (Lemma B.3 in \cite{MR4493394}) to the extended LDP of $\{\bar{X}_n\}_{n\geq 1}$ established in Theorem \ref{theorem:ex-ldp-xnbar}. Consequently, we deduce that $\{\phi(\bar{X}_n)\}_{n\geq 1}$ satisfies the extended LDP with the rate function given by
$$
I_\phi(x,y)
= 
\inf
    \bigg\{
        I^{J_1}(\xi): \sup_{t\in[0,1]}\xi(t)=x, \ 
        \sup_{t\in[0,1]}|\xi(t)-\xi(t-)| = y
    \bigg\}
=
\bigg\lceil\frac{x}{y}\bigg\rceil
.
$$
That is, 
\begin{align*}
    -\inf_{(x,y)\in([b,\infty)\times[0,c])^\circ} \bigg\lceil\frac{x}{y}\bigg\rceil
    &
    \leq
    \liminf_{n\rightarrow\infty} \frac{\log\pr{\phi(\bar{X}_n)\in [b,\infty)\times[0,c]}}{r(\log n)} 
    \\
    &
    \leq
    \limsup_{n\rightarrow\infty} \frac{\log\pr{\phi(\bar{X}_n)\in [b,\infty)\times[0,c]}}{r(\log n)} 
    \leq 
    -\lim_{\epsilon\downarrow 0}\inf_{(x,y)\in([b,\infty)\times[0,c])^\epsilon} \left\lceil\frac{x}{y}\right\rceil.
\end{align*}
If $b/c$ is not an integer, this yields a tight asymptotic limit
$$
\lim_{n\rightarrow\infty} \frac{\log\pr{\phi(\bar{X}_n)\in [b,\infty)\times[0,c]}}{r(\log n)} = -\left\lceil \frac{b}{c}\right\rceil,
$$
whereas if $b/c$ is an integer, we get a lose asymptotic limit 
$$
    -\frac{b}{c}-1\leq\liminf_{n\rightarrow\infty} \frac{\log\pr{\phi(\bar{X}_n)\in [b,\infty)\times[0,c]}}{r(\log n)}\leq \limsup_{n\rightarrow\infty} \frac{\log\pr{\phi(\bar{X}_n)\in [b,\infty)\times[0,c]}}{r(\log n)} \leq -\frac{b}{c}.
$$

\section{Proofs}\label{section:proofs}

\subsection{Proofs for Section~\ref{section:notation-definitions}}
\label{section:preliminary-proof}

\begin{proof}[Proof of Lemma \ref{lemma:m1prime-gap-criteria}]
\linksinpf{lemma:m1prime-gap-criteria}
Fix a $(u,r) \in \Gamma(\xi)$. 
Let $(v_1, s_1) \in \Pi(\xi)$ and $(v_2, s_2) \in \Pi(\zeta)$ be given arbitrarily. 
There exists a $t_0$ such that $v_1(t_0) = u$ and $s_1(t_0) = r$. 
\begin{align*}
    d\big((u,r),\, \Gamma(\zeta) \big)
    &
    \leq
    d\Big(\big(v_1(t_0),s_1(t_0)\big),\, \big(v_2(t_0), s_2(t_0)\big)\Big)
    \\
    &
    =
    \big|v_1(t_0)-v_2(t_0)\big| \vee \big| s_1(t_0)-s_2(t_0)\big|
    \leq
    \|v_1-v_2\| \vee \|s_1 - s_2\|.
\end{align*}
Since the choice of $(v_1, s_1)$ and $(v_2, s_2)$ was arbitrary, we arrive at the conclusion of the lemma by taking the infimum over $(v_1, s_1)\in\Pi(\xi)$ and $(v_2, s_2) \in \Pi(\zeta)$. 
\end{proof}

\begin{proof}[Proof of Lemma \ref{lemma:constant-jump-m1prime-distance}]
\linksinpf{lemma:constant-jump-m1prime-distance}
    Since $(\zeta(t), t)$ and $(\zeta(s-),s)$ belong to $\Gamma(\zeta)$, it is enough, in view of Lemma \ref{lemma:m1prime-gap-criteria}, to show that 
    \begin{equation}\label{eq:lower-bound-M1'-vee}
        d\big((\zeta(t), t), \Gamma(\xi)\big) \vee d\big((\zeta(s-), s), \Gamma(\xi)\big) \geq \delta                 
        .
    \end{equation}
    To see that this is the case, note first that either 
    \begin{equation*}
        |\zeta(t)-\xi(t)| \geq \delta
        \text{\quad or \quad}
        |\zeta(s-) - \xi(s) | \geq \delta.
    \end{equation*}
    Suppose that $|\zeta(t)-\xi(t)| \geq \delta$.
    Note also that for any $(u,r) \in \Gamma(\xi)$, either $u = \xi(t)$ or $|t-r| \geq \delta$.
    Then,
    $$
        d\big((\zeta(t),t), (u,r)\big)  
        =
        |\zeta(t) - u| \vee |t - r|
        \geq
        \delta.
    $$    
    Taking infimum over $(u,v)\in \Gamma(\xi)$, we get
    $
        d\big((\zeta(t), t), \Gamma(\xi)\big)> \delta.
    $
    Similarly,
    $
        d\big((\zeta(s-), s), \Gamma(\xi)\big)> \delta
    $
    if $|\zeta(s-) - \xi(s) | \geq \delta$.
    Therefore, we have \eqref{eq:lower-bound-M1'-vee}.

\end{proof}

\begin{proof}[Proof of Lemma~\ref{E-LDP-on-subspaces-full-measure}]
\linksinpf{E-LDP-on-subspaces-full-measure}
Since $\mb{E}$ is a closed set, it is clear that newly defined function $I'$ is lower-semicontinuous on $\mb{D}$. 
For any open set $G$ in $\mb{D}$, $G\cap\mb{E}$ is open in $\mb{E}$ and $I$ take $\infty$ value on $G\cap \mb{E}^c$. Therefore,
\begin{equation*}
    \liminf_{n\rightarrow\infty} \frac{\log\pr{X_n\in G}}{a_n} = \liminf_{n\rightarrow\infty} \frac{\log\pr{X_n\in G\cap\mb{E}}}{a_n}\geq -\inf_{x\in G\cap\mb{E}}I(x) = -\inf_{x\in G}I'(x)
\end{equation*}
Also, for any closed set $F$ in $\mb{D}$, $(F\cap\mb{E})^\epsilon\subset F^\epsilon$, hence 
\begin{equation*}
    \limsup_{n\rightarrow\infty} \frac{\log\pr{X_n\in F}}{a_n} = \limsup_{n\rightarrow\infty} \frac{\log\pr{X_n\in F\cap\mb{E}}}{a_n}\leq -\lim_{\epsilon\downarrow 0}\inf_{x\in (F\cap\mb{E})^\epsilon}I(x) \leq -\lim_{\epsilon\downarrow 0}\inf_{x\in F^\epsilon}I'(x)
\end{equation*}
Since the upper and lower bounds for the extended LDP are all satisfied with the lower-semicontinuous function $I'$, the lemma is proved.
\end{proof}

\subsection{Proofs for Section \ref{section:extended-LDP-levy-J_1}}\label{section: extended-ldp-J1-proof}

This section proves Theorem~\ref{theorem:ex-ldp-xnbar} and Theorem~\ref{theorem:ex-ldp-xnbar-m1prime}. 
Before describing the structures of the proofs, we introduce some notations. 
Rearranging the terms in 
(\ref{equ:levy-ito-decomposition}),
we can decompose $\bar X_n$ as follows:
\begin{equation}
    \bar{X}_n = \newnota{ynbar-big-jumps-process}{\bar{Y}_n} + \newnota{rnbar-martingale-process}{\bar{R}_n}
\end{equation}
where
\linkdest{notation:ynbar-and-rnbar}
\begin{align}
    &
    \bar{Y}_n(t) 
        \delequal 
        \frac{1}{n}
        \int_{[1,\infty)}
        (x-\mu_1)
        \hat{N}
        ([0,nt]\times dx),
    \label{ynbar}
    \\
    &
    \bar{R}_n(t) 
    \delequal 
    \frac{aB(nt)}{n} 
        + \frac{1}{n}
        \int_{(0,1)}    
        x
        \big(
            \hat{N}
            ([0,nt]\times dx)
            - nt\nu(dx)
        \big)
        + \frac{\mu_1\hat{N}\big([0,nt]\times [1,\infty)\big)}{n}
        - t\nu_1 \mu_1
    \label{rnbar}
\end{align}
\mdefnota
    {ynbar-big-jumps-process}
    {$=\frac{1}{n}\int_{[1,\infty)}(x-\mu_1)\hat{N}([0,nt]\times dx)$}%
\mdefnota
    {rnbar-martingale-process}
    {$=\frac{aB(nt)}{n} + \frac{1}{n}\int_{(0,1)}x\big(\hat{N}([0,nt]\times dx) - nt\nu(dx)\big)+ \frac{\mu_1\hat{N}\big([0,nt]\times [1,\infty)\big)}{n} - t\nu_1\mu_1$}%
for $t\in [0,1]$,
where
\linkdest{notation:nu1-and-mu1}%
\newnota
    {n-nu-1}
    {$\nu_1$}
\defnota
    {n-nu-1}
    {$ = \nu[1,\infty)$} 
and 
\newnota
    {m-mu-1}
    {$\mu_1$}
\defnota
    {m-mu-1}
    {$ = \int_{[1,\infty)} x\frac{\nu(dx)}{\nu_1}$} 
.

Since $\hat N$ has a finite mean measure on $[1,\infty)\times [0, n]$, the process $\bar{Y}_n$ is a compound Poisson process. 
Before being centered and normalized by $\mu_1$ and $1/n$, respectively, the jump sizes of $\bar Y_n$ are independent and identically distributed (i.i.d.) with the distirbution $\nu|_{[1,\infty)}/\nu_1$. 
Let 
\notationdef
    {big-jump-generic-rv}
    {$Z_1, Z_2, \cdots$}
be i.i.d.\ random variables with distribution $\nu|_{[1,\infty)}/\nu_1$ and let 
\notationdef
    {poisson_process}
    {$N(t)$}%
    $\delequal \hat{N}([0,t]\times[1,\infty))$.
Then, $\bar{Y}_n$ has the following distributional representation:
\begin{equation}\label{eq: first distributional representation of Y_n bar}
    \bar{Y}_n(t) \stackrel{\ms{D}}{=} \frac{1}{n}\sum_{i=1}^{N(nt)} (Z_i - \mu_1).
\end{equation}
We further decompose \eqref{eq: first distributional representation of Y_n bar} into two parts: the $k$ largest jumps and the rest. 
Let 
\notationdef
    {permutation}
    {$P_n(\cdot)$} 
be the random permutation of the indices $\{1,2,\cdots, N(n)\}$ of $Z_i$'s such that $P_n(i)$ is the rank of $Z_i$ in the decreasing order among $Z_1,Z_2,\cdots, Z_{N(n)}$. Then

\begin{equation}\label{checkjnk}
    \bar{Y}_n(t) \stackrel{\ms{D}}{=} 
    \underbrace{\frac{1}{n} \sum_{i=1}^{N(nt)}Z_i\I\{P_n(i)\leq k\}}_{\bar{J}_n^{\leqslant k}(t)} +   \underbrace{\frac{1}{n} \sum_{i=1}^{N(nt)}\big(Z_i\I\{P_n(i)> k\} - \mu_1\big)}_{{\bar{H}_n^{\leqslant k}(t)}}
\end{equation}
\mdefnota{jncheck-over-k-largest-jump}{$=\frac{1}{n} \sum_{i=1}^{N(nt)}\big(Z_i\I\{P_n(i)> k\} - \mu_1\big)$}%
\mdefnota{jnbar-k-largest-jump}{$\frac{1}{n} \sum_{i=1}^{N(nt)}Z_i\I\{P_n(i)\leq k\}$}%
Note that
\newnota
    {jnbar-k-largest-jump}
    {$\bar{J}_n^{\leqslant k}$} 
is the process that consists of the $k$ largest jumps in $\bar{Y}_n$, whereas \newnota{jncheck-over-k-largest-jump}{$\bar{H}_n^{\leqslant k}$} is the centered process that consists of the remaining jumps.
Let $Y_i$'s be i.i.d.\ exponential random variables with mean $1$, 
and \newnota{gamma-i}{$\Gamma_i$}\defnota{gamma-i}{$\triangleq Y_1 + \cdots +Y_i$}.
Let \newnota{uniform-rv}{$U_i$}'s be \defnota{uniform-rv}{i.i.d.\ uniform random variables on $[0,1]$}, independent from the $Y_i$'s. 
Let 
\begin{equation}\label{definition:Qn_and_inverse}
    \newnota
        {q-n}
        {Q_n(\cdot)}
    \defnota
        {q-n}
        {\triangleq n\nu[\cdot,\infty)}
        \qquad\text{and}\qquad
    \newnota
        {q-n-inverse}
        {Q^\leftarrow_n(y)}
    \defnota
        {q-n-inverse}
        {\triangleq \inf\{s > 0:n\nu[s,\infty) < y\}}.
\end{equation}
According to the argument presented in \cite{MR3271332} (p.305) and \cite{MR2271424} (p.163), a coupling with respect to $Y_i$, $U_i$ and $\hat{N}$ can be constructed by
\begin{equation}\label{poisson-measure-pointmass}
    \hat{N}([0,n]\times (0,\infty)) = \sum_{i=1}^\infty \epsilon_{(n\cdot U_i,\ Q^\leftarrow_n(\Gamma_i))}
\end{equation}
Here 
\newnota{epsilon_dirac-measure}{$\epsilon_{(x,y)}$} denotes 
\defnota{epsilon_dirac-measure}{the Dirac measure at $(x, y)$}. 
The sequence $\{Q^\leftarrow_n(\Gamma_i), i\geq 1\}$ record the second coordinate value of point masses of $\hat{N}$ with decreasing order.
The coupling in (\ref{poisson-measure-pointmass}) facilitates a further decomposition of $\bar{J}_n^{\leqslant k}$ into \newnota{jnhat-k-largest-jump-process}{$\hat{J}_n^{\leqslant k}$} + \newnota{j_tilde-compensate-over-1-process}{$\check{J}^{\leqslant k}_n$},where
\mdefnota{jnhat-k-largest-jump-process}{$= \frac{1}{n}\sum_{i=1}^k Q_n^\leftarrow (\Gamma_i)\I_{[U_i,1]}(t)$}
\mdefnota{j_tilde-compensate-over-1-process}{$= - \frac{1}{n}\I\{\tilde{N}_n < k\}\sum_{i=\tilde{N}_n+1}^k Q_n^\leftarrow (\Gamma_i)\I_{[U_i,1]}(t)$}
\begin{align}
    &
    \hat{J}_n^{\leqslant k}(t)
    \delequal 
    \frac{1}{n}\sum_{i=1}^k Q_n^\leftarrow (\Gamma_i)\I_{[U_i,1]}(t),
    \label{hatjnk}
    \\
    &
    \check{J}^{\leqslant k}_n(t)
    \delequal
    - \frac{1}{n}
    \I\{
        \tilde{N}_n < k
      \}
    \sum_{i=\tilde{N}_n+1}^k 
    Q_n^\leftarrow
    (\Gamma_i)
    \I_{[U_i,1]}(t).
    \label{tildejnk}
\end{align}
In (\ref{tildejnk}), \newnota{n_tilde-count-var}{$\tilde N_n$}\defnota{n_tilde-count-var}{$ = \sum_{i=1}^\infty \mathbb{I}_{[0,n]\times[1,\infty)}\big((n\cdot U_i, Q^\leftarrow_n(\Gamma_i))\big)$} counts the number of point masses of $\hat N$ on the set $[0,n]\times[1,\infty)$. Intuitively, $\hat{J}_n^{\leqslant k}$ is composed by the k largest jumps in $\bar{X}_n$, but if some of those jumps are of size smaller than 1 before the $1/n$ normalization, they are offset in $\check{J}^{\leqslant k}_n$.

We arrive at the distributional representation 
\begin{equation}\label{decompose-xnbar}
    \bar{X}_n\stackrel{\ms{D}}{=}\hat{J}_n^{\leqslant k} + \check{J}^{\leqslant k}_n + \bar{H}_n^{\leqslant k}+ \bar{R}_n
\end{equation}
where $\hat{J}_n^{\leqslant k} $, $ \check{J}^{\leqslant k}_n $, $ \bar{H}_n^{\leqslant k}$ and $ \bar{R}_n$ are respectively defined in (\ref{hatjnk}), (\ref{tildejnk}), (\ref{checkjnk}) and (\ref{rnbar}). 
This representation provides a clear road map for demonstrating the main result of this subsection:
Theorem \ref{theorem:k-big-jump-ldp} proves the extended LDP satisfied by the jump sizes of $\hat{J}_n^{\leqslant k}$.
That principle helps to derive the extended LDP of $\hat{J}^{\leqslant k}_n$, which is summarized in  Proposition \ref{theorem:ex-ldp-kjumpprocess}.
Note that (\ref{decompose-xnbar}) is valid for $k \in \mb{N}$.
It turns out that, for large $k$, $\bar{X}_n$'s
is predominantly determined 
by $\hat{J}_n^{\leqslant k}$'s, hence the extended LDP of the later, through applying Proposition \ref{lemma:approximate}, implies the extended LDP of the former.
To meet the conditions of Proposition \ref{lemma:approximate},
several intermediate results are shown with respect to $\check{J}^{\leqslant k}_n$, $\bar{H}_n^{\leqslant k}$ and $\bar{R}_n$.

The proof of Theorem~\ref{theorem:ex-ldp-xnbar} hinges on Proposition~\ref{theorem:ex-ldp-kjumpprocess} and Lemma~\ref{lemma:checkjnk-asymp-neg} below, whose proofs are provided in Section~\ref{subsubsection:proof-of-theorem:ex-ldp-kjumpprocess} and Section~\ref{subsubsection:proof-of-lemma:checkjnk-asymp-neg}, respectively.

\begin{proposition}\label{theorem:ex-ldp-kjumpprocess}
\linksinthm{theorem:ex-ldp-kjumpprocess}
$\{\hat{J}_n^{\leqslant k}\}_{n\geq 1}$ satisfies the extended LDP on $(\mb{D},d_{J_1})$ with the rate function 
\begin{equation}\label{rateik}
\newnota{i_hat-k-largest-jump-process-rate-function}{\hat{I}_k(\xi)}\defnota{i_hat-k-largest-jump-process-rate-function}{ \delequal 
\begin{cases}
\sum_{t\in(0,1]}\I\{\xi(t)\neq \xi(t-)\} & \text{ if }\xi\in\mb{D}_{\leqslant k}\\
\infty & \text{ otherwise}
\end{cases}}
\end{equation}
and speed $r(\log n)$.
\end{proposition}

\begin{lemma}\label{lemma:checkjnk-asymp-neg}
\linksinthm{lemma:checkjnk-asymp-neg}
Recall $\bar{H}_n^{\leqslant k}$ defined in (\ref{checkjnk}).
We conclude
\begin{equation*}
    \lim_{k\rightarrow\infty}\limsup_{n\rightarrow\infty} \frac{\log\pr{\| \bar{H}_n^{\leqslant k}\|> \epsilon}}{r(\log n)} = -\infty
\end{equation*}
\end{lemma}

With these in hand, we are ready to prove Theorem~\ref{theorem:ex-ldp-xnbar}.

\begin{proof}[Proof of Theorem \ref{theorem:ex-ldp-xnbar}]
\linksinpf{theorem:ex-ldp-xnbar}
We apply Proposition \ref{lemma:approximate}, with $\bar{X}_n$ and $\hat J_n^{\leqslant k}$ being $X_n$ and $Y^k_n$ in the Proposition. To apply Proposition \ref{lemma:approximate}, we need to verify the following four conditions:
\begin{enumerate}[(1)]
    \item 
    For each $k\in\mb{N}$, the sequence $\{\hat J_n^{\leqslant k}\}_{n\geq 1}$ satisfies an extended LDP with rate $\hat I_k$ and speed $\{r(\log n)\}_{n\geq 1}$.
    \item 
    For any closed set $F$, 
    \begin{equation}\label{require1-xnbar-theorem}
        \lim_{k\rightarrow\infty}\inf_{\xi\in F} \hat I_k(\xi)\geq \inf_{\xi\in F} I^{J_1}(\xi).
    \end{equation}
    \item
    For $\forall\delta > 0$ and any open set $G$, there exists $\epsilon > 0$ and $K\geq 0$ such that when $k\geq K$
    \begin{equation}\label{require2-xnbar-theorem}
        \inf_{\xi\in G^{-\epsilon}} \hat I_k(\xi)\leq\inf_{\xi\in G} I^{J_1}(\xi)+\delta.
    \end{equation}

    \item
    For $\forall\epsilon > 0$,
    \begin{equation}\label{require3-xnbar-theorem}
        \lim_{k\rightarrow\infty}\limsup_{n\rightarrow\infty} \frac{\log\pr{d_{J_1}(\bar X_n, \hat J_n^{\leqslant k}) > \epsilon}}{\log^\gamma  n} = -\infty.
    \end{equation}
\end{enumerate}

Condition (1) holds due to Proposition \ref{theorem:ex-ldp-kjumpprocess}. Condition (2) is direct as $\hat I_k\geq \bar I$ for any $k\in\mb{N}$.
For the condition (3), we verify a stronger equality that replace (\ref{require2-xnbar-theorem}):
\begin{equation}\label{require2-stronger-xnbar-theorem}
        \inf_{\xi\in G^{-\epsilon}} \hat I_k(\xi) =\inf_{\xi\in G} I^{J_1}(\xi).
\end{equation}
To verify (\ref{require2-stronger-xnbar-theorem}), we consider two cases: 
\begin{itemize}
    \item 
    If $\inf_{\xi\in G} I^{J_1}(\xi) = \infty$, it implies that $G\cap \mb{D}_{<\infty} = \emptyset$. Therefore, $\inf_{x\in G^{-\epsilon}} \hat{I}_k(x) = \infty$ for any $k\in\mb{N}$;

    \item
    If $0\leq \inf_{\xi\in G} I^{J_1}(\xi) = m<\infty$, the following inequalities hold for any $\epsilon >0$
\begin{equation*}
    \inf_{\xi\in G^{-\epsilon}} \hat{I}_k(\xi) \geq \inf_{\xi\in G} \hat{I}_k(\xi) \geq  \inf_{\xi\in G} I^{J_1}(\xi) = m
\end{equation*}
Additionally, $\inf_{\xi\in G} I^{J_1}(\xi) = m$ suggests the existence of some $\eta\in G\cap\mb{D}_{=m}$. Given that $G$ is open, there exists $\epsilon >0 $ such that $B_\epsilon(\eta) \subset G$, implying $\eta\in G^{-\epsilon}$. Therefore, for $k > m$, we have
\begin{equation*}
    \inf_{\xi\in G^{-\epsilon}} \hat{I}_k(x)\leq \hat{I}_k(\eta)= m
\end{equation*}
These two inequalities together confirm that  $\inf_{\xi\in G^{-\epsilon}} \hat{I}_k(\xi) =\inf_{\xi\in G} I^{J_1}(\xi)$ for $k > m$.
\end{itemize}  
For the condition (4), by $\bar{X}_n$'s representation in (\ref{decompose-xnbar}), we have 
\begin{align*}
    &
    \pr{
        d_{J_1}(\tilde{X}_n, \hat{J}_n^{\leqslant k})
        > 
        \epsilon
    }
    \leq 
    \pr{
        \|\tilde{X}_n- \hat{J}_n^{\leqslant k}\| 
        >
        \epsilon
    }
    \\
    &
    \leq
    \pr{    
        \|\bar{H}_n^{\leqslant k}\| 
        >
        \frac{\epsilon}{3}
    } 
    + 
    \pr{
        \|\check{J}^{\leqslant k}_n\| 
        >
        \frac{\epsilon}{3}
    }
    +
    \pr{
        \|\bar{R}_n\| 
        >
        \frac{\epsilon}{3}
    }.
\end{align*}
Therefore, 
\begin{align}
    &
    \lim_{k\rightarrow\infty}
    \limsup_{n\rightarrow\infty}
    \frac
        {\log\pr{d_{J_1}(\bar{Y}_n, \hat{J}_n^{\leqslant k}) > \epsilon}}{r(\log n)}
    \leq
    \lim_{k\rightarrow\infty}
    \limsup_{n\rightarrow\infty}
    \frac
        {\log\pr{\|\bar{Y}_n- \hat{J}_n^{\leqslant k}\| > \epsilon}}{r(\log n)}
    \nonumber
    \\
    &
    \leq
    \lim_{k\rightarrow\infty}
    \limsup_{n\rightarrow\infty}
    \frac{
        \log\Big(
        \pr{
            \|\bar{H}_n^{\leqslant k}\| 
            >
            \frac{\epsilon}{3}
        }
        +
        \pr{
            \|\check{J}^{\leqslant k}_n\| > \frac{\epsilon}{3}
        }
        +\pr{
        \|\bar{R}_n\| 
        >
        \frac{\epsilon}{3}}\Big)}{r(\log n)}
    \nonumber
    \\
    &
    \leq
    \max
    \big\{
        \underbrace
        {
        \lim_{k\rightarrow\infty}
        \limsup_{n\rightarrow\infty}
        \frac{
            \log
            \pr{
                \|\bar{H}_n^{\leqslant k}\|
                >
                \frac{\epsilon}{3}}}{r(\log n)}
        }_{(\text{I})}, \underbrace
        {
            \lim_{k\rightarrow\infty}
            \limsup_{n\rightarrow\infty} \frac{\log\pr{\|\check{J}^{\leqslant k}_n\|
            > 
            \frac{\epsilon}{3}}
        }
        {
            r(\log n)
        }
        }_{(\text{II})},
        \underbrace
        {
            \limsup_{n\rightarrow\infty}
            \frac
                {\log\pr{\|\bar{R}_n\| 
                > 
                \frac{\epsilon}{3}}}
            {r(\log n)}
        }_{(\text{III})}
    \big\}.
    \label{tempequation}
\end{align}
Among the three terms in (\ref{tempequation}), (I) $=-\infty$ by Lemma \ref{lemma:checkjnk-asymp-neg}.
 For term (II), by the definition of $\check{J}^{\leqslant k}_n$ (\ref{tildejnk}), we have
\begin{equation*}
    (\text{II})= \pr{\sup_{t\in[0,1]}\Big|- \frac{1}{n}\I\{\tilde{N}_n < k\}\sum_{i=\tilde{N}_n+1}^k Q_n^\leftarrow
    (\Gamma_i)\I_{[U_i,1]}(t)\Big| > \frac{\epsilon}{3}}\leq\pr{\tilde N_n < k}
\end{equation*}
Since $\tilde N_n$ can be treated as $n$ independent summation of Poisson random variables with mean $\nu[1,\infty)$, Sanov's theorem concludes that $\pr{\tilde N_n < k}\sim e^{-n\cdot C}$ for some constant $C$. This implies (II)$=-\infty$. 
For term (III), $\bar{R}_n$ is a L\'evy process such that $\bar{R}_n(1)$ has a finite moment generating function.
Theorem 2.5 of \cite{MR1207223} confirms that term (III) increases at a linear rate, thus (III)$=-\infty$.
Hence condition (4) is satisfied and the proof is finished.
\end{proof}

Now we move on to the proof of Theorem~\ref{theorem:ex-ldp-xnbar-m1prime}. 
Before providing the proof, we establish the following lemma.

\begin{lemma}\label{lemma: rate-func-m1prime-levelset-closed}
\linksinthm{lemma: rate-func-m1prime-levelset-closed}
For any $j\in\mb{N}$,  $\hat{\mb{D}}_{\leqslant j}$ is closed  w.r.t the $M'_1$ topology.
\end{lemma}
\begin{proof}
\linksinpf{lemma: rate-func-m1prime-levelset-closed}
Note that $\mb{D}\backslash\hat{\mb{D}}_{\leqslant j}=A\cup B\cup C \cup D$ where
\begin{align*}
    A
    &
    \triangleq
    \{\eta\in\mb{D}: \eta(0) < 0\}
    \\
    B
    &
    \triangleq
    \{\eta\in\mb{D}: \eta(0) \geq 0,\, \eta\text{ is not a non-decreasing function}\}
    \\
    C 
    &
    \triangleq 
    \{\eta\in\mb{D}: \eta(0) \geq 0,\, \eta\text{ is non-decreasing, but not a pure jump function}\}
    \\
    D
    &
    \triangleq
    \{\eta\in\mb{D}: \eta(0) \geq 0,\, \eta\text{ is  a non-decreasing pure jump function with more than $j$ jumps}\}
\end{align*}
We argue that any given $\eta \notin \hat \D_{\sleq j}$ 
is bounded away from $\hat{\mb{D}}_{\leqslant j}$ by considering $\eta$'s in $A$, $B$, $C$, and $D$ separately. 

Suppose that $\eta \in A$. 
For any $\xi \in \hat \D_{\sleq j}$, $(u,v)\in \Gamma(\xi)$ implies that $u\geq 0$, and hence, $d\big((\eta(0), 0),\, \Gamma(\xi)\big) \geq |\eta(0)|$.  
This along with Lemma~\ref{lemma:m1prime-gap-criteria} implies that $d_{M_1'}(\eta, \xi) > |\eta(0)|>0$. Since $\xi$ was arbitrarily chosen in $\hat \D_{\sleq j}$, we conclude that $d_{M_1'}(\eta, \hat \D_{\sleq j}) > |\eta(0)|>0$.

Suppose that $\eta \in B$. 
Then, there exists $\delta>0$ and $t_1,t_2\in [0,1]$ such that $t_2 - t_1 > 4\delta$ and $\eta(t_1) - \eta(t_2) > 4 \delta$. 
We claim that for any $\xi \in \hat \D_{\sleq j}$, 
$d_{M_1'}(\eta, \xi) > \delta$, and hence, $d_{M_1'}(\eta, \hat \D_{\sleq j}) > \delta$. 
To see why, suppose not.
That is, suppose that there exists $\xi \in \hat \D_{\sleq j}$ such that $d_{M_1'}(\eta, \xi) \leq \delta$.
Then, due to Lemma~\ref{lemma:m1prime-gap-criteria}, there are $(u_1,s_1),(u_2,s_2)\in \Gamma(\xi)$ such that 
$d\big((\eta(t_1), t_1), (u_1, s_1)\big) \leq 2\delta$ where 
$u_1 \in [\xi(s_1-),\, \xi(s_1)]$, and
$d\big((\eta(t_2), t_2), (u_2, s_2)\big) \leq 2\delta$ where
$u_2 \in [\xi(s_2-),\, \xi(s_2)]$.
Note that these imply 
$$
s_1 \leq t_1 - 2\delta < t_2 + 2\delta \leq s_2
$$
and
$$
\xi(s_1) \geq u_1 \geq \eta(t_1)-2\delta \geq \eta(t_2)+ 2\delta \geq u_2 \geq \xi(s_2-),
$$
which is contradictory to $\xi$ being non-decreasing. 
This proves the claim.

Suppose that $\eta \in C$.
Then there exists an interval within $[0,1]$ on which $\eta$ is continuous and strictly increasing. 
By subdividing the increment over this interval into the ones with small enough increments, one can find a sufficiently small $\delta>0$ and more than $j$ non-overlapping subintervals of the form $[s-\delta, t+\delta]$ such that $t-s > 2\delta$ and $\eta(t) - \eta(s) > 2\delta$. 
For any $\xi \in \hat \D_{\sleq j}$, it needs to be constant on at least one of these intervals. 
From Lemma~\ref{lemma:constant-jump-m1prime-distance}, we see that $d_{M_1'}(\eta, \xi) \geq \delta$. 
Again, since this is for an arbitrary $\xi\in \hat \D_{\sleq j}$, we conclude that $d_{M_1'}(\eta, \hat \D_{\sleq j}) \geq \delta$. 

Suppose that $\eta \in D$. 
Then $\eta = \sum_{i=1}^k z_i\I_{[t_1,1]}$ some $k > j$ and $0\leq t_1 < t_2 < \cdots < t_k \leq 1$.
Set $T_i \teq [t_i-\delta, t_i + \delta] \cap [0,1]$ and 
pick $\delta > 0$ small enough so that
$T_i$'s are disjoint
and
$    
2\delta \leq \min_{i= 1,\cdots, k} \big(\eta(t_i) - \eta(t_i-)\big)
$.
Then, any path $\xi\in \hat{\mb{D}}_{\leqslant j}$ is constant on at least one of $T_i$'s, and $\eta$ jumps at $t_i$ with jump size no smaller than $2\delta$. 
From Lemma \ref{lemma:constant-jump-m1prime-distance}, we have that $d_{M'_1}(\xi, \eta)\geq \delta/2$. Since $\xi$ is arbitrary in path in $\hat \D_{\sleq j}$, we conclude that  $d_{M_1'}\big(\eta, \hat{\mb{D}}_{\leqslant j} \big) > \delta$.
\end{proof}

\begin{proof}[Proof of Theorem \ref{theorem:ex-ldp-xnbar-m1prime}]
\linksinpf{theorem:ex-ldp-xnbar-m1prime}
Since $I^{M'_1}$ is obviously non-negative, and the sublevel set $\Psi_{I^{M'_1}}(\alpha) = \hat \D_{\sleq \lfloor \alpha \rfloor}$ of $I^{M'_1}$ is closed for any $\alpha\geq 0$ due to Lemma \ref{lemma: rate-func-m1prime-levelset-closed}, $I^{M_1'}$ is a rate function.

Consider closed set $F$ and open set $G$ under the $M'_1$ topology. They are also closed and open under the $J_1$ topology, respectively. 
Due to Theorem \ref{theorem:ex-ldp-xnbar}, we have:
\begin{equation}\label{prem1primeldp}
        -\inf_{x\in G} I^{J_1}(x) \leq \liminf_{n\rightarrow\infty} \frac{\log\pr{\bar{X}_n\in G}}{r(\log n)}\text{ and }  \limsup_{n\rightarrow\infty} \frac{\log\pr{\bar{X}_n\in F}}{r(\log n)}\leq-\inf_{x\in F^{\epsilon,J_1}} I^{J_1}(x)
\end{equation}
In the above inequalities, we use extra the super-script in $F^{\epsilon, J_1}$ to clarify that it is the $\epsilon$-fattening w.r.t.\ $J_1$ topology. 
Likewise, we will denote the $\epsilon$-fattening w.r.t.\ $M_1'$ topology with $F^{\epsilon, M'_1}$.

Note that the $M'_1$ metric is bounded by the $J_1$ metric, hence $F^{\epsilon, J_1}\subset F^{\epsilon, M'_1}$. 
Moreover, $I^{M'_1}\leq I^{J_1}$ obviously from their definitions. 
Hence 
\begin{equation}\label{m1primeldplower}
    \limsup_{n\rightarrow\infty} \frac{\log\pr{\bar{X}_n\in F}}{r(\log n)}\leq-\inf_{x\in F^{\epsilon,J_1}} I^{J_1}(x)\leq-\inf_{x\in F^{\epsilon,M'_1}} I^{M'_1}(x)
\end{equation}
which proves the upper bound for the extended LDP on $(\mb{D}, M'_1)$.

We claim that for any open set $G$ with respect to the $M_1'$ topology, the following holds:
\begin{equation}\label{equ:equal-of-rate-m1prime-j1}
    -\inf_{x\in G} I^{M'_1}(x)  = -\inf_{x\in G} I^{J_1}(x) 
\end{equation}
The '$\geqslant$' direction is trivial as $I^{M'_1}\leq I^{J_1}$. 
To see the '$\leqslant$' direction, we assume $-\inf_{x\in G} I^{M'_1}(x) = -m$ for some integer $m$. Otherwise, $-\inf_{x\in G} I^{M'_1}(x) = -\infty$ makes the lower bound inequality of extended LDP holds automatically. Under this assumption, $G$ contains some $\xi\in\hat{\mb{D}}_{=m}$, which can be categorized into one of the following two scenarios:
\begin{itemize}
    \item 
    If none of the jump of $\xi$ occur happen at $0$ or $1$. This implies $\xi\in\mb{D}_{=m}$ and 
    \begin{equation*}
        -\inf_{x\in G} I^{J_1}(x) \geq -I^{J_1}(\xi) = -m = -\inf_{x\in G} I^{M'_1}(x)
    \end{equation*}
    
    \item
    If jumps of $\xi$ occur at $0$ or $1$, or both. We can construct a new path $\xi'\in\mb{D}_{=m}$ by copying $\xi$'s jump sizes but only perturbing $\xi$'s jump location at $0$($1$) to $\delta$($1-\delta$).  By choosing $\delta$ arbitrarily small, $d_{M'_1}(\xi, \xi')\leq \delta$, hence $\xi'\in G$ due to $G$ is an open set. This implies 
    \begin{equation*}
        -\inf_{x\in G} I^{J_1}(x) \geq -I^{J_1}(\xi') = -m = -\inf_{x\in G} I^{M'_1}(x)
    \end{equation*}
\end{itemize}
Combining the conclusions from both scenarios confirm the equality (\ref{equ:equal-of-rate-m1prime-j1}), which further imlies 
\begin{equation}\label{m1primeldplower}
    -\inf_{x\in G} I^{M'_1}(x) = -\inf_{x\in G} I^{J_1}(x) \leq \liminf_{n\rightarrow\infty} \frac{\log\pr{\bar{X}_n\in F}}{r(\log n)}
\end{equation}
This proves the lower bound for the extended LDP on $(\mb{D}, M'_1)$ and finishes the proof.
\end{proof}

\subsubsection{Proof of Proposition~\ref{theorem:ex-ldp-kjumpprocess}}
\label{subsubsection:proof-of-theorem:ex-ldp-kjumpprocess}
Proposition~\ref{theorem:ex-ldp-kjumpprocess} 
hinges on the extended LDP for $\big({Q_n^\leftarrow (\Gamma_1)}/{n},\cdots, {Q_n^\leftarrow(\Gamma_k)}/{n}\big)$,
which is established
in Proposition~\ref{theorem:k-big-jump-ldp}. 
We divide the proof of Proposition~\ref{theorem:k-big-jump-ldp} into the extended LDP lower bound (Proposition~\ref{proposition:k-big-jump-ldp-lowerbound}) and upper bound (Proposition~\ref{proposition:k-big-jump-ldp-upperbound}). 
The following lemma is useful in the proof of Proposition~\ref{proposition:k-big-jump-ldp-lowerbound}.

\begin{lemma}
\label{lemma:neighborhood-of-exp-dist-for-lower-bound}
\linksinthm{lemma:neighborhood-of-exp-dist-for-lower-bound}
Suppose that $x = (x_1,\ldots, x_k) \in \R^{k\downarrow}_+$ and and $\iota = \check I_k(x)$. 
Let
$$
E(x;\delta) \teq \{ y \in \R_+^k: y_i \in E_j(x,y;\delta)\}
$$
where
\begin{align*}
E_j(x,y;\delta) = 
\left\{\begin{array}{ll}
    \big(Q_n(n(x_1+\delta),\, Q_n(n(x_1-\delta)\big) 
        &\text{for } j=1
        ;
    \\[3pt]
    \textstyle
    \big(Q_n(n(x_j+\delta)- \sum_{i=1}^{j-1} y_i,\, Q_n(n(x_j-\delta)-\sum_{i=1}^{j-1} y_i\big)
        &\text{for } j = 2,\ldots, \iota
        ;
    \\[3pt]
    \big(Q_n(n\delta)- Q_n(n(x_\iota + \delta)),\, \infty \big)
         &\text{for } j = \iota + 1
         ;
    \\[3pt]
    \big[0,\infty\big)
        &\text{for } j>\iota+1
        .
\end{array}
\right.
\end{align*}
Then $y \in E(x;\delta)$ implies that
$$
\bigg(\frac{Q^\leftarrow_n(y_1)}{n}, \cdots, \frac{Q^\leftarrow_n(y_1+\cdots +y_k)}{n}\bigg)
\in 
\prod_{i=1}^{\iota}[x_i-\delta, x_i+\delta]\times [0,\delta]^{(k-\iota)}
\teq D(x;\delta).
$$
\end{lemma}

\begin{proof}[Proof of Lemma~\ref{lemma:neighborhood-of-exp-dist-for-lower-bound}]
\linksinpf{lemma:neighborhood-of-exp-dist-for-lower-bound}
Note that by the construction of $E_j(x,y;\delta)$, it is straightforward to check by induction that if $y_i \in E_i(x,y;\delta)$ for $i=1,\ldots,j$, then 
\begin{equation}
\label{eq:1-in-lemma:neighborhood-of-exp-dist-for-lower-bound}
    \sum_{i=1}^j y_j \in  \big(Q_n(n(x_j+\delta)), Q_n(n(x_j-\delta))\big)
\end{equation}
for $j\leq \iota$, and 
\begin{equation}
\label{eq:2-in-lemma:neighborhood-of-exp-dist-for-lower-bound}
\sum_{i=1}^j y_j \in [Q_n(n\delta), \infty)
\end{equation}
for $j > \iota$. 
The conclusion of the lemma follows from the fact that $c\in \big(Q_n(b),  Q_n(a)\big]$ if and only if $Q_n^\leftarrow(c) \in [a, b)$  for  any $c$ and $b > a \geq 0$.
\end{proof}

Recall that
\begin{equation*}
    \check{\mu}_n(\cdot) \delequal \pr{\Big(\frac{Q_n^\leftarrow (\Gamma_1)}{n},\cdots, \frac{Q_n^\leftarrow(\Gamma_k)}{n}\Big)\in \cdot}\text{ for }n\in\mb{N},
\end{equation*}
where $\Gamma_i=Y_1+\cdots+Y_i$, and $Y_i$'s are iid $\mathrm{Exp}(1)$. 
Therefore, Lemma~\ref{lemma:neighborhood-of-exp-dist-for-lower-bound} implies
\begin{equation}\label{eq:implication-of-lemma:neighborhood-of-exp-dist-for-lower-bound}
\int_{D(x;\delta)} \check{\mu}_n(dx)
>
\int_{E(x;\delta)}  e^{-\sum_{i=1}^k y_i} dy_1\cdots dy_k.    
\end{equation}
We are ready to prove the lower bound of Proposition~\ref{theorem:k-big-jump-ldp}.

\begin{proposition}
\label{proposition:k-big-jump-ldp-lowerbound}
\linksinthm{proposition:k-big-jump-ldp-lowerbound}
For any open set $G\subset \mb{R}_+^{k\downarrow}$,
\begin{equation}\label{equ:checkmun-ldp-lowerbound-appendix}
    -\inf_{x\in G} \check{I}_k(x) \leq \liminf_{n\rightarrow\infty} \frac{\log\check{\mu}_n( G)}{r(\log n)}
\end{equation}
\end{proposition}

\begin{proof}
\linksinpf{proposition:k-big-jump-ldp-lowerbound}
Fix an arbitrary $x\in G$ and let $\iota = \check I_k (x)$.
We first observe that for any given $x\in G$, we can find $x'\in G$ such that $\check I_k(x') = \iota = \check I_k (x)$ and $x'_1 > x'_2 > \cdots > x'_{\iota}>0$. 
Therefore, for the purpose of establishing  \eqref{equ:checkmun-ldp-lowerbound-appendix}, we can assume w.l.o.g.\ that 
\begin{equation}\label{eq:strict-monotonicity-condition}
x_1 > \ldots > x_\iota > x_{\iota+1} = 0.
\end{equation}
Note that we can find a continuous function $f:\mb{R}_+^{k\downarrow}\rightarrow [0,1]$ such that $f(x) = 1$ and $f(y) = 0$ for $y\in G^c$ since $\mb{R}_+^{k\downarrow}$ is a completely regular topological space. 
Further, define $f_m(\cdot) = m(f(\cdot)-1)$. 
Since $f_m$ is $-m$ on $G^c$ and at most $0$ elsewhere, 
\begin{equation*}
    \int_{\mb{R}_+^{k\downarrow}} e^{r(\log n)f_m(s)} \check{\mu}_n(ds)\leq e^{-mr(\log n)}\check{\mu}_n(G^c) + \check{\mu}_n(G),
\end{equation*}
and hence,
\begin{align}
    \nonumber
    \liminf_{n\rightarrow\infty}\frac{\log\int_{\mb{R}_+^{k\downarrow}} e^{r(\log n)f_m(s)} \check{\mu}_n(ds)}{r(\log n)}
    &\leq 
    \liminf_{n\rightarrow\infty}\frac{\log\big(e^{-mr(\log n)}+ \check{\mu}_n(G)\big)}{r(\log n)} 
    \\
    &= 
    \max\bigg\{-m,\, \liminf_{n\rightarrow\infty}\frac{\log\check{\mu}_n(G)}{r(\log n)} \bigg\}.
    \label{midstep1}
\end{align}
On the other hand, we claim that the following holds:
\begin{equation}\label{inequality:midstep2}
    -\check{I}_k(x)
    \leq 
    \liminf_{n\rightarrow\infty}\frac{\log\int_{\mb{R}_+^{k\downarrow}} e^{r(\log n)f_m(s)} \check{\mu}_n(ds)}{r(\log n)}.
\end{equation}
Combining \eqref{midstep1} and \eqref{inequality:midstep2}, then taking the limit $m\to\infty$ and the infimum over $x\in G$, we arrive at the conclusion of the proposition. 

Now we are left with the proof of the claim \eqref{inequality:midstep2}.
Recall the neighborhood $D(x;\delta) = \prod_{i=1}^{\iota}[x_i-\delta, x_i+\delta]\times [0,\delta]^{(k-\iota)}$ of $x$ defined in Lemma~\ref{lemma:neighborhood-of-exp-dist-for-lower-bound}.
Pick an arbitrary $\epsilon > 0$.
Since $G$ is open and $f_m$ is continuous, assumption \eqref{eq:strict-monotonicity-condition} allows us to choose a small enough $\delta = \delta(x,\epsilon)$ so that
\begin{itemize}
    \item[(i)]
    $[x_i-\delta, x_i+\delta]\cap [x_j-\delta, x_j+\delta] = \emptyset$ for any $i$ and $j$ such that $i < j \leq \iota$;
        
    \item[(ii)]
    $f_m(y) > -\epsilon$  for $y\in D(x;\delta)$.
\end{itemize}
From (ii) and \eqref{eq:implication-of-lemma:neighborhood-of-exp-dist-for-lower-bound}
\begin{align}
    \liminf_{n\rightarrow\infty}\frac{\log\int_{\mb{R}_+^{k\downarrow}} e^{r(\log n)f_m(x)} \check{\mu}_n(dx)}{r(\log n)}
    &
    \geq
    \liminf_{n\rightarrow\infty}\frac{\log \left(e^{r(\log n)(-\epsilon)} \int_{D(x;\delta)} \check{\mu}_n(dx)\right)}{r(\log n)}
    \nonumber
    \\&
    \geq
    -\epsilon + \liminf_{n\rightarrow\infty}\frac{\log\int_{E(x;\delta)}  e^{-\sum_{i=1}^k y_i} dy_1\cdots dy_k}{r(\log n)}
    \label{eq:lower-bound-in-terms-of-exponential-rvs-in-proposition:k-big-jump-ldp-lowerbound}
\end{align}
From the forms of $E(x;\delta)$ and $E_i(x,y;\delta)$'s in Lemma~\ref{lemma:neighborhood-of-exp-dist-for-lower-bound}, we see that the integral in \eqref{eq:lower-bound-in-terms-of-exponential-rvs-in-proposition:k-big-jump-ldp-lowerbound} can be decomposed as follows:
\begin{align}
    &\log\int_{E(x;\delta)}  e^{-\sum_{i=1}^k y_i} dy_1\cdots dy_k\nonumber
    \\
    &
    =
    \log\int_{E'(x;\delta)}e^{-\sum_{i=1}^{\iota} y_i} dy_{1}\cdots dy_\iota
    + 
    \log \int_{Q_n(n\delta) - Q_n(n(x_\iota + \delta))}^\infty e^{-x}dx
    + 
    \sum_{i=\iota+2}^k\cdot \log\int_0^\infty e^{-x}dx
    \nonumber
    \\
    &
    =
    \log
    \underbrace{
    \int_{E'(x;\delta)}e^{-\sum_{i=1}^{\iota} y_i} dy_{1}\cdots dy_\iota
    }_{(\text{I})} 
    -\big(Q_n(n\delta) - Q_n(n(x_{\iota}+\delta)) \big)
    \label{eq:further-lower-bound-in-terms-of-exponential-rvs-in-proposition:k-big-jump-ldp-lowerbound}
\end{align}
where $E'(x;\delta) = \{y \in \R_+^\iota: y_i \in E_j (x, y; \delta) \text{ for }j=1,\ldots, \iota\}$.
To bound (I), note first that the integrand is bounded from below due to \eqref{eq:1-in-lemma:neighborhood-of-exp-dist-for-lower-bound}: 
$$
e^{-\sum_{i=1}^\iota y_i} \geq e^{-Q_n(n(x_\iota - \delta))}
\quad\text{on $E'(x;\delta)$}
$$
and the length of the domain of integral in $i$\nth coordinate is $Q_n(n(x_j-\delta))-Q_n(n(x_j+\delta))$ for each $i \leq \iota$, and hence,
$$
\int_{E'(x;\delta)} dy_\iota\cdots dy_1 
= 
\prod_{i=1}^\iota \big(Q_n(n(x_j-\delta))-Q_n(n(x_j+\delta))\big).
$$
Therefore, 
\begin{align}
    (I) = \int_{E'(x;\delta)}e^{-\sum_{i=1}^{\iota} y_i} dy_\iota\cdots dy_1
    &
    \geq 
    e^{-Q_n(n(x_\iota - \delta))}\int_{E'(x;\delta)} dy_\iota\cdots dy_1
    \nonumber
    \\
    &
    =e^{-Q_n(n(x_{\iota}-\delta))}\cdot \prod_{i=1}^{\iota}\big(Q_n(n(x_i-\delta)) - Q_n(n(x_i+\delta))\big).
    \label{eq:bound-on-(I)-in-proposition:k-big-jump-ldp-lowerbound}
\end{align}
From \eqref{eq:lower-bound-in-terms-of-exponential-rvs-in-proposition:k-big-jump-ldp-lowerbound}, \eqref{eq:further-lower-bound-in-terms-of-exponential-rvs-in-proposition:k-big-jump-ldp-lowerbound}, and \eqref{eq:bound-on-(I)-in-proposition:k-big-jump-ldp-lowerbound},

\begin{align*}
    &\liminf_{n\rightarrow\infty}\frac{\log\int_{\mb{R}_+^{k\downarrow}} e^{r(\log n)f_m(x)} \check{\mu}_n(dx)}{r(\log n)}\\
    &
    \geq 
    -\epsilon + \lim_{n\rightarrow\infty}\frac{-\big(Q_n(n\delta) - Q_n(n(x_{\iota}+\delta)) \big)}{r(\log n)}
    \\
    &+ \lim_{n\rightarrow\infty}\frac{-Q_n(n(x_{\iota}-\delta))}{r(\log n)} + \sum_{i=1}^{\iota}\lim_{n\rightarrow\infty}\frac{\log\big(Q_n(n(x_i-\delta)) - Q_n(n(x_i+\delta))\big)}{r(\log n)}
    \\
    &
    = -\epsilon - \iota,
\end{align*}
where the last equality is from (\ref{limit2}) and (\ref{limit4}). 
Taking $\epsilon\rightarrow 0$, we arrive at (\ref{inequality:midstep2}), proving the desired claim. 
This concludes the proof of Proposition~\ref{proposition:k-big-jump-ldp-lowerbound}.
\end{proof}

We move on to the extended LDP upper bound for $\big({Q_n^\leftarrow (\Gamma_1)}/{n},\cdots, {Q_n^\leftarrow(\Gamma_k)}/{n}\big)$.

\begin{proposition}
\label{proposition:k-big-jump-ldp-upperbound}
\linksinthm{proposition:k-big-jump-ldp-upperbound}
For any closed set $F\subset\mb{R}_+^{k\downarrow}$, 
\begin{equation}\label{equ:checkmun-ldp-upperbound-appendix}
    \limsup_{n\rightarrow\infty} \frac{\log\check{\mu}_n(F)}{r(\log n)}\leq  -\lim_{\epsilon\downarrow 0}\inf_{x\in F^\epsilon} \check{I}_k(x) 
\end{equation}
\end{proposition}

\begin{proof}
\linksinpf{proposition:k-big-jump-ldp-upperbound}
Let $\iota \teq \inf_{x\in F^\epsilon} \check I_k(x)$.
Then, it is straightforward to see that there exists $r>0$ such that $x_\iota > r$ for all $x \in F$. 
Therefore, 
\begin{align*}
    \limsup_{n\rightarrow\infty}
    \frac   
        {\log\check{\mu}_n(F)}
        {r(\log n)}
    =
    &
    \limsup_{n\rightarrow\infty}
    \frac
        {
        \log\pr{\Big(\frac{Q^\leftarrow_n(\Gamma_1)}{n}, \frac{Q^\leftarrow_n(\Gamma_2)}{n}
        , \cdots 
        \frac{Q^\leftarrow_n(\Gamma_k)}{n}\Big)\in F}
        }
        {r(\log n)}
    \\
    &
    \leq
    \limsup_{n\rightarrow\infty}
    \frac
        {\log\pr{\{ Q^\leftarrow_n(\Gamma_{\iota})
        \geq
        nr\}}}
        {r(\log n)}
    \\
    &
    =
    \lim_{n\rightarrow\infty}
    \frac
        {\log\pr{\{ \Gamma_{\iota} \leq Q_n(nr)\}}}
        {r(\log n)}
    =
    -\iota,
\end{align*}
where the last equality is from (\ref{limit5}). 
\end{proof}

With Proposition~\ref{proposition:k-big-jump-ldp-lowerbound} and Proposition~\ref{proposition:k-big-jump-ldp-upperbound} in hand, 
the proof of Proposition~\ref{theorem:k-big-jump-ldp} is straightforward.

\begin{proposition}
\label{theorem:k-big-jump-ldp}
\linksinthm{theorem:k-big-jump-ldp}
$\{\check\mu^k_n\}_{n\geq 1}$ satisfies the extended LDP on $\mb{R}_+^{k\downarrow}$ with the rate function 
\begin{equation}\label{ikrate}
    \newnota{i_check-k-largest-jump-rate-function}{\check{I}_k(\bm x)}\defnota{i_check-k-largest-jump-rate-function}{\delequal \sum_{i=1}^k \I\{x_i\neq 0\}\text{ for } \bm x = (x_1,\cdots, x_k)\in \mb{R}_+^{k\downarrow}}
\end{equation}
and the speed $r(\log n)$.
\end{proposition}

\begin{proof}[Proof of Proposition \ref{theorem:k-big-jump-ldp}]
\linksinpf{theorem:k-big-jump-ldp}
    Proposition~\ref{proposition:k-big-jump-ldp-lowerbound} 
    and \ref{proposition:k-big-jump-ldp-upperbound} establish the upper and lower bounds of the extended LDP. 
    We are left with showing that the rate function $\check{I}_k(\cdot)$ is lower semi-continuous. 
    To see the lower semi-continuity, consider an $\alpha < k$ (since the case $\alpha\geq k$ is trivial) and
    the corresponding sublevel set $\Psi_{\check I_k} (\alpha) \teq \{x \in \R_+^k: \check I_k(x) \leq \alpha\}$. 
    Note that  $\Psi_{\check I_k} (\alpha) = \{x\in \R_+^k:\ x_{\lfloor \alpha\rfloor+1} = 0\}$. 
    This is a closed set, because for any $y$ in $\Psi_{\check I_k} (\alpha)^\complement$, we have $y_{\lfloor \alpha\rfloor+1} > 0$, and hence, $B(y; \delta)$ with $\delta = y_{\lfloor \alpha\rfloor+1}/2$ is a neighborhood of $y$ within $\Psi_{\check I_k} (\alpha)^\complement$.  
    This proves the lower semicontinuity of $\check{I}_k(\cdot)$ and concludes the proof of the proposition.
\end{proof}

We are now ready to prove Proposition~\ref{theorem:ex-ldp-kjumpprocess}.
\begin{proof}[Proof of Proposition~\ref{theorem:ex-ldp-kjumpprocess}]
\linksinpf{theorem:ex-ldp-kjumpprocess}
We first prove the extended LDP on $(\mb{D}_{\leqslant k}, d_{J_1})$, then lift it to the larger space $(\mb{D}, d_{J_1}) $ via Lemma \ref{E-LDP-on-subspaces-full-measure}. 
Note that the conditions in Lemma \ref{E-LDP-on-subspaces-full-measure} are satisfied as $\P\big(\hat{J}^{\leqslant k}_n\in \mb{D}_{\leqslant k}\big) = 1$, and $\mb{D}_{\leqslant k}$ is  closed.  

To see that $\hat I_k$ is a legitimate rate function, note that $\Psi_{\hat{I}_k}(c) = \mb{D}_{\sleq\lfloor c\rfloor \wedge k}$ for any $c\in\mb{R}_+$, which is closed, and hence, $\hat{I}_k$ is lower-semicontinuous. 
To verify the lower bound, we show that 
\begin{equation}\label{equ:barjkn-ldp-lowerbound-mid-step}
    \liminf_{n\rightarrow\infty} \frac{\log\pr{\hat{J}_n^{\leqslant k}\in G}}{r(\log n)}\geq - \hat{I}_k(\xi)
\end{equation}
for any open set $G\subset \mb{D}_{\leqslant k}$ and any $\xi\in G$.
Fix a $\xi\in G$ and set $j\teq \hat I_k(\xi)\leq k$.
Then $\xi = \sum_{i=1}^j x_i\I_{[u_i,1]}$, where $x_1\geq \cdots \geq x_j > 0$ and $u_i\in (0,1)$ for $i=1,\ldots,j$. 
We construct a neighborhood of $\xi$ in $\mb{D}_{\leqslant k}$ by perturbing $\xi$'s jump sizes and times. 
Specifically, for some $w^*\in(0,1)$ that does not align with $\xi$'s jump time, consider
\begin{equation*}
    C_\delta:=\{\sum_{i=1}^k y_i\I_{[w_i, 1]}\in\mb{D}_{\leqslant k}: y_i\in Y_i\text{ and } w_i\in W_i\text{ for }i=1,\cdots, k\}
\end{equation*}
where $Y_i$'s and $W_i$'s are defined as:
\begin{equation*}
    Y_i = 
    \begin{cases}
        (x_i-\delta, x_i+\delta) & i=1,\cdots j\\
        [0,\delta) & i =j+1,\cdots, k
    \end{cases}
    ,\ \ \ \ 
    W_i = 
    \begin{cases}
        (u_i-\delta, u_i+\delta) & i=1,\cdots j\\
        (w^*-\delta, w^*+\delta) & i =j+1,\cdots, k
    \end{cases}
\end{equation*}
It is evident that $d_{J_1}(\xi, \eta) < k\delta$ for any $\eta\in C_\delta$. Indeed, we could find a suitable time homeomorphism $\rho$ such that $\eta\circ \rho$'s first $j$ jumps and $\xi$'s jumps have aligned locations. Then the uniform metric between $\xi$ and $\eta\circ\rho$ is bounded by $k$ accumulation of size $\delta$.  We can choose $\delta$ small enough so that (1) $C_\delta\subset G$; (2)for $p\leq j$, the p-th largest jump of $\eta$ happens on $(u_p-\delta, u_p+\delta)$ with size in range $(x_p-\delta, x_p+\delta)$. Such a choice of $\delta$ allows for the following derivation::
\begin{align*}
    &   
    \pr{\hat{J}_n^{\leqslant k}\in G}
    \stackrel{\text{due to }(1)}{\geq}
    \pr{\hat{J}_n^{\leqslant k}\in C_\delta}
    = 
    \pr{\sum_{i=1}^k \frac{Q_n^\leftarrow(\Gamma_i)}{n}\I_{[U_i,1]}\in C_\delta}
    \\ 
    &
    \stackrel{\text{due to }(2)}{=}
    \pr{
        \Big(
            \frac{Q_n^\leftarrow (\Gamma_1)}{n},\cdots, \frac{Q_n^\leftarrow(\Gamma_k)}{n}
        \Big)
        \in 
        \big(
            \prod_{i=1}^k Y_i
        \big)
        \cap\mb{R}_+^{k\downarrow}, (U_1, \cdots, U_k)
        \in
        \big(
            \prod_{i=1}^k W_i
        \big)
    }
    \\ 
    &
    = 
    \pr{
        \Big(
            \frac{Q_n^\leftarrow (\Gamma_1)}{n},\cdots, \frac{Q_n^\leftarrow(\Gamma_k)}{n}
        \Big)
        \in 
        \big(
            \prod_{i=1}^k Y_i
        \big)
        \cap\mb{R}_+^{k\downarrow}}\cdot \text{Const}
    \\
    &
    =
    \check\mu^k_n(
    \big(
        \prod_{i=1}^k Y_i
    \big)
    \cap\mb{R}_+^{k\downarrow})\cdot \text{Const}.
\end{align*}
Note that $\prod_{i=1}^k Y_i$ and $\prod_{i=1}^k W_i$ represent the product space of $Y_i$'s and $W_i$'s, respectively. Given that $\big(\Pi_{i=1}^k A_i\big)\cap\mb{R}_+^{k\downarrow}$ is an open set, the extended LDP of  $\{\check\mu^k_n\}_{n\geq 1}$ as unraveled by Theorem \ref{theorem:k-big-jump-ldp} implies 
\begin{align*}
    &
    \liminf_{n\rightarrow\infty}
    \frac
        {\log\pr{\hat{J}_n^{\leqslant k}\in G}}
        {r(\log n)}
    \geq
    \liminf_{n\rightarrow\infty}
    \frac
        {\log\check\mu^k_n(\big(\prod_{i=1}^k Y_i\big)\cap\mb{R}_+^{k\downarrow})}
        {r(\log n)}
    +
    \liminf_{n\rightarrow\infty}
    \frac{\log\text{Const}}{r(\log n)}
    \\
    &
    \geq
    - \inf_{
        x\in \big(\Pi_{i=1}^k Y_i\big)\cap\mb{R}_+^{k\downarrow}
        }
        \check{I}_k\big(x\big) +0 
    \\
    &
    = 
    -j 
    =
    -\hat{I}_k(\xi).
\end{align*}
This establishes inequality (\ref{equ:barjkn-ldp-lowerbound-mid-step}), and, consequently, the lower bound of extended LDP is proved.

Then, we verify the upper bound of extended LDP: for a closed set $F\subset\mb{D}_{\leqslant k}$:
\begin{equation}\label{equ:barjkn-ldp-upperbound}
    \limsup_{n\rightarrow\infty} \frac{\log\pr{\hat{J}_n^{\leqslant k}\in F}}{r(\log n)}\leq -\lim_{\epsilon\downarrow 0}\inf_{\xi\in F^\epsilon} \hat{I}_k(\xi)
\end{equation}
Note that $F^\epsilon$ in the above inequality is the fattening set of $F$ restricted to $\mb{D}_{\leqslant k}$. If $F$ contains the zero function, then $\lim_{\epsilon\downarrow 0}\inf_{\xi\in F^\epsilon} \hat{I}_k(\xi) = 0$, and the upper bound (\ref{equ:barjkn-ldp-upperbound}) holds trivially. For the rest of the proof, we assume $F$ does not contain the zero function. Define the index $i_F$ as follows:
\begin{equation}\label{jfunction}
        i_F =\max\{j\in \{0, 1, 2, \cdots k-1\}: d_{J_1}(F, \mb{D}_{\leqslant j}) > 0\}+1
\end{equation}
The definition of $i_F$ has two implications:(i) $F^\epsilon\cap \mb{D}_{\leqslant {i_F-1}} = \emptyset$ for sufficiently small $\epsilon$; (ii) $F^\epsilon\cap\mb{D}_{\leqslant i_F}\neq \emptyset$ for any $\epsilon >0 $. Therefore,
\begin{equation}\label{midstep4}
    \lim_{\epsilon\downarrow 0} \inf_{\xi\in F^\epsilon} \hat{I}_k(\xi)= i_F
\end{equation}

Given that $ d_{J_1}(F, \mb{D}_{\leqslant i_F-1}) > 0$, there exist an $r > 0$ such that any path $\xi$ in $F$ has its $i_F$-th largest jump greater than $r$. Therefore, $F$ is a subset of  $C'$, with $C'$ defined as: 
\begin{align*}
     C'
     := 
     \{
        \sum_{i=1}^k y_i\I_{[w_i,1]}:
    & 
    y_1\geq y_2\geq\cdots \geq y_k\geq 0, y_{i_F}> r,
    \\
     &
     w_i\in (0,1)\text{ for } i= 1,\cdots, k
     \}
\end{align*}
The set inclusion $F\subset C'$ implies
\begin{align*}
     &
     \pr{
        \hat{J}_n^{\leqslant k}\in F
        }
    \leq
    \pr{
        \hat{J}_n^{\leqslant k}\in C'
        }
    = 
    \pr{
        \sum_{i=1}^k \frac
            {Q_n^\leftarrow(\Gamma_i)}
            {n}
        \I_{[U_i,1]}\in C'
        }
    \\
     &
     =\pr{
        \frac
            {Q_n^\leftarrow(\Gamma_{i_F})}
            {n}
        > r
        }
     =
    \pr{
        \Gamma_{i_F}
        \leq
        Q_n(nr)
        }.
\end{align*}
Taking the limit of both sides and employing limit result (\ref{limit5}), we deduce
\begin{equation}\label{midstep3}
     \limsup_{n\rightarrow\infty} \frac{\log\pr{\hat{J}_n^{\leqslant k}\in F}}{r(\log n)}\leq\limsup_{n\rightarrow\infty} \frac{\log\pr{\Gamma_{i_F}\leq Q_n(nr)}}{r(\log n)} =  - i_F
\end{equation}
In view of the inequalities (\ref{midstep3}) and (\ref{midstep4}), the upper bound (\ref{equ:barjkn-ldp-upperbound}) is established, thus concluding the proof.

\end{proof}

\subsubsection{Proof of Lemma~\ref{lemma:checkjnk-asymp-neg}}
\label{subsubsection:proof-of-lemma:checkjnk-asymp-neg}

Lemma \ref{lemma:max-sum-zi-greaterndelta-upperbound} and \ref{lemma:max-sum-zi-smallerndelta-upperbound} below are useful in Lemma \ref{lemma:checkjnk-asymp-neg}.
\ZS{the theorem in the random walk section to be added.}%

\begin{lemma}\label{lemma:max-sum-zi-greaterndelta-upperbound}
\linksinthm{lemma:max-sum-zi-greaterndelta-upperbound}
Consider a random variable $Z$ supported on $[0,\infty)$ belonging to the class $LN(\lambda,\gamma)$. Let $Z_i$'s be i.i.d.\ random variables following $Z$'s distribution and construct $Y_i^n$ as $Y_i^n:=Z_i\I\{Z_i\leq n\delta\}$ for some given $\delta>0$. Then given any $\epsilon >0$ and $M\in\mb{N}$, the following inequality holds 
\begin{equation}\label{inequality:truncate-deviation-1}
    \limsup_{n\rightarrow\infty} \frac{\log\underset{j = 1,\cdots, M\cdot n}{\max}\pr{\sum_{i=1}^j \big(Y_i^n-\E{Z}\big) > n\epsilon}}{r(\log n)}< - \frac{\epsilon}{2\delta}
\end{equation}
\end{lemma}
\begin{proof}
\linksinpf{lemma:max-sum-zi-greaterndelta-upperbound}
The truncated random variable $Y_i^n$ differs from $Z_i$ in that it has a bounded moment generating function. This allows for the application of the Cram\'er-Chernoff method to derive an upper bound for the probability term in (\ref{inequality:truncate-deviation-1}). For $s > 0$, we have
\begin{align}
    \pr{
        \sum_{i=1}^j \big(Y_i^n-\E{Z}\big) 
        >
        n\epsilon
    }
    &
    =
    \pr{
        \exp\Big\{s\sum_{i=1}^j Y_i^n\Big\} 
        >
        \exp\big\{ns\epsilon+js\E{Z}\big\}
    }
    \nonumber
    \\
    &
    \leq
    \exp\big\{-ns\epsilon-js\E{Z}\big\}
    \E{\exp\Big\{s\sum_{i=1}^j Y_i^n \Big\}}
    \nonumber
    \\
    &
    =
    \exp\Big\{
        -ns\epsilon-js\E{Z}+ \sum_{i=1}^j \log\E{\exp(sY_i^n)}
    \Big\}
    \label{cramer-chernoff-bound}
\end{align}

To further analyze $\E{\exp(sY_i^n)}$ for $s > \frac{1}{n\delta}$, we decompose its integration domain into three parts:
\begin{equation}\label{tempfomula1}
    \E{e^{sY_i^n}} = \int_{[0,1/s]} e^{sy}d\pr{Z\leq y} + \int_{(1/s, n\delta]} e^{sy}d\pr{Z_1\leq y} + \int_{(n\delta,\infty)} 1 d\pr{Z_1\leq y}\\
\end{equation}
Given the inequality $e^x\leq 1 + x + x^2$ on $x\in[0,1]$, the first integration term in (\ref{tempfomula1}) can be bounded as:
\begin{equation*}
   \int_{[0,1/s]} e^{sy}d\pr{Z\leq y} \leq \int_{[0,1/s]} (1+ sy + s^2y^2)d\pr{Z\leq y}\leq 1+ s\E{Z} + s^2\E{Z^2}
\end{equation*}
This bound is valid since the expected value $\E{Z}$ and the second moment $\E{Z^2}$ are well-defined for a random variable $Z$ satisfying Assumption~\ref{assumption-lognormaltail2}.

We apply integration by parts to the second term in (\ref{tempfomula1}),
\begin{align*}
    &
    \int_{(1/s, n\delta]} e^{sy}d\pr{Z\leq y}
    =
    e^{sy}\pr{Z\leq y}\big|_{1/s}^{n\delta} - s\int_{1/s}^{n\delta} e^{sy}\pr{Z\leq y}dy
    \nonumber
    \\
    &
    =
    e^{sn\delta}\pr{Z\leq n\delta} - e\pr{Z\leq 1/s} - s\int_{1/s}^{n\delta} e^{sy}dy +  s\int_{[1/s, n\delta]} e^{sy}\pr{Z> y}dy
    \\
    &
    \leq
    e\pr{Z> 1/s} + s\int_{[1/s, n\delta]} e^{sy}\pr{Z> y}dy
\end{align*}
In this context,  $\Delta >0$ is chosen so that $e^{\Delta sn\delta}\pr{Z>n\delta} > \pr{Z>1/s}$ for large $n$. Therefore, for $y\in[1/s, n\delta]$,
\begin{equation*}
    e^{sy}\pr{Z > y} \leq e^{sn\delta}\pr{Z>y}\cdot\frac{e^{\Delta sn\delta}\pr{Z>n\delta}}{\pr{Z>y}} = e^{(1+\Delta)sn\delta}\pr{Z >n\delta}
\end{equation*}
This leads to an upper bound for the second term in (\ref{tempfomula1}),
\begin{equation*}
    \int_{(1/s, n\delta]} e^{sy}d\pr{Z\leq y}\leq e\pr{Z> 1/s}  +  sn\delta\cdot e^{(1+\Delta)sn\delta}\pr{Z >n\delta}
\end{equation*}

The final term in (\ref{tempfomula1}) is $\pr{Z> n\delta}$. Since $\pr{Z> n\delta} \leq \pr{Z> 1/s}\leq s^2\E{Z^2}$, combining all three integration terms in (\ref{tempfomula1}) yield
\begin{equation*}
    \E{e^{sY_i^n}}  \leq 1 + sn\delta\cdot e^{(1+\Delta)sn\delta}\pr{Z >n\delta} +  s\E{Z} + s^2(e+2)\E{Z^2}
\end{equation*}

By inserting the above bounds into (\ref{cramer-chernoff-bound}), and utilizing the inequality $\log(1+x)\leq x$ for $x >0$, we obtain 
\begin{align*}
    &
    \pr{
        \sum_{i=1}^j \big(Y_i^n-\E{Z}\big)
        >
        n\epsilon
    }
    \\
    &
    \leq 
    \exp\{
        -ns\epsilon - js\E{Z} + 
        j\log\big(
        1 + sn\delta\cdot e^{(1+\Delta)sn\delta}\pr{Z >n\delta}  + s\E{Z} + s^2(e+2)\E{Z^2}
        \big)
    \}
    \\
    &
    \leq 
    \exp\{
        -ns\epsilon - js\E{Z} + 
        j\big(
            sn\delta\cdot e^{(1+\Delta)sn\delta}\pr{Z >n\delta}  + s\E{Z} + s^2(e+2)\E{Z^2}
        \big)
    \}
    \\
    &
    \leq
    \exp\{
        -ns\epsilon  + 
        j\big(
            sn\delta\cdot e^{(1+\Delta)sn\delta}\pr{Z >n\delta} + s^2(e+2)\E{Z^2}
        \big)
    \}.
\end{align*}
Maximizing over $j= 1,\cdots, M\cdot n$ and applying the logarithm, and choosing some $\lambda'<\lambda$ and $s = \frac{\lambda'\log^\gamma(n\delta)}{n\delta}>\frac{1}{n\delta}$, we derive:
\begin{align*}
    &
    \log
    \underset
        {j = 1,\cdots, M\cdot n}
        {\max}
    \pr{
        \sum_{i=1}^j \big(Y_i^n-\E{Z}\big) 
        > 
        n\epsilon
    }
    \\
    &
    \leq
    -ns\epsilon  + M\cdot 
    n\big(
        sn\delta\cdot e^{(1+\Delta)sn\delta}
        \pr{Z >n\delta} + s^2(e+2)\E{Z^2}
    \big)
    \\
    &
    =  -\lambda'\log^\gamma(n\delta)\cdot
    \frac{\epsilon}{\delta}
    + Mn\lambda'\log^\gamma(n\delta)e^{(1+\Delta)\lambda'\log^\gamma (n\delta) }\pr{Z>n\delta}
    +
    \frac
        {M\lambda'^2
         \log^{2\gamma} (n\delta)
         }
        {n\delta^2}
    (e+2)\E{Z^2}.
\end{align*}
The third term in the last line above approach to zero as $n\rightarrow\infty$. Select $\Delta<1$ small enough such that $\lambda'(1+\Delta) < \lambda$. With this choice, and because $\pr{Z>n\delta}\sim e^{-\lambda\log^\gamma (n\delta)}$, the second term diminishes to zero as well. Hence
\begin{align*}
      &
      \limsup_{n\rightarrow\infty} 
      \frac
        {\log
         \underset
            {j = 1,\cdots, M\cdot n}
            {\max}
         \pr{
            \sum_{i=1}^j \big(Y_i^n-\E{Z}\big)
            >
            n\epsilon
            }
        }
        {r(\log n)}
      \\
      &
      \leq
      \lim_{n\rightarrow\infty}
      \frac
        { -\lambda'\log^\gamma(n\delta)\cdot
         \frac{\epsilon}{\delta}
        }
        {r(\log n)} 
      =
      -
      \frac{\epsilon}{\delta}
      \cdot
      \frac{\lambda'}{\lambda}
      <
      - \frac{\epsilon}{(1+\Delta)\delta}
      <
      - \frac{\epsilon}{2\delta}.
\end{align*}
This finishes the proof of the lemma.
\end{proof}

\linkdest{location,lemma:max-sum-zi-smallerndelta-upperbound}
\begin{lemma}\label{lemma:max-sum-zi-smallerndelta-upperbound}
\linksinthm{lemma:max-sum-zi-smallerndelta-upperbound}
Consider a random variable $Z$ supported on $[0,\infty)$ satisfying Assumption~\ref{assumption-lognormaltail2}.
Let $Z_i$'s be i.i.d.\ random variables following $Z$'s distribution and construct $Y_i^n$ as $Y_i^n:=Z_i\I\{Z_i\leq n\delta\}$ for some given $\delta>0$. Then given any $\epsilon >0$ and $M\in\mb{N}$, 
\begin{equation}\label{inequality:truncate-deviation-2}
    \limsup_{n\rightarrow\infty} \frac{\log\underset{j = 1,\cdots, M\cdot n}{\max}\pr{\sum_{i=1}^j \big(\E{Z} - Y_i^n\big) > n\epsilon}}{r(\log n)} = -\infty
\end{equation}
\end{lemma}
\begin{proof}
\linksinpf{lemma:max-sum-zi-smallerndelta-upperbound}
Define $Y := Z\I\{Z\leq n\delta\}$. The term $\pr{\sum_{i=1}^j \big(\E{Z} - Y_i^n\big) > n\epsilon}$ satisfies
\begin{align*}
    &
    \pr{
        \sum_{i=1}^j \big(\E{Z} - Y_i^n\big)
        >
        n\epsilon
    }
    \\
    &
    =
    \pr{
        \sum_{i=1}^j \big(\E{Y_i^n}- Y_i^n\big)
        >
        n\epsilon - j\cdot(\E{Z}-\E{Z_i\I(Z_i\leq n\delta)}
    }
    \\
    &
    =
    \pr{
        \sum_{i=1}^j \big(\E{Y_i^n}- Y_i^n\big)
        > 
        n\epsilon - j\cdot\E{Z\I(Z>  n\delta)}
    }.
\end{align*}
Since $Z$ satisfies Assumption~\ref{assumption-lognormaltail2}, it has the first moment, thus large $n$, $M\cdot\E{Z_i\I(Z_i >  n\delta)} < \epsilon/2$. For such $n$ and $j =1,\cdots, M\cdot n$, we have 
\begin{equation}\label{change_variable_Y_Z_bound}
    \pr{\sum_{i=1}^j \big(\E{Z} - Y_i^n\big) > n\epsilon}\leq \pr{\sum_{i=1}^j \big(\E{Y_i^n}- Y_i^n\big)> \frac{n\epsilon}{2}}
\end{equation}
We claim the following two properties about the random variable $\E{Y_i^n} - Y_i^n$:
\begin{itemize}
    \item 
    $\E{Y_i^n} - Y_i^n \leq \E{Y_i^n}\leq \E{Z}$. This is true due to $Y_i^n$ is a non-negative random variable.

    \item 
    Given $\delta$, we can choose large $n$ such that $\E{Z\I(Z > n\delta)} < 1$, hence
    \begin{align*}
        &
        \var{\E{Y_i^n} - Y_i^n }
        \\
        &
        =
        \var{Y_i^n}
        =
        \E{(Y_i^n)^2} - (\E{Y_i^n})^2
        \\
        &
        \leq
        \E{Z^2} - (\E{Z})^2 + (\E{Z})^2 - (\E{Y_i^n})^2
        \\
        &
        =
        \var{Z} + (\E{Z}+\E{Y_i^n})(\E{Z} - \E{Y_i^n})
        \\
        &
        \leq
        \var{Z} + 2\E{Z}\cdot \E{Z\I(Z > n\delta)}
        \\
        &
        \leq
        \var{Z} + 2\E{Z}
    \end{align*}
\end{itemize}
The two properties of $\E{Y_i^n} - Y_i^n $ shown above make us ready to apply the Bernstein's inequality (see Lemma \ref{berstein} in appendix Section \ref{appendix}), and this yields
\begin{align*}
    &
    \pr{\sum_{i=1}^j \big(\E{Y_i^n}- Y_i^n\big)> \frac{n\epsilon}{2}}
    \\
    &
    \leq 
    \exp
    \Big\{
        -\frac{n^2\epsilon^2/4}{2(\sum_i^j\var{\E{Y_i^n}-Y_i^n} + \frac{n\epsilon\E{Z}}{6}}
    \Big\}
    \\
    &
    \leq
    \exp
    \Big\{
        -\frac{n^2\epsilon^2/4}{2(j(\var{Z} + 2\E{Z}) + \frac{n\epsilon\E{Z}}{6}}
    \Big\}.
\end{align*}
Since the exponential term in the above last line increases as $j$ increases, and combining with (\ref{change_variable_Y_Z_bound}), for large $n$,
\begin{equation}\label{numerator-upper}
    \log\underset{j = 1,\cdots, M\cdot n}{\max}\pr{\sum_{i=1}^j \big(\E{Z} - Y_i^n\big) > n\epsilon}\leq -\frac{n^2\epsilon^2/4}{2(M\cdot n(\var{Z} + 2\E{Z}) + \frac{n\epsilon\E{Z}}{6}}
\end{equation}
In view of the ratio in (\ref{inequality:truncate-deviation-2}), the numerator has upper bound (\ref{numerator-upper}), which decrease with speed $-n$. Thus the lemma is proved by running $n\rightarrow\infty$.
\end{proof}

With Lemma \ref{lemma:max-sum-zi-greaterndelta-upperbound} and \ref{lemma:max-sum-zi-smallerndelta-upperbound} in hand, we are ready to prove Lemma~\ref{lemma:checkjnk-asymp-neg}.

\begin{proof}[Proof of Lemma~\ref{lemma:checkjnk-asymp-neg}]
\linksinpf{lemma:checkjnk-asymp-neg}
For any fixed $k\in\mb{N}$, chose $\delta>0$ such that $k\delta < \epsilon/2$. We begin by analyzing $\|\{\bar{H}_n^{\leqslant k}\|> \epsilon\}$ conditioning on the event $\{N(n)\geq k, Z_{P_n^{-1}(k)} \leq n\delta\}$ and its complement. This leads to the following probability upper bound:
\begin{align}
    &
    \pr{
        \| \bar{H}_n^{\leqslant k}\|
        >
        \epsilon
        }
    =  
    \pr{
        \sup_{t\in[0,1]
        }
    \big|
        \sum_{i=1}^{N(nt)}
        \big(
            Z_i\I(P_n(i) > k) - \mu_1
        \big)
    \big|
    >
    n\epsilon
    }
    \nonumber
    \\
    &\leq 
    \pr{
        \underset   
            {j = 1,\cdots, N(n)}    
            {\max}
        \big|
            \sum_{i=1}^{j}
            \big(
                Z_i\I(P_n(i) > k) - \mu_1
            \big)
        \big|
        > 
        n\epsilon, Z_{P_n^{-1}(k)}
        \leq
        n\delta,  N(n)
        \geq 
        k
    }\nonumber
    \\
    &
    +
    \pr{
        \{Z_{R_n^{-1}(k)}
        \leq
        n\delta, N(n)
        \geq
        k\}^c
    }
    \nonumber
    \\
    &
    \leq 
    \pr{
        \underset
            {j = 1,\cdots, N(n)}
            {\max}
        \sum_{i=1}^{j}
        \big(
            Z_i\I(P_n(i) 
            >
            k) - \mu_1
        \big)
        > 
        n\epsilon, Z_{P_n^{-1}(k)}
        \leq
        n\delta,  N(n)
        \geq 
        k
    }
    \label{constraint_prob1}
    \\
    &
    +\pr{
        \underset
            {j = 1,\cdots, N(n)}
            {\max}
        \sum_{i=1}^{j}
        \big(
            \mu_1-Z_i\I(P_n(i) 
            >
            k)
        \big)
        >
        n\epsilon, Z_{P_n^{-1}(k)}
        \leq
        n\delta,  N(n) 
        \geq
        k
    }
    \label{constraint_prob2}
    \\
    &
    +\pr{   
        Z_{P_n^{-1}(k)}
        >
        n\delta
    }
    +\pr{
        N(n)< k
    }.
    \label{removed_events}
\end{align}
We proceed with investigating the terms (\ref{constraint_prob1}) and (\ref{constraint_prob2}) separately. Given that $\{P_n(i) > k\}$ is a subset of $\{Z_i\leq n\delta\}$ under the condition $Z_{P_n^{-1}(k)} \leq n\delta$ and $N(n) \geq k$, the sum $\sum_{i=1}^{j}\big(Z_i\I(Z_i\leq n\delta) - \mu_1\big)$ includes more positive $Z_i$'s than $\sum_{i=1}^{j}\big(Z_i\I(P_n(i) > k) - \mu_1\big)$. Therefore, we can upper bound the probability (\ref{constraint_prob1}) as
\begin{align}
    &
    (\ref{constraint_prob1})
    \leq
    \pr{
        \underset
            {j = 1,\cdots, N(n)}
            {\max}
        \sum_{i=1}^{j}
        \big(
            Z_i\I(Z_i\leq n\delta) - \mu_1
        \big)
        >
        n\epsilon, Z_{P_n^{-1}(k)}
        \leq
        n\delta,  N(n)
        \geq 
        k
    }
    \nonumber
    \\
    &
    \leq
    \pr{
        \underset
            {j = 1,\cdots, N(n)}
            {\max}
        \sum_{i=1}^{j}
        \big(
            Z_i\I(Z_i
            \leq
            n\delta) - \mu_1
        \big)
        > 
        n\epsilon
    }
    \nonumber
    \\
    &
    \leq
    \pr{
        \underset\
            {j = 1,\cdots,n(\lfloor \nu_1\rfloor+1)}
            {\max}
        \sum_{i=1}^{j}
        \big(
            Z_i\I(Z_i
            \leq
            n\delta) - \mu_1
        \big)
        > 
        n\epsilon
    }
    +\pr{
            N(n) 
            > 
            n(\lfloor \nu_1\rfloor+1)
    }.
    \label{upperbound-positive}
\end{align}
To evaluate (\ref{constraint_prob2}), note that $\I(Z_j\leq n\delta)-\I(P_n(j) > k)= 1$ if and only if $Z_j$ is among the $k$-largest value of $Z_i$'s and $Z_j\leq n\delta$, hence $\sum_{i=1}^{j}Z_i\big(\I(Z_i\leq n\delta)-\I(P_n(i) > k)\big)\leq kn\delta$ for any $j$. This leads to the following bound for (\ref{constraint_prob2}):
\begin{align}
    &
    (\ref{constraint_prob2})
    = \mathbb{P}\Big(
        \underset
            {j = 1,\cdots, N(n)}
            {\max}
        \sum_{i=1}^{j}
        \big(
            \mu_1-Z_i\I(Z_i\leq n\delta)
        \big)
        + 
        \sum_{i=1}^{j}Z_i
        \big(
            \I(Z_i
            \leq
            n\delta)-\I(P_n(i) 
            > 
            k)
        \big)
        >
        n\epsilon,
        \nonumber
    \\
    &
    Z_{P_n^{-1}(k)} 
    \leq
    n\delta,  N(n) 
    \geq
    k
    \Big)
    \nonumber
    \\
    &
    \leq
    \mathbb{P}
    \Big(
        \underset
            {j = 1,\cdots, N(n)}
            {\max}
        \sum_{i=1}^{j}
        \big(
            \mu_1-Z_i\I(Z_i
            \leq
            n\delta)
        \big)
        > 
        n(\epsilon - k\delta), Z_{P_n^{-1}(k)}
        \leq
        n\delta,  N(n) 
        \geq 
        k
    \Big)
    \nonumber
    \\
    &
    \leq
    \mathbb{P}
    \Big(
        \underset  
            {j = 1,\cdots, N(n)}
            {\max}
        \sum_{i=1}^{j}
        \big(
            \mu_1-Z_i\I(Z_i\leq n\delta)\big) > \frac{n\epsilon}{2}
        \Big)
        \nonumber
    \\
    &
    \leq
    \pr{
        \underset
            {j = 1,\cdots,n(\lfloor \nu_1\rfloor+1)}
            {\max}
        \sum_{i=1}^{j}\
        \big(   
            \mu_1-Z_i\I(Z_i\leq n\delta)
        \big)
        > 
        \frac{n\epsilon}{2}
    }
    +\pr{
        N(n) 
        > 
        n(\lfloor \nu_1\rfloor+1)
    }
    \label{upperbound-negative}
\end{align}

Having established the upper bounds (\ref{upperbound-positive}) and (\ref{upperbound-negative}) for (\ref{constraint_prob1}) and (\ref{constraint_prob2}), and considering the additional terms in (\ref{removed_events}), we can conclude that 
\begin{align}
    &
    \pr{
        \| \bar{H}_n^{\leqslant k}\|
        >
        \epsilon
    }
    \leq
    2\pr{
        \underset
            {j = 1,\cdots,n(\lfloor \nu_1\rfloor+1)}
            {\max}
        \Big|
            \sum_{i=1}^{j}
            \big(
                Z_i\I(Z_i
                \leq
                n\delta) - \mu_1
            \big)
        \Big|
        >
        \frac{n\epsilon}{2}
    }
    \nonumber
    \\
    &
    +\pr{
        Z_{P_n^{-1}(k)}
        >
        n\delta
    }
    + \pr{
        N(n)< k
    }
    +2\pr{
        N(n) > n(\lfloor \nu_1\rfloor+1)
    }.
    \label{final-upperbound}
\end{align}

For the third and fourth terms in (\ref{final-upperbound}), since $N(n)$ is distributed as the sum of $n$ i.i.d.\ Poisson($\nu_1$) random variables and by Sanov's theorem (see, for example, theorem 6.1.3 in \cite{MR2571413}), those two terms decay as $e^{-n\cdot\text{const}}$. Consequently,
\begin{equation*}
    \limsup_{n\rightarrow\infty}\frac{\log\pr{N(n)< k}}{r(\log n)} = -\infty\text{ and }\limsup_{n\rightarrow\infty}\frac{\log2\pr{N(n) > n(\lfloor \nu_1\rfloor+1)}}{r(\log n)} = -\infty
\end{equation*}
For the second term in (\ref{final-upperbound}), the $k$-th largest $Z_i$'s is distributed as $Q^{-1}_n(\Gamma_k)$. By leveraging the limit result (\ref{limit5}) in appendix, this term satisfies
\begin{equation*}
    \limsup_{n\rightarrow\infty}
    \frac
    {
        \log
        \pr{
            Z_{P_n^{-1}(k)}>n\delta
        }
    }
    {r(\log n)}
    =
    \limsup_{n\rightarrow\infty}
    \frac
        {\log\pr{Q^{\leftarrow}_n(\Gamma_k)>n\delta}}
        {r(\log n)}
    \leq
    \limsup_{n\rightarrow\infty}
    \frac
        {\log\pr{\Gamma_k\leq Q_n(n\delta)}}
        {r(\log n)} 
    = -k
\end{equation*}
For the first term in (\ref{final-upperbound}), we employ Etemadi's inequality (see Lemma \ref{etemadi} in appendix Section \ref{appendix} ) to externalize the maximum from the probability. Then, invoking Lemma \ref{lemma:max-sum-zi-greaterndelta-upperbound} and Lemma \ref{lemma:max-sum-zi-smallerndelta-upperbound} from the appendix, we analyze the term as follows:
\begin{align*}
    &
    \limsup_{n\rightarrow\infty}
    \frac
        {\log2
        \pr{
            \underset
                {j = 1,\cdots,n(\lfloor \nu_1\rfloor+1)}
                {\max}
            \Big|
                \sum_{i=1}^{j}\big(Z_i\I(Z_i\leq n\delta) - \mu_1\big)
            \Big|
            > 
            \frac
                {n\epsilon}{2}}}
                {r(\log n)}
    \\
    &
    \leq
    \limsup_{n\rightarrow\infty}
    \frac
        {\log6
        \underset
            {j = 1,\cdots,n(\lfloor \nu_1\rfloor+1)}
            {\max}
        \pr{
            \Big|\sum_{i=1}^{j}\big(Z_i\I(Z_i\leq n\delta) - \mu_1\big)\Big| 
            \geq 
            \frac{n\epsilon}{6}
        }
        }
        {r(\log n)}
    \\
    &
    \leq
    0+\max\Big\{
        \limsup_{n\rightarrow\infty}
        \frac
            { \log\underset{j = 1,\cdots,n(\lfloor \nu_1\rfloor+1)}
            {\max}
            \pr{\sum_{i=1}^{j}
            \big(
                Z_i\I(Z_i\leq n\delta) - \mu_1
            \big)
            > 
            \frac{n\epsilon}{7}}}{r(\log n)},
    \\
    &
    \limsup_{n\rightarrow\infty}
    \frac
        {\log\underset{j = 1,\cdots,n(\lfloor \nu_1\rfloor+1)}
        {\max}
    \pr{
        \sum_{i=1}^{j}
        \big(
            \mu_1 - Z_i\I(Z_i\leq n\delta)
        \big)
        > 
        \frac{n\epsilon}{7}}}{r(\log n)}
        \Big\}
    \leq
    -\frac{\epsilon}{14\delta}.
\end{align*}
Returning to (\ref{final-upperbound}), for each fixed $k$, we can choose $\delta$ small enough such that $-\epsilon/14\delta < -k$. Therefore, we obtain
\begin{equation*}
    \limsup_{n\rightarrow\infty}\frac{\log \pr{\| \bar{H}_n^{\leqslant k}\|> \epsilon}}{r(\log n)} \leq \max\{-\frac{\epsilon}{14\delta}, -k, -\infty,-\infty\} = -k
\end{equation*}
The conclusion of the lemma follows by considering $k\rightarrow \infty$.
\end{proof}

Let \newnota{r-euclidean-kdim-decrease}{$\mb{R}_+^{k\downarrow}$}\defnota{r-euclidean-kdim-decrease}{$\delequal\{(x_1,\cdots, x_k)\in\mb{R}^k: x_1\geq x_2\geq\cdots\geq x_k\geq 0\}$}.
Consider the sequence of measures on $\mb{R}_+^{k\downarrow}$ defined below: for $n\in\mb{N}$,
\begin{equation}\label{def:measure-largestkjump}
    \newnota{mu-k-largest-jump-measure}{\check\mu^k_n(\cdot)}\defnota{mu-k-largest-jump-measure}{ \delequal \pr{\Big(\frac{Q_n^\leftarrow (\Gamma_1)}{n},\cdots, \frac{Q_n^\leftarrow(\Gamma_k)}{n}\Big)\in \cdot}}
\end{equation}
In view of (\ref{hatjnk}), note that $\check\mu^k_n(\cdot)$ is the distribution of the jump sizes of $\hat{J}^{\leqslant k}_n$.

Proposition \ref{proposition:k-big-jump-ldp-lowerbound} and \ref{proposition:k-big-jump-ldp-upperbound} below will be used in Theorem \ref{theorem:k-big-jump-ldp}, and they prove the sequence of measure $\{\check\mu_n\}_{n\geq 1}$ (defined in (\ref{def:measure-largestkjump}) and reviewed below) satisfies the upper and lower bound inequality for extended large deviation, respectively.
We start with the following observation.

Define \newnota{h-kdim-leqi}{$H^k_{\leqslant i}$}\defnota{h-kdim-leqi}{$\delequal \{x\in \mb{R}_+^{k\downarrow}, x_{i+1} = 0\}$} for $i\in\{0, 1,2,\cdots, k-1\}$ and $H_{\sleq k}^k = \R_+^{k\downarrow}$.
Note that each $H^k_{\leqslant i}$ is closed.

\subsection{Proofs for Section \ref{section:randomwalk-ldp-J1}}\label{section:randomwalk-proof}
Instead proving the above theorem directly, we consider the process $\bar{S_n}$ defined below, whose jumps, except for the last one, are uniformly distributed on $[0,1]$:
\begin{equation}\label{notation-snbar}
    \newnota{sn-bar}{\bar{S}_n}\defnota{sn-bar}{\delequal\frac{1}{n}\sum_{i=1}^{n-1} \big(Z_i - \E{Z}\big)\I_{[U_i, 1]} + \frac{1}{n} (Z_n-\E{Z})\I_{\{1\}}}
\end{equation}
Here $Z$ is a generic random variable of the same distribution as $Z_1$. $U_1,\cdots, U_{n-1}$ are i.i.d.\ random variable uniform on $[0,1]$.

Similar to (\ref{hatjnk}) and (\ref{tildejnk}), where jumps in a L\'evy process is presented in the order of their size,  we can list the jumps in $\bar{S_n}$  on $[0,1)$ as $\{\tilde{Q}^\leftarrow(V_{(i)}), i = 1,\cdots, n-1\}$. Here \newnota{q-one-minus-cdf}{$\tilde{Q}(x)$}\defnota{q-one-minus-cdf}{$\delequal \pr{Z \geq x}$} and \newnota{q-tilde-inverse}{$\tilde{Q}^\leftarrow(x)$}\defnota{q-tilde-inverse}{$= \inf\{s\geq 0: \tilde{Q}(s) < x\}$}.
\newnota{v-uniform-0-1-v}{$V_1,V_2,\cdots, V_{n-1}$} are \defnota{v-uniform-0-1-v}{i.i.d.\ random variables uniform on $[0,1]$}, with
\newnota{ordered-v}{$V_{(1)}, V_{(2)}, \cdots, V_{(n-1)}$} being their order statistics in the descending order\mdefnota{ordered-v}{(descending) order statistics of $V_1,\cdots, V_{n-1}$}. 
Taking this ordered jump representation into consideration,  $\bar{S_n}$ has the same distribution with $\newnota{tilde-sn}{\tilde{S}_n}\defnota{tilde-sn}{\delequal \tilde{J}^k_n + \tilde{H}^k_n}$ where for $t\in[0,1]$:
\begin{equation}\label{def:tilde-jnk}
    \newnota{j-tilde-jnk}{{\tilde{J}^k_n(t)}}\defnota{j-tilde-jnk}{\delequal \frac{1}{n}\sum_{i=1}^k \tilde{Q}^\leftarrow(V_{(i)}) \I_{[U_i,1]}(t) + \frac{1}{n} Z_n\I_{\{1\}}(t)}
\end{equation}

and 
\begin{equation}\label{def:tilde-hnk}
    \newnota{hn-tilde}{{\tilde{H}^k_n(t)}}\defnota{hn-tilde}{\delequal \frac{1}{n}\sum_{i=k+1}^{n-1} \tilde{Q}^\leftarrow(V_{(i)})\I_{[U_i,1]}(t) -\frac{1}{n}\E{Z}\sum_{i=1}^{n-1}\I_{[U_i,1]}(t) -\frac{1}{n} \E{Z}\I_{\{1\}}(t)}
\end{equation}

We have the following three lemmas as intermediate results:

\begin{lemma}\label{lemma:k-big-jump-rw-ldp}
\linksinthm{lemma:k-big-jump-rw-ldp}
For any $k\in\mb{N}$, recall that $\mb{R}_+^{k\downarrow}=\{(x_1,\cdots, x_k)\in\mb{R}^k: x_1\geq x_2\geq\cdots\geq x_k\geq 0\}$.
The sequence of measures $\{\tilde{\mu}_n\}_{n\geq 1}$ on $\mb{R}_+^{k\downarrow}$ defined by
\begin{equation}\label{jump-random-walk-measure}
    \newnota{m-jump-random-walk-measure}{\tilde{\mu}_n(\cdot)}\defnota{m-jump-random-walk-measure}{\delequal \pr{\big(\frac{\tilde{Q}^\leftarrow(V_{(1)})}{n}, \cdots,\frac{\tilde{Q}^\leftarrow(V_{(k)})}{n} \big)\in \cdot}}
\end{equation}
satisfies the extended LDP with the rate  $\tilde{I}^k:\mb{R}_+^{k\downarrow}\rightarrow\mb{R}$ given by:
\begin{equation}
    \newnota{i-rate-func-jump-random-walk}{\tilde{I}^k (\bm x)}\defnota{i-rate-func-jump-random-walk}{\delequal \sum_{i=1}^k \I\{x_i\neq 0\}\text{ with } \bm = (x_1,\cdots, x_k)\in\mb{R}_+^{k\downarrow}}
\end{equation}
and speed $r(\log n)$.
\end{lemma}

\begin{lemma}\label{lemma:ex-ldp-tilde-jnk}
\linksinthm{lemma:ex-ldp-tilde-jnk}
The sequence $\{\tilde{J}_n^k\}_{n\geq 1}$ defined in (\ref{def:tilde-jnk}) satisfies the extended LDP  on $(\mb{D}, d_{J_1}$) with rate function $\tilde{I}_k:\mb{D}\rightarrow\mb{R}$ given by:
\begin{equation}\label{rate-tilde-ik}
\newnota{i-rate-func-k-jump-rw}{\tilde{I}_k(\xi)}\defnota{i-rate-func-k-jump-rw}{\delequal 
\begin{cases}
\sum_{t\in(0,1]}\I\{\xi(t)\neq \xi(t-)\} & \text{ if }\xi\in\tilde{\mb{D}}_{\leqslant k}\\
\infty & \text{ otherwise}
\end{cases}}
\end{equation} 
and speed $r(\log n)$.
\end{lemma}

\begin{lemma}\label{lemma: extended-ldp-tildesn}
\linksinthm{lemma: extended-ldp-tildesn}
The sequence of processes $\{\bar{S}_n\}_{n\geq 1}$ satisfies the extended large deviation on $(\mb{D}, d_{J_1})$ with the rate function given by (\ref{def:rate-function-randomwalk-J1}) and speed  $r(\log n)$.
\end{lemma}

Lemma \ref{lemma:k-big-jump-rw-ldp}, \ref{lemma:ex-ldp-tilde-jnk} and \ref{lemma: extended-ldp-tildesn} above are in parallel of Proposition \ref{theorem:k-big-jump-ldp}, Proposition~\ref{theorem:ex-ldp-kjumpprocess} and Theorem~\ref{theorem:ex-ldp-xnbar} in section \ref{section:extended-LDP-levy-J_1}, with comparable proof framework but nuances in computation. We will make the proof of Lemma \ref{lemma:k-big-jump-rw-ldp}, \ref{lemma:ex-ldp-tilde-jnk} and \ref{lemma: extended-ldp-tildesn} succinct and put them in appendix section \ref{section:randomwalk-proof}.

With the conclusion in Lemma \ref{lemma: extended-ldp-tildesn}, we are ready to prove Theorem \ref{theorem: extended-ldp-randomwalk}:

\begin{proof}[Proof of Theorem \ref{theorem: extended-ldp-randomwalk}]
\linksinpf{theorem: extended-ldp-randomwalk}
Recall the definition of $\bar{W}_n$ in (\ref{definition:scaled-centered-randomwalk}).
Incorporating $\{\tilde{Q}^\leftarrow(V_{(i)}), i = 1,\cdots, n-1\}$ enables a distributional representation of $\bar{W}_n$:
\begin{equation*}
    \bar{W}_n \stackrel{\ms{D}}{=} \frac{1}{n}\sum_{i=1}^{n-1} \big(\tilde{Q}^\leftarrow(V_{(i)}) - \E{Z}\big)\I_{[\frac{R_i}{n}, 1]} + \frac{1}{n} (Z_n-\E{Z})\I_{\{1\}}
\end{equation*}
and this representation creates a coupling between $\bar{W}_n$ and $\tilde{S}_n$. In the above representation, $(R_1,\cdots,R_{n-1})$ is the random permutation of $\{1,2,\cdots, n-1\}$ indicating the increasing rank of $(U_1,\cdots, U_{n-1})$ appeared in (\ref{def:tilde-jnk}) and (\ref{def:tilde-hnk}), i.e $U_i$ is the $R_i$-th smallest among $\{U_1,\cdots, U_{n-1}\}$.
This makes the $i$-th largest jump size $\tilde{Q}^\leftarrow(V_{(i)})$ being the $R_i$-th increment of both $\tilde{S}_n$ an d $\bar{W}_n$ w.r.t time. 

According to the definition of $J_1$ metric, We have 
\begin{align*}
    &
    \pr{
        d_{J_1}(\bar{W}_n,\tilde{S}_n)
        >
        \epsilon
    } 
    \leq
    \pr{
        \sup_{i\in[n-1]}\big|\frac{i}{n}- U_{(i)}\big|
        >
        \epsilon
    }
    \\
    &
    \leq
    \pr{
        \sqrt{n}
        \sup_{x\in[0,1]}\big|\frac{1}{n}\sum_{i=1}^n\I(U_i\leq x) - x\big|
        >
        \sqrt{n}\epsilon
    }
    \leq 
    2e^{-2\epsilon^2n}
\end{align*}
where the last inequality above is due to  Corollary 1 of \cite{MR1062069}. This further implies 
\begin{equation*}
    \limsup_{n\rightarrow\infty} \frac{\log\pr{d_{J_1}(\bar{W}_n,\tilde{S}_n)}}{r(\log n)}\leq  -\lim_{n\rightarrow\infty} \frac{2\epsilon^2n}{r(\log n)} = -\infty
\end{equation*}
Hence by the corollary \ref{lemma:approximate2}, $\{\tilde{S}_n\}_{n\geq 1}$ and $\{\bar{W}_n\}_{n\geq 1}$ satisfies the an extended LDP with the same rate function and speed sequence. Note that the extended LDP of $\{\tilde{S}_n\}_{n\geq 1}$ is confirmed by Lemma \ref{lemma: extended-ldp-tildesn}. This finishes the proof. 
\end{proof}

\begin{proof}[Proof of Lemma \ref{lemma:k-big-jump-rw-ldp}]
\linksinpf{lemma:k-big-jump-rw-ldp}
It is clear that $\tilde{I}^k$ defined in (\ref{rate-tilde-ik}) is lower semicontinuous on $\mb{R}_+^{k\downarrow}$. 
What's left is to verify the lower and upper bound inequalities required in an extended LDP, i.e for any open  set $G$
\begin{equation}\label{equ:tildemun-ldp-lowerbound}
     \liminf_{n\rightarrow\infty} \frac{\log\tilde{\mu}_n( G)}{r(\log n)}\geq -\inf_{x\in G} \tilde{I}^k(x)
\end{equation}
and for any closed set $F$
\begin{equation}\label{equ:tildemun-ldp-upperbound}
     \limsup_{n\rightarrow\infty} \frac{\log\tilde{\mu}_n( F)}{r(\log n)}\leq -\lim_{\epsilon\downarrow 0}\inf_{x\in F^\epsilon} \tilde{I}^k(x)
\end{equation}

Let's first show (\ref{equ:tildemun-ldp-lowerbound}). Fix a $\hat{x}\in G$, then $\hat{x}$ has  $\tilde{I}^k(\hat{x})$ nonzero entries. 
Using a similar argument for (\ref{midstep1}) in Proposition \ref{proposition:k-big-jump-ldp-lowerbound}, for $m\in\mb{N}$, there is $f_m:\mb{R}_+^{k\downarrow}\rightarrow [-m,0]$ such that $f_m(\hat{x}) = 0$ and $f_m(y) = -m$ for $y\in G^c$, and we obtain 
\begin{equation}\label{midstep6}
    \liminf_{n\rightarrow\infty}\frac{\log\int_{\mb{R}_+^{k\downarrow}} e^{-\log\tilde{Q}(n)f_m(x)} \tilde{\mu}_n(dx)}{-\log\tilde{Q}(n)}\leq \max\{\liminf_{n\rightarrow\infty}\frac{\log\tilde{\mu}_n(G)}{-\log\tilde{Q}(n)} , -m\}
\end{equation}

We try to get a lower bound for the LHS of (\ref{midstep6}). For this purpose, define $\tilde{D}_\delta$ as a neighbour of $\hat{x}$ taking the following form:
\begin{equation*}
        \newnota{d-delta-set-rw}{\tilde{D}_\delta}\defnota{d-delta-set-rw}{\delequal \prod_{i=1}^{\tilde{I}^k(\hat{\bm x})}[\hat{x}_i-\delta, \hat{x}_i+\delta]\times [0,\delta]^{(k-\tilde{I}^k(\hat{\bm x}))}}
\end{equation*}
Fix $\forall\epsilon > 0$, we could choose a $\delta$ small enough (only dependent on $m$ and $\epsilon$) such that
\begin{enumerate}
    \item 
    If two nonzero entries $\hat{x}_i\neq \hat{x}_j$, then $[\hat{x}_i-\delta, \hat{x}_i+\delta]\cap [\hat{x}_j-\delta, \hat{x}_j+\delta] = \emptyset$.
    
    \item
    If $\hat{x}_i\neq 0$, then $[\hat{x}_i-\delta, \hat{x}_i+\delta]\cap [0,\delta]=\emptyset$.
    
    \item
    $\tilde{D}_\delta\subset G$ (This is possible since $G$ is open)
    
    \item
    $f_m(y) > f_m(\hat{x}) - \epsilon = 0 -\epsilon = -\epsilon$ for $y\in \tilde{D}_\delta$. (This is possible due to the continuity of $f_m$).
\end{enumerate}
By the properties 3 and 4 of $\tilde{D}_\delta$ above, we can use a similar derivation of (\ref{bigequation1}) in Proposition \ref{proposition:k-big-jump-ldp-lowerbound} to get :
\begin{equation}\label{midstep5}
    \liminf_{n\rightarrow\infty}\frac{\log\int_{\mb{R}_+^{k\downarrow}} e^{-\log\tilde{Q}(n)f_m(x)} \tilde{\mu}_n(dx)}{-\log\tilde{Q}(n)}\geq \liminf_{n\rightarrow\infty}\frac{\log\int_{\tilde{D}_\delta} e^{\tilde{Q}(n)\epsilon} \tilde{\mu}_n(dx)}{-\log\tilde{Q}(n)}
    \geq -\epsilon + \liminf_{n\rightarrow\infty}\frac{\log\tilde{\mu}_n(\tilde{D}_\delta)}{-\log\tilde{Q}(n)}
\end{equation}
According to $\hat{D}_\delta$'s form along with its properties 1 and 2, $\tilde{\mu}_n(\tilde{D}_\delta)$ in (\ref{midstep5}) satisfies:
\begin{align*}
    \tilde{\mu}_n(\tilde{D}_\delta)
    \geq
    &
    \mb{P}\Big(
        \{
        \frac
            {\tilde{Q}^\leftarrow(V_{(1)})}{n}
            \in
            [\hat{x}_1-\delta, \hat{x}_1+\delta),\cdots,
        \frac
            {\tilde{Q}^\leftarrow(V_{(\tilde{I}^k(\hat{\bm x}))})}
            {n}
            \in [\hat{x}_{\tilde{I}^k(\hat{\bm x})}-\delta, \hat{x}_{\tilde{I}^k(\hat{\bm x})}+\delta),
        \\
        &
         \frac
            {\tilde{Q}^\leftarrow(V_{(j)})}{n}
        \in
        [0, \delta)
        \text{ for } j= \tilde{I}^k(\hat{\bm x}) + 1, \cdots, n-1
        \}
    \Big)
    \\
    =& 
    {{n-1}\choose{k}}
    \cdot
    \prod_{i=1}^{\tilde{I}(\hat{\bm x})}
    \big(\tilde{Q}(n(\hat{x}_i - \delta)- \tilde{Q}(n(\hat{x}_i + \delta)\big)
    \cdot
    \big(1- \tilde{Q}(n\delta)\big)^{n-1-\tilde{I}^k(\hat{\bm x})},
\end{align*}
and this implies
\begin{align}
    \liminf_{n\rightarrow\infty}
    \frac
        {\log\tilde{\mu}_n(\tilde{D}_\delta)}
        {-\log\tilde{Q}(n)}
    \geq
    &
    \underbrace
        {\lim_{n\rightarrow\infty}
         \frac
            {\log{{n-1}\choose{k}}}
            {-\log\tilde{Q}(n)}
        }_{(\text{I})} 
    + \underbrace
        {\sum_{i=1}^{\tilde{I}^k(\hat{\bm x})}
        \lim_{n\rightarrow\infty}
        \frac
            {\log\big(\tilde{Q}(n(\hat{x}_i - \delta)- \tilde{Q}(n(\hat{x}_i + \delta)\big)}
            {-\log\tilde{Q}(n)}
        }_{(\text{II})}
    +
    \\
    &
    \underbrace{                
        \lim_{n\rightarrow\infty}
        \frac
            {(n-1-\tilde{I}^k(\hat{x}))\cdot\log\big(1- \tilde{Q}(n\delta)\big)}
            {-\log\tilde{Q}(n)}
    }_{(\text{III})}
    \label{midstep11}
    \\
    \geq
    &
    0-\tilde{I}^k(\hat{\bm x})+0
    =
    -\tilde{I}^k(\hat{\bm x}).
    \nonumber
\end{align}
Note that in (\ref{midstep11}), (I)$=0$ due to $-\log\tilde{Q}(n) = r(\log n) (1+o(1))$ and $\log{{n-1}\choose{k}}\sim \log(n-1)^k \sim k\log n$;
(II)$=-\tilde{I}^k(\hat{x})$ and (III)$=0$ are the limit results (\ref{limit7}) and (\ref{limit8}), respectively.

Now, we can combine the inequalities (\ref{midstep6}), (\ref{midstep5}) and (\ref{midstep11}), and since $\epsilon$ is arbitrarily small, we obtain
\begin{equation*}
    \max\{\liminf_{n\rightarrow\infty}\frac{\log\tilde{\mu}_n(G)}{-\log\tilde{Q}(n)} , -m\}\geq\liminf_{n\rightarrow\infty}\frac{\log\int_{\mb{R}_+^{k\downarrow}} e^{-\log\tilde{Q}(n)f_m(x)} \tilde{\mu}_n(dx)}{-\log\tilde{Q}(n)}\geq -\tilde{I}^k(\hat{x}),
\end{equation*}
The extended LDP lower bound (\ref{equ:tildemun-ldp-lowerbound}) is achieved by letting $m\rightarrow\infty$ and taking infimum over $\hat{x}\in G$.

Now we turn to the upper bound (\ref{equ:tildemun-ldp-upperbound}).
Let $H^k_{\leqslant i} =\{x\in \mb{R}_+^{k\downarrow}, x_{i+1} = 0\}$ and define
\begin{equation}\label{def:index-set2}
    \ms{I}\delequal\{i\in\{0,1,2,\cdots,k-1\}: d(F,  H^k_{\leqslant i}) >0\}  
\end{equation}
We borrow the same argument in Proposition \ref{proposition:k-big-jump-ldp-upperbound}: if the set (\ref{def:index-set2}) is empty, the upper bound (\ref{equ:tildemun-ldp-upperbound}) becomes trivial. Otherwise, $\lim_{\epsilon\downarrow 0}\inf_{x\in F^\epsilon} \check{I}_k(x) = i+1$ with $i = \max\ \ms{I}$. 
Since $d(F,H^k_{\leqslant i})>0$,  there is $r>0$ such that for any $x\in F$, $x_{i+1} > r$, this implies:
\begin{align*}
    \limsup_{n\rightarrow\infty}
    \frac
        {\log\check{\mu}_n(F)}
        {-\log\tilde{Q}(n)}
    &
    =
    \limsup_{n\rightarrow\infty}
    \frac
        {
            \log
            \pr{
            \Big(
            \frac
                {\tilde{Q}^\leftarrow(V_{(1)})}{n},
                \frac{\tilde{Q}^\leftarrow(V_{(2)})}{n},
                \cdots \frac{\tilde{Q}^\leftarrow(V_{(k)})}{n}
            \Big)
            \in F
            }
        }
        {-\log\tilde{Q}(n)}
    \\
    &
    \leq
    \limsup_{n\rightarrow\infty}
    \frac
        {\log\pr{\{ \tilde{Q}^\leftarrow(V_{(i+1)}) \geq nr\}}}
        {-\log\tilde{Q}(n)}
    =
    \limsup_{n\rightarrow\infty}
    \frac
        {\log\pr{\{ V_{(i+1)}\leq\tilde{Q}(nr)\}}}
        {-\log\tilde{Q}(n)}
    \\    
    &
    \leq
    -(i+1)
    =
    -\lim_{\epsilon\downarrow 0}\inf_{x\in F^\epsilon} \tilde{I}^k(x).
\end{align*}
The first equality in the last line is due to the limit result (\ref{limit9}). This establishes the upper bound in (\ref{equ:tildemun-ldp-upperbound}).
\end{proof}

\begin{proof}[Proof of Lemma \ref{lemma:ex-ldp-tilde-jnk}]
\linksinpf{lemma:ex-ldp-tilde-jnk}

We only prove the extended LDP on $(\tilde{\mb{D}}_{\leqslant k}, d_{J_1})$, as that can be turned to an extended LDP on $(\mb{D}, d_{J_1}) $ with the help of Lemma \ref{E-LDP-on-subspaces-full-measure}. The conditions to apply Lemma \ref{E-LDP-on-subspaces-full-measure} are met, as $\pr{\tilde{J}^k_n\in \tilde{\mb{D}}_{\leqslant k}} = 1$ and $\tilde{\mb{D}}_{\leqslant k}$ is  closed.  

The rate function $\tilde{I}_k$ is lower-semicontinuous, as its level set $\Psi_{\tilde{I}_k}(c)=\tilde{\mb{D}}_{\leqslant\min\{\lfloor c\rfloor, k\}}$ is closed. 
It remains to show the extened LDP's lower bound: for any $G$ as a subset $\tilde{\mb{D}}_{\leqslant k}$ and open 
\begin{equation}\label{equ:tilde-jkn-ldp-lowerbound}
    \liminf_{n\rightarrow\infty} \frac{\log\pr{\tilde{J}_n^k\in G}}{r(\log n)}\geq -\inf_{\xi\in G} \tilde{I}_k(\xi)
\end{equation}
and extended LDP's upper bound: for any  $F$ a subset $\tilde{\mb{D}}_{\leqslant k}$ and closed 
\begin{equation}\label{equ:tilde-jkn-ldp-upperbound}
    \limsup_{n\rightarrow\infty} \frac{\log\pr{\tilde{J}_n^k\in F}}{r(\log n)}\leq -\lim_{\epsilon\downarrow 0}\inf_{\xi\in F^\epsilon} \tilde{I}_k(\xi)
\end{equation}
(Here $F^\epsilon$ is the set $\{\eta\in\tilde{\mb{D}}_{\leqslant k}: d_{J_1}(\eta, F)\leq \epsilon\}$).

For the lower bound (\ref{equ:tilde-jkn-ldp-lowerbound}), we discuss the following two cases:

\textbf{Case 1}: If $\xi$ has $j$ jumps for some $j\geq k$ and all those jumps fall on $(0,1)$.
In such case, $\xi = \sum_{i=1}^j x_i\I_{[u_i,1]}$ with $x_i$'s being the jump sizes of non-increasing order and $u_i\in(0,1)$ being the corresponding jump times. 
We construct a neighbourhood of $\xi$ of the following form
\begin{align}
    C_\delta
    \delequal
    \Big\{
        \eta = \sum_{i=1}^k y_i\I_{[w_i, 1]} + y_{k+1}\I_{\{1\}}\in\mb{D}:
        &
        \bm y
        =
        (y_1,\cdots,y_{k+1})\in  \big((\prod_{i=1}^{k}Y_i)\cap\mb{R}_+^{k\downarrow}\big)\times Y_{k+1},
        \nonumber
        \\
        &
        \bm w 
        =
        (w_1,\cdots,w_k)\in \prod_{i=1}^k W_i, 
    \Big\}
    \label{def:the-set-C}
\end{align}
with the sets $Y_i$'s and $W_i$'s defined as below
\begin{equation*}
    Y_i = 
    \begin{cases}
        (x_i-\delta, x_i+\delta) & i=1,\cdots, j\\
        [0,\delta) & i =j+1, \cdots, k+1
    \end{cases}
    ,\ \ \ \ 
    W_i = 
    \begin{cases}
        (u_i-\delta, u_i+\delta) & i=1,\cdots, j\\
        (w^*-\delta, w^*+\delta) & i =j+1, \cdots, k
    \end{cases}
\end{equation*}
In the above, $w^*$ is a fixed time that does not belong to $\{u_1,\cdots, u_j\}$. 
It is not hard to see for any $\eta\in C_\delta$  $d_{J_1}(\xi, \eta) < (k+1)\delta$ is satisfied, hence we can make $C_\delta\in G$ by choosing $\delta$ small.

Based on $C_\delta\in G$, we have
\begin{align*}
    &
    \pr{\tilde{J}_n^k \in G}
    \geq
    \pr{\tilde{J}_n^k\in C_\delta}
    \\
    &
    =
    \pr{
        \frac{1}{n}
        \sum_{i=1}^k \tilde{Q}^\leftarrow(V_{(i)}) \I_{[U_i,1]} + \frac{1}{n} Z_n\I_{\{1\}}\in C_\delta
    }
    \\ 
    &
    =
    \pr{
        \Big(
        \frac
            {\tilde{Q}^\leftarrow(V_{(1)})} {n}
        ,\cdots,
        \frac
            {\tilde{Q}^\leftarrow(V_{(k)})} {n}
        \Big)
        \in (\prod_{i=1}^{k}Y_i)\cap\mb{R}_+^{k\downarrow}}
        \pr{\frac{Z_n}{n}\in Y_{k+1}}\pr{ (U_1, \cdots, U_k)
        \in B
    }
    \\ 
    &
    =
    \tilde{\mu}_n((\prod_{i=1}^{k}Y_i)\cap\mb{R}_+^{k\downarrow})
    \cdot
    \pr{Z_n<n\delta}
    \cdot
    \text{Const}.
\end{align*}
Therefore,
\begin{align}
    &
    \liminf_{n\rightarrow\infty} 
    \frac
        {\log\pr{\tilde{J}_n^k\in G}}
        {r(\log n)}
    \nonumber
    \\
    &
    \geq
    \liminf_{n\rightarrow\infty} 
    \frac
        {\log\tilde{\mu}_n((\prod_{i=1}^{k}Y_i)\cap\mb{R}_+^{k\downarrow}) + \log\pr{Z_n<n\delta}}
        {r(\log n)}
    \nonumber
    \\
    &
    =
    \liminf_{n\rightarrow\infty} 
    \frac
        {\log\tilde{\mu}_n\big((\prod_{i=1}^{k}Y_i)\cap\mb{R}_+^{k\downarrow}\big)}
        {r(\log n)}
    +
    \liminf_{n\rightarrow\infty} 
    \frac
        {\log\big(1- \tilde{Q}(n\delta)\big)}
        {r(\log n)}
    \label{midstep13}
    \\
    &
    \geq
    - \inf_{
        \bm x\in (\prod_{i=1}^{k}Y_i)\cap\mb{R}_+^{k\downarrow}
        }
    \tilde{I}^k\big(\bm x\big) + 0
    \geq
    -\tilde{I}^k\big((x_1,\cdots, x_j,0,\cdots, 0)\big)
    =
    -j
    =
    -\tilde{I}_k(\xi).
    \label{midstep14}
\end{align}
Note that to lower bound two limit terms in from (\ref{midstep13}), we use the conclusion from Lemma \ref{lemma:k-big-jump-rw-ldp} such that $\{\tilde{\mu}_n\}_{n\geq 1}$ satisfying the extended LDP, and the limit results (\ref{limit8}). The second inequality in to (\ref{midstep14}) holds due to $(x_1,\cdots, x_j,0,\cdots, 0)\in (\prod_{i=1}^{k}Y_i)\cap\mb{R}_+^{k\downarrow}$.

\textbf{Case 2}:If $\xi$ has $j$ jumps for some $j\geq k$ and one of those jump occur at $t=1$.
In such case, $\xi = \sum_{i=1}^{j-1} x_i\I_{[u_i,1]} + x_j\I_{\{1\}}$ with $x_1, \cdots, x_{j-1}$ being the jump sizes of non-increasing order. 
We construct the set $C_\delta$ as (\ref{def:the-set-C}), but modify $Y_i$'s and $W_i$'s definitions as below
\begin{equation*}
    Y_i = 
    \begin{cases}
        (x_i-\delta, x_i+\delta) & i=1,\cdots, j-1\\
        [0,\delta) & i=j,\cdots, k\\
        (x_j-\delta, x_j+\delta) & i = k+1
    \end{cases}
    ,\ \ \ \ 
    B_i = 
    \begin{cases}
        (u_i-\delta, u_i+\delta) & i=1,\cdots, j-1\\
        (\frac{1}{2}-\delta, \frac{1}{2}+\delta) & i=j,\cdots, k
    \end{cases}
\end{equation*}
Similar to the case 1 above, $C_\delta$ in this case is a neighborhood of $\xi$ and we can choose $\delta$ small enough to make it as a subset of $G$.
By the extended LDP of $\{\tilde{\mu}_n\}_{n\geq 1}$ and $(x_1,\cdots, x_{j-1},0,\cdots, 0)\in (\prod_{i=1}^{k}Y_i)\cap\mb{R}_+^{k\downarrow}$, we can conclude
\begin{align}
    &
    \liminf_{n\rightarrow\infty} 
    \frac
        {\log\pr{\tilde{J}_n^k\in G}}{r(\log n)}
    \geq
    \liminf_{n\rightarrow\infty} 
    \frac
        {\log\pr{\tilde{J}_n^k\in C_\delta}}
        {r(\log n)}
    \nonumber
    \\
    &
    =
    \liminf_{n\rightarrow\infty} 
    \frac
        {\log\tilde{\mu}_n
         \big(
            (\prod_{i=1}^{k}Y_i)\cap\mb{R}_+^{k\downarrow}
         \big)
         +
         \log
         \pr{\frac{Z_n}{n}\in Y_{k+1}}
         +
         \log\pr{(U_1, \cdots, U_k)\in B}
        }
        {r(\log n)}
    \nonumber
    \\
    \geq
    &
    \liminf_{n\rightarrow\infty}
    \frac
        {\log\tilde{\mu}_n(A\cap\mb{R}_+^{k\downarrow})
        +
        \log\pr{n(x_j - \delta) <Z_n < n(x_j + \delta)}}
        {r(\log n)}
    \nonumber
    \\
    &
    \geq
    \liminf_{n\rightarrow\infty}
    \frac
        {\log\tilde{\mu}_n(A\cap\mb{R}_+^{k\downarrow})}
        {r(\log n)}
        +
    \liminf_{n\rightarrow\infty} 
    \frac
        {\log\big(\tilde{Q}(n(x_j - \frac{\delta}{2})) - \tilde{Q}(n(x_j + \frac{\delta}{2}))\big)}
        {r(\log n)}
    \label{midstep12}
    \\
    &
    \geq
    - \inf_{
        \bm x\in (\prod_{i=1}^{k}Y_i)\cap\mb{R}_+^{k\downarrow}
    }
    \tilde{I}^k\big(\bm x\big) - 1
    \geq
    -j
    \geq
    - \tilde{I}^k\big((x_1,\cdots,x_{j-1},0,\cdots, 0)\big) - 1 
    =
    -\tilde{I}_k(\xi).
    \nonumber
\end{align}
Note that the second limit term in (\ref{midstep12}) equals -1 is due to the limit result (\ref{limit7}). Combining the two above cases yield
\begin{equation*}
    \liminf_{n\rightarrow\infty} \frac{\log\pr{\tilde{J}_n^k\in G}}{r(\log n)}\geq -\tilde{I}_k(\xi)
\end{equation*}
and the lower bound (\ref{equ:tilde-jkn-ldp-lowerbound}) is established by taking supremum over all $\xi\in G$.

Now we turn to prove the upper bound (\ref{equ:tilde-jkn-ldp-upperbound}). If the closed set $F$ contains the zero function, then $\lim_{\epsilon\downarrow 0}\inf_{\xi\in F^\epsilon} \tilde{I}_k(\xi) = 0$, and the upper bound (\ref{equ:tilde-jkn-ldp-upperbound}) is trivial.
Hence we consider $F$ that does not contain the zero function and define
\begin{equation}\label{jfunction2}
        i^* =\max\{j\in \{0, 1, 2, \cdots k-1\},: d_{J_1}(F, \tilde{D}_{\leqslant j}) > 0\}+1
\end{equation}
Based on this definition, $F$ is bounded away from $\tilde{\mb{D}}_{\leqslant i^*-1}$, which implies
\begin{equation}\label{rw-tilde-jnk mid-equation}
     \lim_{\epsilon\downarrow 0} \inf_{\xi\in F^\epsilon} \tilde{I}_k(\xi) = i^*
\end{equation}
Also, paths in $F$ has at least $i^*$ jumps, and due to $d_{J_1}(F, \tilde{D}_{\leqslant i^*-1}) > 0$, we can find some $r > 0$ such that paths in $F$ should have their $i^*$-th largest jump size no smaller than $r$. 

Clearly, $F$ can be represented by $F_1\cup F_2 $, where 
\begin{align*}
    &
    F_1
    =
    \{\xi\in F, \xi(1)-\xi(1-)\text{ does not belong to }\xi's\ i^*\text{ largest jumps sizes}\}
    \\
    &
    F_2
    =
    \{\xi\in F, \xi(1)-\xi(1-)\text{ belongs to }\xi's\ i^*\text{ largest jumps sizes}\}.
\end{align*}
Therefore, 
\begin{equation}\label{midstep7}
\limsup_{n\rightarrow\infty}\frac{\log\pr{\tilde{J}^k_n\in F}}{ r(\log n)} =
    \max\Big\{\underbrace{\limsup_{n\rightarrow\infty}\frac{\log\pr{\tilde{J}^k_n\in F_1}}{r(\log n)}}_{\text{(I)}}, \underbrace{\limsup_{n\rightarrow\infty}\frac{\log\pr{\tilde{J}^k_n\in F_2}}{r(\log n)}}_{\text{(II)}}\Big\}
\end{equation}
For the term (I), we have
\begin{align}
     \limsup_{n\rightarrow\infty}
     \frac
        {\log\pr{\tilde{J}^k_n\in F_1}}
        {r(\log n)}
    =
    &
    \limsup_{n\rightarrow\infty}
    \frac
        {
            \log\pr{
            \frac{1}{n}
            \sum_{i=1}^k \tilde{Q}^\leftarrow(V_{(i)}) \I_{[U_i,1]} + \frac{1}{n} Z\I_{\{1\}}\in F_1}
        }
        {r(\log n)}
    \nonumber
    \\
    \leq
    &
    \limsup_{n\rightarrow\infty}
    \frac
        {\log\pr{
            \tilde{Q}^\leftarrow(V_{(i^*)}) \geq nr, Z\in\mb{R}_+, U_i\in (0,1)\text{ for } i 
            =
            1,\cdots,k}
        }
        {r(\log n)}
    \nonumber
    \\
    =
    &
    \lim_{n\rightarrow\infty}
    \frac
        {\log\pr{\tilde{Q}^\leftarrow(V_{(i^*)}) \geq nr}}
        {r(\log n)} 
    =
    \lim_{n\rightarrow\infty}
    \frac
        {\log\pr{V_{(i^*)} \geq \tilde{Q}(nr)}}
        {r(\log n)}
    =
    -i^*.
    \label{midstep15}
\end{align}
For the term (II), we have
\begin{align}
     \limsup_{n\rightarrow\infty}
     \frac
        {\log\pr{\tilde{J}^k_n\in F_2}}
        {r(\log n)}
    =
    &
    \limsup_{n\rightarrow\infty}
    \frac
        {\log\pr{
            \frac{1}{n}
            \sum_{i=1}^k \tilde{Q}^\leftarrow(V_{(i)}) \I_{[U_i,1]} + \frac{1}{n} Z\I_{\{1\}}\in F_2}
        }
        {r(\log n)}
    \nonumber
    \\
    \leq
    &
    \limsup_{n\rightarrow\infty}
    \frac
        {\log\pr{
            \tilde{Q}^\leftarrow(V_{(i^*-1)}) 
            \geq
            nr, 
            Z \geq nr,
            U_i\in (0,1)
            \text{ for } i = 1,\cdots,k
            }
        }
        {r(\log n)}
    \nonumber
    \\
    =
    &
    \lim_{n\rightarrow\infty}
    \frac
        {\log\pr{V_{(i^*-1)}\geq \tilde{Q}(nr)}}
        {r(\log n)}
    +
    \lim_{n\rightarrow\infty}
    \frac
        {\log\tilde{Q}(nr)}
        {r(\log n)}
    =
    -(i^*-1) -1
    =
    - i^*.
    \label{midstep16}
\end{align}
Note that in line (\ref{midstep15}) and (\ref{midstep16}), we use the limit result (\ref{limit6}) and (\ref{limit9}) with the fact $-\log\tilde{Q}(n)= r(\log n) (1+o(1))$. 

Combining the calculation for (I) and (II), (\ref{midstep7}) becomes
\begin{equation*}
    \limsup_{n\rightarrow\infty}\frac{\log\pr{\tilde{J}^k_n\in F_2}}{r(\log n)}\leq - i^* 
\end{equation*}
With (\ref{rw-tilde-jnk mid-equation}), the upper bound (\ref{equ:tilde-jkn-ldp-upperbound})is established. This finishes the proof.
\end{proof}

\begin{proof}[Proof of Lemma \ref{lemma: extended-ldp-tildesn}]
\linksinpf{lemma: extended-ldp-tildesn}
Since $\bar{S}_n$ and $\tilde{S}_n$ has the same distribution, we will infer the extended LDP of $\{\tilde{S}_n\}_{n\geq 1}$. For this purpose, we apply the approximation lemma \ref{lemma:approximate}. 
As the extended LDP of $\{\tilde{J}^k_n\}_{n\geq 1}$ is confirm by Lemma \ref{lemma:ex-ldp-tilde-jnk}, it remains to verify Lemma \ref{lemma:approximate}'s conditions (\ref{require1}), (\ref{require2}) and (\ref{require3}), with $\tilde{I}_k$, $\tilde{I}$, $\tilde{J}_n^k$, $\tilde{S}_n$, $r(\log n)$ being  $I_k$, $I$, $Y_n^k$,$X_n$, $a_n$ according to Lemma \ref{lemma:approximate}'s notations.

Since $\tilde{I}_k(\cdot) \geq \tilde{I}(\cdot)$ on $\mb{D}$, the condition (\ref{require1}) is clearly met.
In Theorem \ref{theorem:ex-ldp-xnbar}, we have shown the condition (\ref{require2}) is satisfied w.r.t $\hat{I}_k$ and $I^{J_1}$.
Here, we use the same argument the condition (\ref{require2}) is satisfied w.r.t $\tilde{I}_k$ and $\tilde{I}$ holds. 
To see the condition (\ref{require3}) is satisfied, note that
\begin{align*}
    &
    \pr{d_{J_1}(\tilde{S}_n, \tilde{J}^k_n) > \epsilon} 
    \leq 
    \pr{\|\tilde{H}_n^k\| > \epsilon}
    \\
    &
    =
    \pr{\|\frac{1}{n}\sum_{i=k+1}^{n-1} \tilde{Q}^\leftarrow(V_{(i)})\I_{[U_i,1]} -\frac{1}{n}\E{Z}\sum_{i=1}^{n-1}I_{[U_i,1]} -\frac{1}{n} \E{Z}\I_{\{1\}}\|>\epsilon}
    \\
    &
    \leq 
    \pr{\|\frac{1}{n}\sum_{i=k+1}^{n-1} \tilde{Q}^\leftarrow(V_{(i)})\I_{[U_i,1]} -\frac{1}{n}\E{Z}\sum_{i=1}^{n-1}I_{[U_i,1]}\| > \frac{\epsilon}{2}}
    \\
    &
    +\pr{\|\frac{1}{n} \E{Z}\I_{\{1\}}\|>\frac{\epsilon}{2}}.
\end{align*}
It's clear the term $\pr{\|\frac{1}{n} \E{Z}\I_{\{1\}}\|_\infty>\frac{\epsilon}{2}}$ becomes zero for large $n$. For such large $n$ and any fixed $0<\delta< \frac{\epsilon}{2k}$, we have  
\begin{align}
    &
    \pr{d_{J_1}(\tilde{S}_n, \tilde{J}^k_n) > \epsilon}
    \nonumber
    \\
    &\leq
    \pr{\|\frac{1}{n}\sum_{i=k+1}^{n-1} \tilde{Q}^\leftarrow(V_{(i)})\I_{[U_i,1]} -\frac{1}{n}\E{Z}\sum_{i=1}^{n-1}I_{[U_i,1]}\|
    >
    \frac{\epsilon}{2}, \tilde{Q}^\leftarrow(V_{(k)})\geq n\delta}
    \nonumber
    \\
    &
    +
    \pr{\|\frac{1}{n}\sum_{i=k+1}^{n-1} \tilde{Q}^\leftarrow(V_{(i)})\I_{[U_i,1]} -\frac{1}{n}\E{Z}\sum_{i=1}^{n-1}I_{[U_i,1]}\|
    >
    \frac{\epsilon}{2}, \tilde{Q}^\leftarrow(V_{(k)})\leq n\delta}
    \nonumber
    \\
    &
    \leq 
    \pr{\tilde{Q}^\leftarrow(V_{(k)})\geq n\delta}
    \nonumber
    \\
    &
    +
    \pr{\|\frac{1}{n}\sum_{i=k+1}^{n-1} \tilde{Q}^\leftarrow(V_{(i)})\I_{[U_i,1]} -\frac{1}{n}\E{Z}\sum_{i=1}^{n-1}I_{[U_i,1]}\|
    >
    \frac{\epsilon}{2}, \tilde{Q}^\leftarrow(V_{(k)})\leq n\delta}.
    \label{midstep9}
\end{align}
Let $Z'_1,\cdots, Z'_{n-1}$ be $n-1$ independent copies of the generic random variable $Z$ and $P'_{n-1}(\cdot)$ be a random permutation on $\{1,\cdots, n-1\}$ such that $Z'_{P'_{n-1}(i)}$ is the i-th largest item among $Z'_1,\cdots, Z'_{n-1}$. Since $\tilde{Q}^\leftarrow (V_{(1)}), \cdots, \tilde{Q}^\leftarrow (V_{(n-1)})$ has the same distribution of the ascending order of $Z'_1,\cdots, Z'_{n-1}$, the second term in (\ref{midstep9}) can be further upper bounded by:
\begin{align*}
    &
    \pr{
        \max_{j = 1,\cdots, n-1}
        \big|\sum_{i=1}^j \big(Z'_i\I\{P'_{n-1}(i) > k\} - \E{Z}\big)\big|
        >
        \frac{n\epsilon}{2},
        Z'_{P'_{n-1}(k)}
        \leq 
        n\delta
    }
    \\
    &
    \leq
    \pr{
        \max_{j = 1,\cdots, n-1}\sum_{i=1}^j \big(Z'_i\I\{P'_{n-1}(i) > k\} - \E{Z}\big)
        >
        \frac{n\epsilon}{2},
        Z'_{P'_{n-1}(k)}
        \leq
        n\delta
    } +
    \\
    &
    \pr{
        \max_{j = 1,\cdots, n-1}\sum_{i=1}^j
        \big(
            \E{Z}-Z'_i
            \I\{P'_{n-1}(i)> k\}
        \big)
        >
        \frac{n\epsilon}{2},
        Z'_{P'_{n-1}(k)}
        \leq
        n\delta
    }
    \\
    &\leq
    \pr{
        \max_{j=1,\cdots, n-1} \sum_{i=1}^j \big(
            Z'_i\I\{Z'_i \leq n\delta\} - \E{Z}
        \big) 
        > 
        \frac{n\epsilon}{2}
    } +
    \\
    &
    \pr{
        \max_{j=1,\cdots, n-1} \sum_{i=1}^j \big(
            \E{Z} - Z'_i\I\{Z'_i \leq n\delta\}
        \big)
        + n\delta k
        > 
        \frac{n\epsilon}{2}
    }.
\end{align*}
With the above upper bound for the second term in (\ref{midstep9}), we have
\begin{align}
    &
    \limsup_{n\rightarrow\infty}
    \frac
        {\log\pr{
            d_{J_1}(\tilde{S}_n, \tilde{J}^k_n) 
            >
            \epsilon}
            }
        {r(\log n)}
    \nonumber
    \\
    \leq
    &
    \max
    \Big\{
        \limsup_{n\rightarrow\infty}
        \frac
            {\log\pr{V_{(k)}
             \leq
             \tilde{Q}(n\delta)}}
            {r(\log n)}
            ,
            \nonumber
            \\
    &
    \limsup_{n\rightarrow\infty}
    \frac
        {\log\pr{
            \underset
                {j=1,\cdots, n-1}{\max}
            \sum_{i=1}^j
            \big(
                Z'_i\I\{Z'_i \leq n\delta\} - E{Z}
            \big)
            >
            \frac{n\epsilon}{2}
            }
        }
        {r(\log n)}
        ,\nonumber
        \\
    &
    \limsup_{n\rightarrow\infty}
    \frac
        {\log\pr{
            \underset
                {j=1,\cdots, n-1}
                {\max}
            \sum_{i=1}^j 
            \big(
                E{Z} - Z'_i\I\{Z'_i \leq n\delta\}
            \big)
            >
            n(\frac{\epsilon}{2}-k\delta)
            }
        }
        {r(\log n)}
    \Big\}
    \label{midstep17}
    \\
    &
    =
    \max\{-k, -\frac{\epsilon}{4\delta}, -\infty\}.
    \label{midstep18}
\end{align}
To achieve (\ref{midstep18}) in the above, we have used the limit result (\ref{limit6}), Lemma \ref{lemma:max-sum-zi-greaterndelta-upperbound} and Lemma \ref{lemma:max-sum-zi-smallerndelta-upperbound} respectively to bound the three limsup terms in the (\ref{midstep17}). We can choose arbitrary small $\delta$ to make the value in  \ref{midstep18} to $-k$, and letting $k\rightarrow \infty$ confirms the condition (\ref{require3}). This finishes the proof.
\end{proof}

\subsection{Proofs for Section \ref{section:extended-ldp-j1-multidimension}}\label{section: extended-ldp-product-proof}

We start with a lemma that will be useful. 
\begin{lemma}
\label{lemma:eLDP-upper-bound-for-sets-with-special-cylindrical-form}
\linksinthm{lemma:eLDP-upper-bound-for-sets-with-special-cylindrical-form}
    If the extended large deviation upper bounds hold for $X_n$ and $Y_n$ with rate functions $I$ and $J$, respectively, and suppose that 
    \begin{equation}
    F \teq \bigcap_{m=1}^M (C_m \times \ycal) \cup (\xcal \times D_m)
    \end{equation}
    where $M$ is a finite integer, and $C_n \subseteq \xcal$ and $D_n \subseteq \ycal$ are not necessarily closed sets. 
    Suppose that $X_n$ and $Y_n$ satisfy the extended LDP with the rate functions $I$ and $J$.
    Then the extended LDP upper bound w.r.t.\ $F$ holds for $(X_n, Y_n)$ with the rate function $K(x,y) = I(x) + J(y)$, i.e., 
    $$
    \limsup_{n\to\infty} a_n^{-1} \P\big((X_n, Y_n) \in F\big)
    \leq
    -\lim_{\epsilon\to 0} 
    \inf_{(x,y) \in F^\epsilon}
    K(x,y)
    $$

\end{lemma}

\begin{proof}[Proof of Lemma~\ref{lemma:eLDP-upper-bound-for-sets-with-special-cylindrical-form}]
\linksinpf{lemma:eLDP-upper-bound-for-sets-with-special-cylindrical-form}
Note that
\begin{equation}
    \label{eq:two-representations-of-cylindar-sets}
    F = \bigcup_{\ell\in\{0,1\}^M} \Big(\bigcap_{m: \ell_m =0} C_m \times \bigcap_{m: \ell_m =1} D_m\Big).
\end{equation}
Therefore, 
\begin{align*}
    &
    \limsup_{n\to\infty} a_n^{-1} \log 
        \P\big((X_n, Y_n) \in  F \big)
    \\
    &
    =
    \limsup_{n\to\infty} a_n^{-1} \log 
        \sum_{\ell\in \{0,1\}^M} 
            \P\big(X_n \in \bigcap_{m:\ell_m = 0} C_m\big)
            \cdot
            \P\big(Y_n \in \bigcap_{m:\ell_m = 1} D_m\big)
    \\
    &
    \leq
    \sup_{\ell \in \{0,1\}^N}
    \left\{
        \limsup_{n\to\infty} a_n^{-1} \log 
            \P\big(X_n \in \bigcap_{m:\ell_m = 0} C_m\big)
            \cdot
            \P\big(Y_n \in \bigcap_{m:\ell_m = 0} D_m\big)
    \right\}
    \\
    &
    =
    \sup_{\ell \in \{0,1\}^N}
    \left\{
        -\lim_{\epsilon\to0} \inf_{x\in \big(\bigcap_{m: \ell_m = 0} C_m\big)^\epsilon} I(x)
        -\lim_{\epsilon\to0} \inf_{x\in \big(\bigcap_{m: \ell_m = 1} D_m\big)^\epsilon} J(x)
    \right\}
    \\
    &
    =
    -\lim_{\epsilon\to0} 
    \inf_{\ell \in \{0,1\}^N}
        \inf_{(x,y)\in \big(\bigcap_{m: \ell_m = 0} C_m\big)^\epsilon \times  \big(\bigcap_{m: \ell_m = 1} D_m\big)^\epsilon} 
            I(x)+ J(x)
    \\
    &
    =
    -\lim_{\epsilon\to0} 
    \inf_{(x,y) \in F^\epsilon}
        I(x) + J(y)
\end{align*}
\end{proof}

Now we are ready to prove Proposition~\ref{theorem:product-extended-ldp}. 
\begin{proof}[Proof of Proposition~\ref{theorem:product-extended-ldp}.]
\linksinpf{theorem:product-extended-ldp}
By induction, it is enough to show this for $d=2$. 
To ease the notation, we will denote $X_n^{(1)}$ with $X_n$, $X_n^{(2)}$ with $Y_n$, 
$\mathcal X^{(1)}$ with $\mathcal X$, $\mathcal X^{(2)}$ with $\mathcal Y$, $I^{(1)}$ with $I$, $I^{(2)}$ with $J$, and $\bar I$ with $K$.

For the lower semicontinuity of $K$, note that the sublevel set of $K$ can be written as follows:
\begin{align*}
\Phi_K(\alpha) 
&
\teq \{(x,y)\in \xcal\times\ycal: K(x,y) \leq \alpha\} 
\\
&
=
\bigcap_{n=1}^\infty \bigcup_{m=0}^{n+1} 
\underbrace{
\Big\{x\in \xcal: I(x) \leq \frac mn \alpha\Big\}
}_{
\Psi_I\big(\frac mn\alpha\big)
}
\times 
\underbrace{
\Big\{y \in \ycal:J(y) \leq \frac{n+1-m}{n} \alpha\Big\}
}_{
\Psi_J\big(\frac{n+1-m}{n}\alpha\big)
},
\end{align*}
which is closed since $I$ and $J$ are lower semicontinuous, and each of the unions involve only finite number of closed sets.

For the lower bound of the extended LDP, given an open set $G\subseteq\xcal\times\ycal$ and its element $(x,y) \in G$, there exist open neighborhoods $G_\xcal \subseteq \xcal$ and $G_\ycal\subseteq \ycal$ of $x$ and $y$ such that $G_\xcal \times G_\ycal \subseteq G$. 
\begin{align*}
\liminf_{n\to\infty} 
a_n^{-1} \log
    \P\big(
    (X_n, Y_n) \in G
    \big)
&
\geq 
\liminf_{n\to\infty} 
a_n^{-1} \log
    \P(
    X_n \in G_\xcal
    )
    \cdot
    \P(
    Y_n \in G_\ycal
    )
\\
&
=
\liminf_{n\to\infty} 
a_n^{-1} \log
    \P(
    X_n \in G_\xcal
    )
+ 
\liminf_{n\to\infty} 
a_n^{-1} \log
    \P(
    Y_n \in G_\ycal
    )
\\
&
\geq
-
\inf_{x'\in G_\xcal} I(x') 
-
\inf_{y'\in G_\ycal} I(y')
\\
&
= 
-K(x,y)
\end{align*}
Taking infimum over $(x,y) \in G$, we arrive at the desired lower bound. 

For the upper bound of the extended LDP, we introduce a few extra notations. 
Let 
$\mathcal I \teq \{I(x): x \in \xcal\}$, 
$\mathcal J \teq \{J(x): x \in \ycal\}$,
and
$\mathcal K \teq \{K(x,y): x \in \xcal, y \in \ycal\}$
denote the ranges of $I$, $J$, and $K$.
Fix an $F\subseteq \xcal \times\ycal$ and set
\begin{equation}
k'
\teq 
\lim_{\epsilon \to 0}\inf_{(x,y)\in F^\epsilon} K(x,y)   
\label{eq:infimum-of-K-over-F-epsilon}
\end{equation}
If $k = 0$, then the extended LDP upper bound is trivially satisfied, and hence, we assume w.l.o.g.\ $k>0$.  
From the assumption (iii), there exists the largest element $k \in \mathcal K$ such that $k < k'$.
Note also that
\begin{equation*}
    \Psi_K(k) 
    = \bigcup_{(i,j) \in B(k)} \Psi_I(i) \times \Psi_J(j)
\end{equation*}
where 
$B(k) = \{(i,j) \in \mathcal I\times \mathcal J: i+ j \leq k\}$.
\elaborate{
\begin{equation*}
\big(\Psi_K(k)\big)^\delta
=
\bigcup_{(i,j) \in B(k)} \Psi_I(i)^\delta \times \Psi_J(j)^\delta
\end{equation*}
}%
Note also that \eqref{eq:infimum-of-K-over-F-epsilon} implies that there exists $\delta > 0$ such that $d(F, \Psi_K(k)) > \delta$.                                    
This, in turn, implies that 
\begin{equation*}
    F \subseteq \big(\xcal \times \ycal\big) \setminus \big(\Psi_K(k)\big)^\delta
    = 
    \bigcap_{(i,j)\in B(k)} \Big(\big(\xcal \setminus \Psi_I(i)^\delta\big) \times \ycal \Big) \cup \Big(\xcal \times \big(\ycal\setminus\Psi_J(j)^\delta\big) \Big).
\end{equation*}
From this and Lemma~\ref{lemma:eLDP-upper-bound-for-sets-with-special-cylindrical-form}, we get
\begin{align*}
\liminf_{n\to\infty} a_n^{-1} \log \P\big((X_n, Y_n) \in 
F)
&
\leq
\liminf_{n\to\infty} a_n^{-1} \log \P\big((X_n, Y_n) \in 
\big(\xcal \times \ycal\big) \setminus \big(\Psi_K(k)\big)^\delta
\big)
\\
&
\leq
-\lim_{\epsilon \to 0} \inf_{(x,y) \in \big((\xcal \times \ycal) \setminus (\Psi_K(k))^\delta\big)^\epsilon} K(x,y) 
< -k.
\end{align*}
Since $-k \leq -k' = -\lim_{\epsilon \to 0}\inf_{(x,y)\in F^\epsilon} K(x,y)$, this implies the extended LDP upper bound:
$$
\liminf_{n\to\infty} a_n^{-1} \log \P\big((X_n, Y_n) \in 
F)
\leq
-\lim_{\epsilon \to 0}\inf_{(x,y)\in F^\epsilon} K(x,y).
$$
\end{proof}

\subsection{Proofs for Section \ref{section:nonexist-standard-ldp}}
\label{section: nonexist-standard-ldp-proof}
\begin{proof}[Proofs of Lemma \ref{lemma:F-disjoint-property}]
\linksinpf{lemma:F-disjoint-property}
Suppose that $\eta \in \hat \D_{\sleq 1}$. It is sufficient to show that $\eta \notin \bar F$. 
We start with the following observations:
\begin{enumerate}[(1)]
\item\label{fact1-Fn}
    If $\xi\in F_n$, then $\xi(\frac{1}{2}) \geq\log n - \frac{1}{2}\nu_1\mu_1 - \frac{1}{3}n^{-\frac{1}{3}}$. 
    In particular, 
    $\xi(\frac12) > 1$ for any $\xi \in F$. 

\item\label{fact2-Fn} 
    If $\xi\in F_n$, then $\xi$ has a jump between $(\frac{3}{4}, 1]$ of size no smaller than $\frac{1}{3}n^{-\frac{1}{3}}$
\end{enumerate}
Note that since \ref{fact1-Fn} implies that $F_n$'s are bounded away from the zero function. 
Therefore, we will focus on the case $\eta \not\equiv 0$ w.l.o.g. 
This allows us to write $\eta = z \I_{[v,1]}$ for $z>0$ and $v\in [0,1]$. 
We conclude the proof by considering the two possible cases---$v\in (1/2, 1]$ and $v \in [0,1/2]$---separately:

Suppose that $v\in (1/2, 1]$.
For any $\xi \in F$, set $z \teq (\xi(1/2), 1/2)\in \Gamma(\xi)$.
From \ref{fact1-Fn}, we can easily check $d\big(z, \Gamma(\eta)\big)\geq v-1/2$. 
From this and Lemma~\ref{lemma:m1prime-gap-criteria}, we have 
$d_{M_1'}(\eta, \xi) \geq v-1/2$.
Since this is true for any $\xi \in F$, we conclude that $\eta \notin \bar F$. 

Suppose that $v \in [0,1/2]$.
We will proceed with proof by contradiction. 
Suppose that $\eta \in \bar F$ so that there exists a sequence $\xi_n \in F$ such that $d_{M_1'}(\xi_n, \eta) \to 0$ as $n\to\infty$. 
We claim that $\xi_n \in F_{m_n}$ for some $m_n$ such that $m_n \to \infty$ as $n\to \infty$. 
To see this, recall \ref{fact2-Fn} and the fact that $\eta$ is constant between 1/2 and 1.
Therefore, Lemma~\ref{lemma:constant-jump-m1prime-distance} implies that $d_{M_1'}(\xi_n, \eta)$ is bounded away from 0 if the claim doesn't hold.
On the other hand,
$$
    d_{M_1'}(\xi_n, \eta)
    \geq
    d\big((\xi_n(1/2), 1/2), \Gamma(\eta)\big) 
    \geq
    |\xi_n(1/2) - z|
    \geq 
    \log m_n - \frac{1}{2}\nu_1\mu_1 - \frac{1}{3}m_n^{-\frac{1}{3}} - z 
    \to 
    \infty,
$$ 
where the second inequality is from Lemma~\ref{lemma:m1prime-gap-criteria} and the third inequality is from the claim. 
This is contradictory to our earlier assumption, and hence, $\eta$ cannot be in $\bar F$.

\end{proof}

\begin{proof}[Proof of Lemma \ref{lemma:F-disjoint-lowerbound}]
\linksinpf{lemma:F-disjoint-lowerbound}
Based on the L\'evy-Ito decomposition (\ref{equ:levy-ito-decomposition}), $\bar X_n$ has the following distribution representation:
\begin{equation}\label{barrn2}
    \bar{X}_n \stackrel{\ms{D}}{=} \bar{J}_n + \bar{H}_n
\end{equation}
where
\linkdest{notation:jnbar-and-hnbar}
\begin{align}
    &
    \notationdef{jnbar}{\bar{J}_n(t)} 
    \delequal
    \frac{1}{n}\int_{[1,\infty)}x\hat{N}([0,nt]\times dx)
    \label{jnbar}
    \\
    &
    \notationdef{hnbar}{\bar{H}_n(t)}
    \delequal
    \frac{aB(nt)}{n} + \frac{1}{n}\int_{(0,1]}x\big(\hat{N}([0,nt]\times dx) - nt\nu(dx)\big) - t\int_{[1,\infty)}x\nu(dx)
    \label{hnbar}
\end{align}
Since $\bar{J}_n$ and $\bar{H}_n$ are independent, we have
\begin{align}
    \limsup_{n\rightarrow\infty} 
    \frac
        {\log\pr{\bar{X}_n\in F}}
        {r(\log n)}
    &\geq
    \limsup_{n\rightarrow\infty} 
    \frac
        {\log\pr{\bar{X}_n\in F_n}}
        {r(\log n)}
    \nonumber
    \geq
    \limsup_{n\rightarrow\infty} 
    \frac
        {\log\pr{\bar{J}_n\in B_n, \bar{H}_n\in C_n}}
        {r(\log n)}
    \nonumber
    \\
    &
    \geq
    \limsup_{n\rightarrow\infty} 
    \frac
        {\log\pr{\bar{J}_n\in B_n} + \log\pr{\bar{H}_n\in C_n}}
        {r(\log n)}
    \nonumber
    \\
    &
    =
    \underbrace
        {\limsup_{n\rightarrow\infty} \frac{\log\pr{\bar{J}_n\in B_n}}{r(\log n)}
        }_{(\text{I})}
    +
    \underbrace
        {\lim_{n\rightarrow\infty} \frac{ \log\pr{\bar{H}_n\in C_n}}{r(\log n)}}_{(\text{II})}\label{midstep=bn-cn}.
\end{align}
Due to Lemma~\ref{lemma:hnbar-in-Cn-pr-lowerbounded}, we have (II)$=0$. 
To evaluate (I), 
recall $Q_n^\leftarrow(\cdot)$'s definition in (\ref{definition:Qn_and_inverse}) and $\hat{N}$'s representation in (\ref{poisson-measure-pointmass}). 
$\bar{J}_n$ has the following representation.
\begin{equation*}
    \newnota{jnbar-counterexample}{\bar{J}_n(t)} \defnota{jnbar-counterexample}{\stackrel{d}{=} \sum_{i=1}^\infty Q_n^\leftarrow(\Gamma_i)\I_{[1,\infty)}(Q_n^\leftarrow(\Gamma_i))\I_{[U_i, 1]}(t)}
\end{equation*}
Here $\Gamma_i= Y_1+\cdots + Y_i$ and $Y_i$'s are i.i.d.\ Exp(1) variables. 
From the definitions of $A_n$ and $B_n$'s in (\ref{setan}) and (\ref{setbn}),
\begin{align*}
    &
    \{Y_1\in \big(Q_n(n\cdot 2\log n), Q_n(n\log n)\big], Y_2 \in\big[0, Q_n(n\cdot n^{-\frac{1}{3}}) - Q_n(n\log n)\big), U_1\in(\frac{1}{4}, \frac{1}{2}], U_2\in(\frac{3}{4},1]\}
    \\
    &
    \subset
    \{ Y_1\in \big(Q_n(n\cdot 2\log n), Q_n(n\log n)\big], Y_1+ Y_2\in \big[0, Q_n(n\cdot n^{-\frac{1}{3}})\big], U_1\in(\frac{1}{4},\frac{1}{2}], U_2\in(\frac{3}{4},1]\}
    \\
    &
    \subset
    \{\frac{Q^\leftarrow_n(Y_1)}{n}\in  [\log n,2\log n)\text{ and } \frac{Q^\leftarrow_n(Y_1 + Y_2)}{n}\in  [n^{-\frac{1}{3}},\infty), U_1\in(\frac{1}{4}, \frac{1}{2}], U_2\in(\frac{3}{4},1]\}
    \\
    &
    \subset
    \{\pi(\bar{J}_n)\in A_n\}= \{\bar{J}_n\in \pi^{-1}(A_n) = B_n\}.
\end{align*}
where the second inclusion is due the fact that $c\in \big(Q_n(b),  Q_n(a)\big]$ if and only if $Q_n^\leftarrow(c) \in [a, b)$  for  any $c$ and $b > a \geq 0$.
Therefore, the term (I) by can be lower-bounded as below:
\begin{align*}
    &
    \limsup_{n\rightarrow\infty} 
    \frac
        {\log\pr{\bar{J}_n\in B_n}}
        {r(\log n)}
    \nonumber
    \\
    &
    \geq
    \limsup_{n\rightarrow\infty}
    \frac
        {\log
        \pr{
            Y_1\in (Q_n(n\cdot 2\log n), Q_n(n\log n)]
        }
        +
        \log
        \pr{
            Y_2 \in[0, Q_n(n\cdot n^{-\frac{1}{3}}) - Q_n(n\log n)) 
        }
        }
        {r(\log n)}
    \nonumber
    \\
    &
    =
    \limsup_{n\rightarrow\infty}
    \frac{1}{r(\log n)}
    \Big[\
        \log\int_{Q_n(n\cdot 2\log n)}^{Q_n(n\log n)} e^{-y_1}dy_1
        +
        \log\int_0^{Q_n(n^{\frac{2}{3}}) - Q_n(n\log n)} e^{-y_2}dy_2 
    \Big]
    \nonumber
    \\
    &
    \geq
    \underbrace{
    -\lim_{n\rightarrow\infty}
    \frac{Q_n(n\log n)}{r(\log n)}
    }_{\text{(III)}}
    +
    \underbrace{
    \lim_{n\rightarrow\infty}
    \frac
        {\log\big(Q_n(n\log n) - Q_n(n\cdot 2\log n)\big)}
        {r(\log n)}
    }_{\text{(IV)}}
    +
    \underbrace{
    \lim_{n\rightarrow\infty}
    \frac
        {\log\big(1 - e^{-Q_n(n^{\frac{2}{3}}) + Q_n(n\log n)}\big)}
        {r(\log n)}.
    }_{\text{(V)}}
\end{align*}
Now we are left with identifying (III), (IV), and (V). 
Note that due to \eqref{limit3} of Lemma~\ref{lemma:limit-result-4},
$$
\text{(III)} 
= \lim_{n\rightarrow\infty} \frac{Q_n(n\log n)}{r(\log (n\log n))} \frac{r(\log (n\log n))}{{r(\log n)} } 
=0.
$$
On the other hand,
\begin{align}
    \text{(IV)}
    &
    =
    \lim_{n\rightarrow\infty}
    \frac
        {
        \log
        \big(
            c n^{\beta+1} (\log n)^\beta e^{-\lambda (\log n+\log\log n)^\gamma}  -c n (n\cdot 2\log n)^\beta e^{-\lambda (\log n + \log\cdot 2\log n)^\gamma}
        \big)
        }
        {r(\log n)}
    \nonumber
    \\
    &
    =\lim_{n\rightarrow\infty}
    \frac
        {
        \log c 
        +
        (\beta+1)\log n
        +
        \beta\log\log n-\lambda (\log n +\log\log n)^\gamma
        }
        {r(\log n)}
    \label{mid-step-3}
    \\
    &
    \quad
    + 
    \lim_{n\rightarrow\infty}
    \frac   
        {
        \log
        \big(
            1 - 2^\beta
            e^{\lambda (\log n 
            +
            \log\log n)^\gamma
            -
            \lambda (\log n 
            +
            \log\log n +\log 2)^\gamma} 
        \big)
        }
        {r(\log n)}
    \label{mid-step-4}
    \\
    &
    = -1,
    \nonumber
\end{align}
since (\ref{mid-step-3}) is obviously $-1$ and (\ref{mid-step-4}) is 0 as $\lambda(\log n + \log\log n)^\gamma-\lambda (\log n + \log\log n +\log 2)^\gamma$ tends to $-\infty$ as $n\rightarrow\infty$.
Finally, 
\begin{align*}
 \text{(V)}
 &
 =
 \lim_{n\rightarrow\infty} 
 \frac
    {\log\big(Q_n(n^{\frac{2}{3}}) - Q_n(n\log n)\big)}
    {r(\log n)}
 \\
 &
 =
 \lim_{n\rightarrow\infty}
 \frac
    {\log
     \big(
        c n ^{\frac{2}{3}\beta+1} e^{-\lambda(\frac{2}{3})^\gamma r(\log n)}
        -
        c n^{\beta+1} (\log n)^\beta 
        e^{-\lambda (\log n+\log\log n)^\gamma} 
     \big)
     }
    {r(\log n)}
 \\
 &
 =
 \lim_{n\rightarrow\infty}
 \frac
    {\log
     \big(
      c n ^{\frac{2}{3}\beta+1}
      e^{-\lambda(\frac{2}{3})^\gamma r(\log n)} \big)
      +
      \log\big(1 - n^{\frac{1}{3}\beta}
      e^{\lambda r(\log n)
      \big(
        (\frac{2}{3})^\gamma - (1+\frac{\log\log n}{\log n})^\gamma\big)}
      \big)
    }
    {r(\log n)}
  \\
  &
  =
  \lim_{n\rightarrow\infty}
  \frac
      {\log c + (\frac{2}{3}\beta+1)\log n  -\lambda(\frac{2}{3})^\gamma r(\log n)}
      {r(\log n)}
  +
  0 
  =
  -\big(\frac{2}{3}\big)^\gamma 
  \geq 
  -1
  .
\end{align*}
This concludes the proof of the lemma.
\end{proof}

\begin{lemma}\label{lemma:hnbar-in-Cn-pr-lowerbounded}
\linksinthm{lemma:hnbar-in-Cn-pr-lowerbounded}
For $C_n$ in (\ref{setcn}) and $\bar{H}_n$ in (\ref{hnbar}),
\begin{equation*}
    \lim_{n\rightarrow\infty}\frac{\log \pr{\bar{H}_n\in C_n}}{r(\log n)} = 0
\end{equation*}
\end{lemma}
\begin{proof}
\linksinpf{lemma:hnbar-in-Cn-pr-lowerbounded}
It is enough to shows that $\pr{\bar{H}_n\notin C_n}\rightarrow 0$ as $n\rightarrow\infty$. 
\begin{align*}
    \pr{\bar{H}_n\notin C_n}
    &
    =
    \pr{
        \sup_{t\in[0,1]}|\bar{H}_n(t) + t\nu_1\mu_1| 
        >
        \frac{1}{3} \frac{1}{n^\frac{1}{3}}
    }
    \\
    & 
    =
    \pr{
        \sup_{t\in[0,1]}
        \bigg|
            \frac{aB(nt)}{n} 
            +
            \frac{1}{n}\int_{(0,1]}x\big(\hat{N}([0,nt]\times dx) - nt\nu(dx)\big)
        \bigg| 
        >
        \frac{1}{3}
        \frac{1}{n^\frac{1}{3}}
    }
    \\
    &
    =
    \pr{
        \sup_{t\in[0,1]}
        \bigg(aB(nt)
        +
        \int_{(0,1]}x\big(\hat{N}([0,nt]\times dx) - nt\nu(dx)\big)\bigg)^2
        >
        \frac{1}{9} n^\frac{4}{3}
    }
    \\
    &
    \leq
    \frac
        {\E{\big(aB(n) + \int_{(0,1]}x\big(\hat{N}([0,n]\times dx) - nt\nu(dx)\big)\big)^2}}
        {\frac{1}{9} n^\frac{4}{3}}
    \\
    &
    =
    \frac
        {n\E{\big(aB(1) + \int_{(0,1]}x\big(\hat{N}([0,1]\times dx) - t\nu(dx)\big)\big)^2}}
        {\frac{1}{9} n^\frac{4}{3}}
    =
    \frac{\text{Const}}{n^\frac{1}{3}} \to 0,
\end{align*}
where the inequality is due to the Doob's submartingale inequality. 
\end{proof}

\subsection{Supporting lemmata}
\subsubsection{Asymptotic limits associated with $Q_n$ and $\tilde Q$}
This section covers several limits associated with $Q_n$ and $\tilde{Q}$. 
Recall $\nu$ satisfies Assumption~\ref{assumption-lognormaltail} and $Q_n(x) = n\nu[x,\infty)$ for $x > 0$. 

\begin{lemma}
\label{lemma:limit-result-4}
\linksinthm{lemma:limit-result-4}
The following asymptotics hold:
\begin{align}
    &
    \lim_{n\rightarrow\infty} Q_n(nx)
    =
    0
    \label{limit1}
    \\
    &
    \lim_{n\rightarrow \infty}
    \frac{Q_n(nx)}{r(\log n)}
    =  0
    \label{limit2}
    \\
    &
    \lim_{n\rightarrow \infty}
    \frac{\log Q_n(nx)}{r(\log n)}
    =
    -1
    \label{limit3}
    \\
    &
    \lim_{n\rightarrow \infty}
    \frac
        {\log\big(Q_n(nx_1)\pm  Q_n(nx_2)\big)}
        {r(\log n)}
    =
    - 1
    \text{ for }
    0 < x_1< x_2
    \label{limit4}
    \\
    &
    \lim_{n\rightarrow\infty}
    \frac
        {\log\P\big(\Gamma_i \leq  Q_n(nc)\big)}
        {r(\log n)}
    = -i
    \label{limit5}
\end{align}
\end{lemma}

\begin{proof}
\linksinpf{lemma:limit-result-4}
The first three (\ref{limit1}), (\ref{limit2}), and (\ref{limit3}) are trivial.
For (\ref{limit4}), 
\begin{align*}
     &\lim_{n\rightarrow \infty} 
     \frac
        {\log\big(Q_n(nx_1)  \pm Q_n(nx_2)\big)}
        {r(\log n)}
    \\
    &
    =  
    \lim_{n\rightarrow \infty} 
    \frac
        {\log\big( n \exp(-r(\log nx_1)) \pm n \exp( -r(\log nx_2)\big)}
        {r(\log n)}        
    \nonumber
    \\
    &
    =
    \lim_{n\rightarrow \infty} 
    \frac
        {\log\Big(
            n 
            \cdot 
            \exp\big(-r(\log nx_1)\big) 
            \cdot
            \big( 1 \pm \exp\big\{ r(\log n x_1) -r(\log nx_2) \big\}\big)
        \Big)}
        {r(\log n)}
    \nonumber
    \\
    &
    =
    \lim_{n\rightarrow \infty} 
    \frac
        {\log\big(\exp\big(-r(\log nx_1)\big)\big)}
        {r(\log n)}
    + \lim_{n\rightarrow \infty} 
    \frac
        {\log\big( 1 \pm \exp\big\{ r(\log n x_1) -r(\log nx_2) \big\}\big)}
        {r(\log n)}
    \nonumber
    \\
    &
    =
    -1 
\end{align*}
Note that the last equality holds because $\lim_{n\rightarrow\infty} \big(r(\log nx_1) - r(\log nx_2) \big)= -\infty$.
For (\ref{limit5}), recall that $\Gamma_i$ follows the Erlang-$i$ distribution, and hence,
\begin{equation}
    \lim_{n\rightarrow\infty} \frac{\log\pr{\Gamma_i \leq  Q_n(nc)}}{r(\log n)} 
    = 
    \lim_{n\rightarrow\infty} \frac{\log\big( \frac1{i!}\int_0^{Q_n(nc)} s^{i-1}e^{-s}ds\big)}{r(\log n)}
    = 
    \lim_{n\rightarrow\infty} \frac{\log\big( \int_0^{Q_n(nc)} s^{i-1}e^{-s}ds\big)}{r(\log n)}.
    \label{eq:first-eq-in-proof-lemma:limit-result-4}
\end{equation}
Note that since $Q_n(nc)\downarrow 0$ as $n\to\infty$,
$$
\frac1{2i} \big(Q_n(nc)\big)^{i}
\leq
\frac12\int_0^{Q_n(nc)} s^{i-1}ds
\leq 
\int_0^{Q_n(nc)} s^{i-1}e^{-s}ds
\leq 
\int_0^{Q_n(nc)} s^{i-1}ds
\leq 
\frac1{i} \big(Q_n(nc)\big)^{i}
$$
for sufficiently large $n$'s. 
Combining this with \eqref{eq:first-eq-in-proof-lemma:limit-result-4}, we arrive at (\ref{limit5}) as follows:
$$
\lim_{n\rightarrow\infty} \frac{\log\Big(\int_0^{Q_n(nc)} s^{i-1}e^{-s}ds\Big)}{r(\log n)} 
= 
\lim_{n\rightarrow\infty} \frac{\log\Big(\big(Q_n(nc)\big)^i\Big)}{r(\log n)}
.
$$
The last equality is due to \eqref{limit3}.
\end{proof}

We assume the tail of the random variable $Z$ satisfies Assumption~\ref{assumption-lognormaltail2} and  $\tilde{Q}(x) =\pr{Z \geq x}\sim \exp(-r(\log x))$. The limit results in terms of $\tilde{Q}$ is listed in the following lemma:

\begin{lemma}
\label{lemma:limit-result-5}
\linksinthm{lemma:limit-result-5}
For $x >0 $ and $i\in\mb{N}$,
\begin{align}
    &
    \lim_{n\rightarrow\infty}
    \frac
        {\log\tilde{Q}(nx)}
        {-\log \tilde{Q}(n)}
    =
    -1
    \label{limit6}
    \\
    &
    \lim_{n\rightarrow\infty}
    \frac
    {\log(\tilde{Q}(nx_1)- \tilde{Q}(nx_2))}
    {-\log \tilde{Q}(n)}
    =
    -1
    \text{ for }
    x_2 > x_1 > 0
    \label{limit7}
    \\
    &
    \lim_{n\rightarrow\infty}
    \frac
        {n\log(1-\tilde{Q}(nx))}
        {-\log \tilde{Q}(n)}
    = 
    0
    \label{limit8}
    \\
    &
    \lim_{n\rightarrow\infty}
    \frac
        {\log\pr{V_{(i+1)}\leq \tilde{Q}(nx)}}
        {-\log \tilde{Q}(n)}
    \leq
    -(i+1)
    \label{limit9}
\end{align}
\end{lemma}

\begin{proof}
\linksinpf{lemma:limit-result-5}
The result (\ref{limit6}) is similar to the result (\ref{limit3}), more specifically:
\begin{align*}
    &
    \lim_{n\rightarrow \infty}
    \frac
        {\log \tilde{Q}(nx)}
        {-\log\tilde{Q}(n)}
    =
    \lim_{n\rightarrow \infty}
    \frac
        {\log c +\beta \log n + \beta\log x - \lambda(\log n + \log x)^2}
        {-\log c -\beta \log n + \lambda\log^2 n}
    \nonumber
    \\
    &
    =
    \lim_{n\rightarrow \infty}
    \frac
        {-\lambda(\log n + \log x)^2}
        { \lambda\log^2 n}
    =
    \lim_{n\rightarrow\infty}
    -\Big(1 + \frac{\log x}{\log n}\Big)^2 
    =
    - 1
\end{align*}
The result (\ref{limit7}) is immediate from the result (\ref{limit4}) if we realize the quantity $\tilde{Q}(nx)$ and $Q_n(nx)$ only differ with a factor of $n$, thus after taking the log, its effect is negligible. 

For the result (\ref{limit8}), we have the function $\log(1-x)$ around $x=0$ takes the form $-x+o(x)$, hence the term 
\begin{equation*}
    \lim_{n\rightarrow\infty}\frac{n\log(1-\tilde{Q}(nx))}{-\log \tilde{Q}(n)}  = \lim_{n\rightarrow\infty}\frac{n\tilde{Q}(n\delta)}{-\log\tilde{Q}(n)}  = \lim_{n\rightarrow\infty}\frac{n\cdot c(n\delta)^\beta e^{-\lambda \log^2 (n\delta)}}{\lambda \log^2 n} = 0
\end{equation*}

For the last result (\ref{limit9}), we have $V_{(i+1)}$ following the $Beta(i+1, n-i-1)$ distribution which is of the density $\frac{\Gamma(n-1)}{\Gamma(i+1)\Gamma(n-i-1)}y^i(1-y)^{n-i-2}$, hence

\begin{align*}
    &
    \lim_{n\rightarrow\infty}
    \frac
        {\log\pr{V_{(i+1)}\leq \tilde{Q}(nx)}}
        {-\log \tilde{Q}(n)}
    =
    \limsup_{n\rightarrow\infty}
    \frac
        {\log\int_0^{\tilde{Q}(nr)}
            \frac
                {\Gamma(n-1)}
                {\Gamma(i+1)\Gamma(n-i-1)}
            y^i(1-y)^{n-i-2}dy
        }
        {-\log\tilde{Q}(n)}
    \\
    &
    \leq
    \limsup_{n\rightarrow\infty}
    \frac
        {\log\int_0^{\tilde{Q}(nr)}
            \frac
                {\Gamma(n-1)}
                {\Gamma(i+1)\Gamma(n-i-1)}
            y^idy
        }
        {-\log\tilde{Q}(n)}
    \\
    &
    =
    \limsup_{n\rightarrow\infty}
    \frac
        {\log
        \frac
            {\Gamma(n-1)}
            {\Gamma(i+1)\Gamma(n-i-1)(i+1)}
            \tilde{Q}^{i+1}(nr)
        }
        {-\log\tilde{Q}(n)}
    \\
    &
    =
    \lim_{n\rightarrow\infty}
    \frac
        {\log\frac{\Gamma(n-1)}{\Gamma(i+1)\Gamma(n-i-1)(i+1)}}
        {-\log\tilde{Q}(n)} 
    +
    \lim_{n\rightarrow\infty}
    \frac
        {(i+1)\log \tilde{Q}(nr)}
        {-\log\tilde{Q}(n)}
    \\
    &
    =
    0 - (i+1)
    =
    -(i+1)
\end{align*}

\end{proof}

\subsubsection{Basic concentration inequalities}\label{appendix}
This section reviews two well-known concentration inequalities.

\begin{lemma}\label{etemadi}
    [Etemadi's inequality] Let $X_1,\cdots, X_n$ be independent real-valued random variables defined on some common probability space, and let $x\geq 0$. Let $S_k$ denote the partial sum $S_k = X_1 + \cdots + X_k$. Then
    \begin{equation*}
        \pr{\underset{1\leq k\leq  n}{\max} |S_k|\geq 3x} \leq 3\underset{ 1\leq k\leq n}{\max} \pr{|S_k|\geq x}
    \end{equation*}
\end{lemma}


\begin{lemma}\label{berstein}
    [Bernstein's inequality] Let $X_1,\cdots, X_n$ be independent real-valued random variables with finite variance such that $X_i\leq b$ for some $b > 0$ almost surely for all $i\leq n$. Let $S = \sum_{i=1}^n (X_i - \E{X_i})$ and $v = \sum_{i=1}^n \E{X_i^2}$, then for any $t > 0$
    \begin{equation*}
        \pr{S \geq t} \leq \exp\Big\{-\frac{t^2}{2(v+bt/3)}\Big\}
    \end{equation*}
\end{lemma}

\bibliography{ref}


\end{document}